\documentclass[pdflatex,sn-mathphys-num]{sn-jnl}  

\usepackage{graphicx}
\usepackage{multirow} 
\usepackage{amssymb,amsfonts}
\usepackage{amsthm} 
\usepackage{mathrsfs} 
\usepackage[title]{appendix} 
\usepackage{xcolor} 
\usepackage{textcomp} 
\usepackage{manyfoot} 
\usepackage{booktabs} 
\usepackage{algorithm} 
\usepackage{algorithmicx} 
\usepackage{algpseudocode} 
\usepackage{listings} 
\usepackage{natbib}
\usepackage{subfigure}
\usepackage{framed}
\usepackage{adjustbox}
\usepackage{comment} 
\usepackage[tbtags]{amsmath}
\usepackage{pdflscape}
\usepackage{anyfontsize}
\usepackage{enumerate}
\usepackage{algcompatible}
\DeclareMathOperator{\sech}{sech}
 
\theoremstyle{thmstyleone} 
\newtheorem{theorem}{Theorem} 
\newtheorem{proposition}[theorem]{Proposition} 
\newtheorem{lemma}{Lemma} 

\theoremstyle{thmstyletwo}

\newtheorem{assumption}{Assumption}[section]
\theoremstyle{thmstylethree} 
\newtheorem{definition}{Definition} 
\newtheorem{cor}{Corollary}[section]
\newtheorem{prob}{Problem}[section]
\usepackage{array}

\makeatletter
\newcommand{\thickhline}{ 
    \noalign {\ifnum 0=`}\fi \hrule height 1.5pt
    \futurelet \reserved@a \@xhline
}
\newcolumntype{"}{@{\hskip\tabcolsep\vrule width 1pt\hskip\tabcolsep}}
\makeatother

\begin{document}

\title[Nonmonotone Trust-Region Methods for Set Optimization]{Nonmonotone Trust-Region Methods for Optimization of Set-Valued Mapping of Finite Cardinality}

\author[1]{\fnm{Suprova} \sur{Ghosh}}\email{suprovaghosh.rs.mat19@itbhu.ac.in}

\author*[1]{\fnm{Debdas} \sur{Ghosh}}\email{debdas.mat@iitbhu.ac.in}

\author[2,3,4]{\fnm{Zai-Yun} \sur{Peng}}\email{pengzaiyun@126.com}

\author[5]{\fnm{Xian-Jun} \sur{Long}}\email{xianjunlong@ctbu.edu.cn}

\affil[1]{\orgdiv{Department of Mathematical Sciences}, \orgname{Indian Institute of Technology (BHU)}, \orgaddress{\city{Varanasi}, \postcode{221005}, \state{Uttar Pradesh}, \country{India}}}

\affil[2]{School of Mathematics, Yunnan Normal University, Kunming, 650092, People’s Republic of China}

\affil[3]{Yunnan Key Laboratory of Modern Analytical Mathematics and Applications, Yunnan Normal University, Kunming, 650092, People’s Republic of China}

\affil[4]{\orgdiv{School of Mathematics and Statistics}, \orgname{Chongqing Jiaotong University}, \orgaddress{\city{Chongqing}, \postcode{400074}, \country{China}}}

\affil[5]{\orgdiv{School of Mathematics and Statistics}, \orgname{Chongqing Technology and Business University}, \orgaddress{\city{Chongqing}, \postcode{400067}, \country{China}}}

\abstract{
\emph{Non-monotone trust-region methods} are known to provide additional benefits for scalar and multi-objective optimization, such as enhancing the probability of convergence and improving the speed of convergence. For optimization of set-valued maps, non-monotone trust-region methods have not yet been explored and investigated to see if they show similar benefits. Thus, in this article, we propose two non-monotone trust-region schemes--max-type and average-type for set-valued optimization. Using these methods, the aim is to find \emph{K}-critical points for a non-convex unconstrained set optimization problem through vectorization and oriented-distance scalarization. The main modification in the existing trust region method for set optimization occurs in reduction ratios, where max-type uses the maximum over function values from the last few iterations, and avg-type uses an exponentially weighted moving average of successive previous function values till the current iteration. Under appropriate assumptions, we show the global convergence of the proposed methods. To verify their effectiveness, we numerically compare their performance with the existing trust region method, steepest descent method, and conjugate gradient method using performance profile in terms of three metrics: number of non-convergence, number of iterations, and computation time.}

\keywords{Set optimization, Trust-region method, Non-monotone strategy, Reduction ratio} 

\pacs[2020 MSC Classification]{49J53, 49M15, 65K15, 90C30,90C33,90C53}

\maketitle

\section{Introduction}\label{sec1}
Set optimization is a generalization of optimization problems where the objective functions are set-valued maps. They are of great interest for modeling various decision-making processes, especially those that involve ambiguity, uncertainty, and multi-valued outcomes. In recent times, set optimization has gained significant attention due to its applicability in game theory \cite{chinchuluun}, duality theory \cite{Hamel}, mathematical economics \cite{Klein}, engineering, portfolio management in finance  \cite{mastrogiacomo2021set}, etc. For example, in \cite{mastrogiacomo2021set}, a set-value-based approach of optimizing a portfolio is introduced, where instead of scalar risk, a set-valued risk measure is used to reduce the risk of high losses. Formally, a \emph{set optimization problem} (SOP) is presented as follows:
\begin{equation*} 
 \text{minimize }F(x)~\text{subject to}~x\in U,
\end{equation*}
where the set-valued map $F:U \rightrightarrows V$ assigns elements of vector space $U$ to the subsets of a vector space $V$. In this paper, we consider unconstrained SOPs with an objective function $F:\mathbb R^{n} \rightrightarrows \mathbb R^{m}$ of the form 
\begin{align}\label{particl_set} 
F(x) = \left\{ f^{1}(x), f^{2}(x), \ldots, f^{p}(x) \right\}, ~x \in \mathbb R^{n},
\end{align}
where $f^{1}, f^{2}, \ldots, f^{p}: \mathbb R^{n}\to \mathbb R^{m} $ are twice continuously differentiable vector-valued functions. 

Mainly, there are two ways of solving SOPs through numerical techniques: the set approach and the vector approach.  
Based on these two approaches, many efficient algorithms have been developed over the last decade, such as the derivative-free method \cite{Jahnderivativefree,Jahn2018}, the steepest descent method \cite{steepmethset}, sorting-based method \cite{Jahn,Jahn2006,aGunther2019}, branch-bound method \cite{eichfelder2020algorithmic}, and complete-lattice approach \cite{lohne2025solution}. 
However, for the type of SOP considered in this paper, these methods suffer from some form of limitations. For example, derivative-free methods \cite{Jahnderivativefree,Jahn2018} do not make use of any information about the derivative, gradient, or Hessian. The steepest descent method \cite{steepmethset} does not theoretically guarantee global convergence and could be slow to converge. Sorting-based \cite{Jahn,Jahn2006,aGunther2019,bGunther2019} and branch-bound methods \cite{eichfelder2020algorithmic} are not well suited in the problem under consideration because, while sorting-based methods require the feasible set to be finite, branch-bound methods assume the feasible set lies within a box defined by convex constraints and convex uncertainty sets.  The lattice approach \cite{lohne2025solution} is not suitable for problems involving many variables. The conjugate gradient method \cite{kumar2024nonlinear} is classified as a first-order convergence method. The Newton method \cite{ghosh2024newton} achieves only superlinear or quadratic convergence. The quasi-Newton method \cite{ghosh2025quasi} establishes global superlinear convergence under the assumption of uniform continuity of the Hessian approximation.

Given these limitations, Ghosh et al. \cite{trust2025setopt} opted for
trust-region-based scheme and recently applied it to set optimization. However, their method has a monotone property and requires the function to decrease at every successful step \cite{carrizo2016trust,baraldi2023proximal,sun2023trust,ouyang2024trust}. This requirement is very strict and, as observed in the case of single-objective \cite{mo2005nonmonotone,ahookhosh2012efficient} and multi-objective optimization \cite{ramirez2022nonmonotone,ghalavand2023two}, such restrictions may limit the speed of convergence. This has been especially observed for the case when a monotone trust-region scheme is forced to go down along the bottom of a steep and tight curved valley, for example, for the Rosenbrock function \cite{deng1993nonmonotonic, sun2004nonmonotone}. In such scenarios, non-monotone schemes relax the strict monotonicity requirement and allow for the increment of function values in some iterations, leading to improved speed and the possibility of convergence.

\subsection{Related work}
In conventional optimization, there have been many studies employing non-monotone modifications to the trust-region method. Drawing motivation from Grippo's non-monotone scheme \cite{grippo1986nonmonotone} for line-search-based methods, Deng et al. \cite{deng1993nonmonotonic} were among the first to apply non-monotonicity to trust-region methods for single-objective optimization. Since then many other works have developed and used nonmonotone approaches for trust region methods, which are detailed in \cite{trust2025setopt}, for examples, Toint \cite{toint1997non}, Sun \cite{sun2004nonmonotone}, Ulbrich et al. \cite{ulbrich2003non},  Mo et al. \cite{mo2005nonmonotone}, Ahookhosh et al. \cite{ahookhosh2012efficient}, Chen et al. \cite{chen2013non}, Maciel et al. \cite{maciel2013monotone} etc. \emph{Non-monotone trust region methods} (NTRMs) are primarily of two types: Max-type and Average-type, where Max-type uses maximum over previous function values \cite{sun2004nonmonotone, mo2005nonmonotone, ahookhosh2012efficient,maciel2013monotone} and Avg-NTRM uses the average of successive previous function values \cite{mo2007nonmonotone} to determine the acceptance of a step.

In the case of multi-objective optimization as well, non-monotone methods have been applied successfully. For example, Ramirez et al. \cite{ramirez2022nonmonotone} modified the monotone trust-region scheme in \cite{carrizo2016trust} to apply non-monotonic update rules and investigated the reduction in the number of subproblems to be solved, the number of inefficient iterations, and the computation time. Ding et al. \cite{ding2023nonmonotone} further built upon \cite{ramirez2022nonmonotone} to introduce an adaptive NTRM to solve a particular class of multi-objective non-linear bi-level optimization problem by taking the convex combination of the current iteration's function value and the maximum over the function values from previous iterations. Although this max-based adaptive method was previously developed for bilevel problems, Ghalavand et al. \cite{ghalavand2023two} extended it to any general multi-objective problems and additionally studied the average-type variant of the adaptive trust-region method.

\subsection{Motivation}
From the aforementioned literature, non-monotone schemes are known to be efficient, particularly for high-dimensional and ill-conditioned problems. By relaxing the strict monotonicity requirement, these schemes have been shown to converge faster with a higher probability of finding a solution \cite{toint1997non}. Because of these advantages, we expect similar improvements when applied to SOPs. However, because of the challenges and difficulties in identifying numerical techniques for set optimization, NTRMs for this domain have not yet been developed. This gap in the research encourages us to develop and study non-monotone trust-region methods for SOP.

\subsection{Contributions}
In this study, we build upon the monotone TRM for SOP proposed in Ghosh et al. \cite{trust2025setopt} to introduce Max-NTRM and Avg-NTRM for SOP and study their convergence properties and numerical performances. The main modification is in defining the reduction ratio, where Max-NTRM considers the maximum over function values from the last few iterations, and Avg-NTRM takes a weighted moving average of function values up to the current iteration. For both methods, we provide theoretical global convergence guarantees.

To compare the performance of Max-NTRM and Avg-NTRM against the existing monotone TRM, Steepest-descent, and conjugate-gradient methods, we present numerical results on $22$ test SOPs. We compare the number of convergences, the number of iterations to converge, and the computation time. Further, we present the performance profile \cite{dolan2002benchmarking} of the five methods for these metrics. Based on the numerical results, we observe that, similar to the cases of single-objective and multi-objective optimization, the proposed non-monotone methods offer significant performance improvement over TRM in terms of number of iterations, computation time, and number of convergences. Especially, in terms of probability of convergence, NTRM performs the best.

Although the proposed NTRMs are motivated by the works of Ramirez and Sottosanto \cite{ramirez2022nonmonotone}, Sun \cite{sun2004nonmonotone}, Ghalavand et al. \cite{ghalavand2023two}, and Mo et al. \cite{mo2007nonmonotone}, they are distinct and novel in the following manner. 
\begin{itemize}
\item Ramirez and Sottosanto \cite{ramirez2022nonmonotone} proposed NTRM schemes for multi-objective optimization. The proposed method extends these schemes to SOPs, which are more general than vector or multi-objective problems, by modifying the monotone TRM scheme in \cite{trust2025setopt}. Unlike \cite{ramirez2022nonmonotone}, where a fixed vector problem is solved, the proposed approach handles evolving vector subproblems at each iteration, determined by the $K$-minimal element of the partition set and oriented-distance scalarization before applying TRM. Moreover, the max-type formulation (\ref{Ratios_NTR_Max}) generalizes that of \cite{trust2025setopt}, reducing to \cite{ramirez2022nonmonotone} when $p=1, K=\mathbb{R}^n_{+}$, while accommodating the general multi-function, multicomponent SOP structure. 
\item Sun \cite{sun2004nonmonotone} introduced one of the first single-objective max-type NTRM schemes. The proposed method broadens this to the SOP context, computing reduction ratios componentwise using the general ordering cone $K$ and $K$-minimal elements of the partition set, and adapting it to vector-valued components in set-valued maps by tracking past values across components rather than just a single maximum.

\item In relation to Ghalavand et al. \cite{ghalavand2023two}, who proposed adaptive max-type and average-type NTRMs for general multi-objective problems, the proposed work presents an Avg-NTRM for SOPs (\ref{NTR_Average_Ratios})–(\ref{tbddv_uiet}) as a direct generalization of Ghalavand’s average-based adaptive ratio but adapted to the $K$-minimal element of the partition set. 
Additionally, it combines exponential weighting with oriented-distance scalarization in the step acceptance criterion, which is not present in Ghalavand’s method.

\item Finally, compared to Mo et al. \cite{mo2007nonmonotone}, who studied an average-type NTRM for single-objective optimization using an exponentially weighted average of past function values, the present method generalizes Mo’s weighting scheme to the SOP context by combining it with the set ordering relation $\preceq_{K}$ and the oriented distance $\Delta_{-K}$. It further extends the reduction ratio computation to handle multi-component outputs and retains the step acceptance rule for moving toward a $K$-critical point as in \cite{trust2025setopt}; for $p =1, r = 1$, it reduces exactly to Mo et al.’s method.  
\end{itemize}

\subsection{Organization}
This paper is structured as follows: Section \ref{section1} presents the notations used in the paper and recalls some prerequisite definitions, including the notion of $K$-critical point and descent direction for set-valued maps. In Section \ref{Section3}, we describe our proposed two non-monotone schemes in detail and provide their respective pseudo-codes in Algorithm \ref{alg_max} and Algorithm \ref{alg_avg}, respectively. We also state the well-definedness of each step of these algorithms. Section \ref{global_conv_ana} is dedicated to proving the global convergence property of the proposed methods. Section \ref{Numerical_Experiment} contains results from numerical experiments to show the comparison of TRM, Max-NTRM, Avg-NTRM, SD, and CG with respect to different performance metrics. Finally, Section \ref{sect6} provides concluding remarks of the paper.

\section{Preliminaries and Terminologies}
\label{section1} 
\noindent
 The notations $\mathbb R^{m}_{+}$ and $\mathbb R^{m}_{++}$ denote the non-negative and positive hyper-octant of $\mathbb R^{m}$, respectively. $\mathscr{P}(\mathbb R^{m})$ is 
the class of all nontrivial subsets of $\mathbb R^{m}$.  For an $\mathcal{A} \in \mathscr{P}(\mathbb R^{m})$, its boundary and interior are denoted by bd(${\mathcal A}$) and int(${\mathcal A}$), respectively. Throughout the paper, any vector $v$ in $\mathbb R^{m}_{+}$ is a column vector, and its transpose is presented by $v^\top$, which is a row vector. $\lVert \cdot \rVert$ denotes a norm in $\mathbb R^{n}$ or in $\mathbb R^{m}$ or the standard matrix norm.  For a finite set $\mathcal A$, $\lvert \mathcal{A} \rvert$ denotes its cardinality. We denote $[k] := \{1, 2, \ldots, k\}$ for any $k \in \mathbb N$.

A nonempty  $K \in  \mathscr{P}(\mathbb R^{m})$ is a cone in $\mathbb{R}^m$ if for any $y \in K$, we have $ty\in K $ for all $t \ge 0$. A cone $K$ is pointed if $ K \cap (-K) = \{0\}$, convex  if $K+K = K$, and solid if $\text{int}(K) \neq \emptyset$. A convex and pointed cone $K$ in $\mathbb{R}^m$ defines a partial ordering on $\mathbb R^{m}$ as $y \preceq_{K}z \Longleftrightarrow z -y \in K$.
 A solid cone $K$ defines a strict ordering on $\mathbb R^{m}$ as $y \prec_{K} z \Longleftrightarrow z -y \in \text{int}(K)$.
Across the entire paper, we use $K$ to denote a convex, pointed, closed, and solid cone in $\mathbb{R}^m$. 
\begin{definition} 
\cite{khran}
Let $ {\mathcal A} \in \mathscr P(\mathbb R^{m})$ and $y_{0}\in {\mathcal A}$.  
\begin{enumerate}[(i)]
    \item  The element $y_{0}$ is called a $K$-minimal element of ${\mathcal A}$ with respect to $K$ if $(y_{0}-K)\cap {{\mathcal{A}}}= \{ y_{0}\}$. The collection of all ${K}$-minimal elements of $\mathcal A$ is denoted by $\text{Min}(\mathcal{A}, K )$.
    \item The element $y_{0}$ is called a weakly ${K}$-minimal element of ${\mathcal A}$ with respect to $K$ if $(y_{0}-\text{int}{(K)})\cap {{\mathcal{A}}}= \emptyset$. The collection of all weakly ${K}$-minimal elements of $\mathcal{A}$ is denoted by $\text{WMin}(\mathcal{A}, K)$.
\end{enumerate}
\end{definition}

To compare a pair of sets, we use the lower set less relation $\preceq^{l}_K$ on $\mathscr P(\mathbb R^{m})$ with respect to a given $K$. For any ${\mathcal A}, {\mathcal B} \in \mathscr P(\mathbb R^{m})$, the relation $\preceq^{l}_K$ on $P(\mathbb R^{m})$ is defined by ${\mathcal A} \preceq^{l}_{K} {\mathcal B} \Longleftrightarrow {\mathcal B}\subseteq {\mathcal A} +  K,$ and the strict lower set less relation $\prec^{l}_K$ on $\mathscr P(\mathbb R^{m})$ is given by ${\mathcal A} \prec^{l}_{K} {\mathcal B} \Longleftrightarrow
    {\mathcal B}\subseteq {\mathcal A} +  \text{int}(K)$.\\

\noindent
In this study, we consider solving the 
 unconstrained set optimization problem 
\begin{equation}\label{fgcx}
(\preceq^{l}_{K}) ~\text{---}\min_{x\in \mathbb R^{n}}  F(x),
\tag{\text{$\mathcal{SOP}^{l}_K$}}
\end{equation}
where $F :=\{f^i\}_{i \in [p]}$ and 
the $j$-th component of $f^i: \mathbb R^{n}\to \mathbb R^{m}$ is denoted by $f^{i, j}$, i.e., 
\[f^{i}(x) := \left(f^{i, 1}(x), f^{i, 2}(x), \ldots, f^{i, m}(x)\right)^\top, ~ x \in \mathbb{R}^n,\] and
the solution concept with respect to $\preceq^{l}_{K}$ is as follows.  
\begin{definition} A point $\bar x \in \mathbb{R}^{n}$ is called a (resp., weakly) $\preceq^{l}_{K}$-minimal solution of  \eqref{fgcx} if there does not exist any $x \in \mathbb{R}^{n}$ such that (resp., $F(x) \prec^{l}_{K} F(\bar x)$) $F(x) \preceq^{l}_{K} F(\bar x)$.
    A point $\bar x\in \mathbb R^{n}$ is called a \emph{local} weakly $\preceq^{l}_{K}$-minimal point of \eqref{fgcx} if there exists $\delta > 0$ such that for none of $x$ with $\lVert x -\bar{x} \rVert < \delta$, we have $F(x) \prec^{l}_K F(\bar x)$.
\end{definition}

Next, we recall a few auxiliary set-valued maps, which will help to find weakly $K$-minimal points of (\text{$\mathcal{SOP}^{l}_K$}) by solving a family of vector optimization problems. 
 \begin{definition}\cite{steepmethset} \label{active_index} 
\begin{enumerate}[(i)]
\item The function $I: \mathbb R^{n} \rightrightarrows [p] $,  defined by 
\begin{align*}
  I(x) := \{ i \in [p]: f^{i}( x) \in \text{Min}({F}(x),  K)\}, ~x \in \mathbb R^{n},  \end{align*}
 is called the set of active indices of $K$-minimal elements of the set $ F( x)$; 
\item For a given  $ v \in \mathbb R^{m}$, the function $I_{v}: \mathbb R^{n} \rightrightarrows [p]$ is defined by
\begin{align*}
    I_{v}(x) := \{ i \in I(x):  f^{i}(x)= v\}.
\end{align*}
is called the set of active indices of weakly $K$-minimal elements of the set $ F( x)$
\item The map $ \omega : \mathbb R^{n} \to \mathbb R$ is defined by 
$\omega (x) := \lvert \text{Min}(F(x),  K)\rvert, ~x \in \mathbb R^{n}.$
\end{enumerate}
\end{definition}

\begin{definition}\cite{steepmethset} For a given $\bar x \in \mathbb R^{n}$, consider an enumeration $\{v^{\bar x}_{1}, v^{\bar x}_{2}, \ldots, v^{\bar x}_{w(\bar x)} \}$ of $ \text{WMin}(F(\bar x),  K) $.
  The \emph{partition  set} at $\bar x$ is defined by 
  \begin{align*}
      {P}_{\bar x} := I_{v^{\bar x}_{1}}(\bar x) \times I_{v^{\bar x}_{2}}(\bar x) \times \cdots \times I_{v^{\bar x}_{w(\bar x)}}(\bar x).
  \end{align*}
\end{definition}

Throughout the paper, for a given $\bar x \in \mathbb{R}^n$, we denote $\bar \omega := \omega(\bar x)$, and a generic element of the partition set $P_{\bar x}$ is denoted by $\bar a := (\bar a_1, \bar a_2, \ldots, \bar a_{\bar \omega})$. 
\begin{definition}\cite{steepmethset} A point $\bar x\in \mathbb R^{n}$ is said to be a \emph{regular point} of the set-valued map $F: \mathbb{R}^n \rightrightarrows \mathbb{R}^n$ if 
\begin{enumerate}[(i)]
    \item $\text{Min}({F}(\bar x), K) = \text{WMin} 
 ({F}(\bar x), K)$, and 
    \item the cardinality function $\omega$ as given in Definition \ref{active_index} is constant  in a neighbourhood of $\bar{x}$.
\end{enumerate}
\end{definition}
\begin{lemma}\textup{\cite{steepmethset}}\label{regularity_condition} If $\bar x$ is a regular point of $F$, then there exists a neighborhood $\mathcal{N}$ of $\bar x$ such that 
\[\omega(x) = \omega(\bar x) \text{ and } P_x \subseteq P_{\bar x} \text{ for all } x \in \mathcal{N}. \]
\end{lemma}

\begin{definition} \textup{\cite{trust2025setopt}} \label{spcritic} 
A point $ \bar x \in \mathbb R^{n}$ is said to be a $K$-\emph{critical point} for \eqref{fgcx} if there does not exist any $\bar a \in P_{\bar x}$ and $\bar s \in \mathbb{R}^n$ such that 
$$ \nabla f^{\bar a_j}(\bar x)^{\top} \bar s \prec_K 0\in \mathbb R^{m} ~\text{for all}~ j \in [\omega(\bar x)], $$ 
where for any $i \in [p]$, the notation $\nabla f^{i}(\bar x)^{\top} s$ is given by $$\nabla f^{i}(\bar x)^{\top} s := 
\left(   
 \nabla f^{i,1}(\bar x)^{\top}s, 
 \nabla f^{i,2}(\bar x)^{\top}s,  
 \ldots, 
 \nabla f^{i,m}(\bar x)^{\top}s
\right)^\top, s \in \mathbb{R}^n. $$  
\end{definition}

\begin{definition}\textup{\cite{trust2025setopt}}\label{descent_setopt}
   An $s \in \mathbb R^{n}$ is said to be  a \emph{descent direction} of the set-valued map $F = \{ f^{i}\}_{i \in [p]}$ at $\bar x$ if there exists $t_{0} > 0$ such that 
   \begin{align*}
       \{f^{i}(\bar x + t s) \}_{i \in [p]} \prec^{l}_{K} \{ f^{i}(\bar x)\}_{i \in [p]} ~\text{for all}~ t \in (0, t_{0}].
   \end{align*}
\end{definition}

Next, consider the following \emph{vector optimization problem} associated with a partition element $a := (a_1, a_2, \dots, a_{\bar{\omega}}) \in P_{\bar{x}}$, where $\bar{x} \in \mathbb{R}^m$ and $\bar{\omega} := \omega(\bar{x})$. Define the cone
\[
\widetilde{K} := \prod_{j=1}^{\bar{\omega}} K \subseteq \mathbb{R}^{m\bar{\omega}},
\]
and let $\preceq_{\widetilde{ K}}$ and $\prec_{\widetilde{ K}}$ denote the partial and strict orders, respectively, on $\mathbb{R}^{m\bar{\omega}}$ induced by the cone $\widetilde{ K}$. Define the  function $\widetilde{f}^{a}: \mathbb{R}^n \to \prod_{j=1}^{\bar{\omega}} \mathbb{R}^m$ by
\begin{equation}
\widetilde{f}^{a}(x) := 
\begin{pmatrix}
f^{a_1}(x) \\
f^{a_2}(x) \\
\vdots \\
f^{a_{\bar{\omega}}}(x)
\end{pmatrix}.
\label{eq:fa}
\end{equation}
Then, the associated \emph{vector optimization problem} is given by
\begin{equation*}\label{vec_pr}
(\preceq_{\widetilde{ K}})\text{---}\min_{x \in \mathbb R^{n}}  {\widetilde{f}}^{a}(x),
\tag{\text{$\mathcal{VOP}_{a}$}}
  \end{equation*}

Next, we recall the properties of the oriented distance function $\Delta_{-K}$ in characterizing $K$-critical points and descent directions of the vector-valued map $\widetilde{f}^{a}$ in (\ref{vec_pr}). We choose to use the oriented-distance-based scalarization because the class of scalarization functions given by $\Delta_
{-K}$  is known to be more general (See Corollary 2, p.700 of \cite{bouza2019unified}).

\begin{definition}\cite{Ansari2018} Let $\mathcal{A}\in \mathscr{P}(\mathbb R^{m})$. A function $\Delta_{\mathcal{A}}: \mathbb R^{m} \to \mathbb R  \cup \{\pm \infty\}=  :\overline{\mathbb R}$ defined by  
\begin{align*}
    \Delta_{\mathcal A}(y): = d_{\mathcal A}(y) -d_{\mathcal A^{c}}(y), y \in \mathbb R^{m},
\end{align*}
is called the oriented distance function, where $d_{\mathcal A}(y):=\inf\limits_{a \in \mathcal{A}}\lVert y-a \rVert$. By convention, for any $y \in \mathbb{R}^m$, $d_{\emptyset}(y) = + \infty$, and hence, $\Delta_{\emptyset}(y)= + \infty$, whereas $\Delta_{\mathbb R^{m}}(y)= -\infty$.
\end{definition}

\begin{lemma} \textup{\cite{Ansari2018, zaffaroni}}\label{proori}
\begin{enumerate}[(i)]
     \item\label{aeruin} $\Delta_{-K}$ is Lipschitz continuous with Lipschitz constant $1$;
     \item\label{fyf} $\Delta_{-K}(y)< 0$ if and only if $ y \in \operatorname{int}(-K)$.
     \item\label{hryx} $\Delta_{-K}(y)=0$ if and only if $ y \in \operatorname{bd}(-K) $.
     \item\label{ufbmvfd} $\Delta_{-K}(y)> 0$ if and only if $y \in \operatorname{int}[(-K)^{c}] $.
     \item\label{ifdyte} $\Delta_{-K}(y+z)= \Delta_{-K-z}(y)$ for all $y, z \in \mathbb{R}^{m}$.
     \item \label{ubrsyu} $\Delta_{-K}(\lambda y)=\lambda \Delta_{\lambda^{-1}(-K)} (y)$ for all $ y \in \mathbb{R}^{m}$ and $ \lambda >0$.
     \item \label{tewyueui}  $ \Delta_{-K}(y+z)\le \Delta_{-K}(y)+\Delta_{-K}(z)$ and $\Delta_{-K}(y)-\Delta_{-K}(y)\le \Delta_{-K}(y-z)$ for all $y,z\in \mathbb R^{m}$.
     \item \label{htytuttd} Given $ y, z \in \mathbb R^{m}$, if  $y \prec z$ (resp., $y \preceq z$), then $\Delta_{-K}(y) < \Delta_{-K}(z)$ (resp., $\Delta_{-K}(y) \le \Delta_{-K}(z)$).
\end{enumerate}
\end{lemma}

With the help of $\Delta_{-K}$ and its properties, the concept of $K$-critical point for (\ref{vec_pr}) is given below.

\begin{definition}\textup{\cite{trust2025setopt}}\label{vecrit} A point $ \bar x \in \mathbb R^{n}$ is said to be a $\tilde{K}$-\emph{critical point} for (\ref{vec_pr}) with respect to $K$ if for any $s \in \mathbb R^{n}$ there exists $ j_{0}\in \{1, 2,3, \ldots, \bar{\omega} \}$  such that $\Delta_{-K}(\nabla f^{a_{j_{0}}}(\bar x)^{\top}s) \geq 0$.
\end{definition}

\begin{lemma}\textup{\cite{trust2025setopt}}\label{rtyrrsv}
Let $\bar{x} \in \mathbb{R}^m$, $\bar{\omega} = \omega(\bar{x})$, and  $\widetilde{K}  \in \mathscr{P}(\mathbb R^{m\bar{\omega}})$ be the cone $\Pi^{\bar{\omega}}_{j=1}{K}.$ Suppose $\preceq_{\widetilde{ K}}$ and $\prec_{\widetilde{ K}}$ be the partial order and the strict order, respectively, in $\mathbb R^{m\bar{\omega}}$ induced by $\widetilde{ K}$. Furthermore, consider the partition set $P_{\bar{x}}$ associated to $\bar{x}$, and define the function ${\widetilde{f}}^{a}$ as in \eqref{eq:fa} for each $a \in P_{\bar{x}}$. Then, $\bar{x}$ is a ${K}$-critical point of \eqref{fgcx} if and only if for every $a \in {P}_{\bar x}$, $\bar x$ is a $\tilde{K}$-critical point of \eqref{vec_pr}.
\end{lemma}

\begin{definition} \cite{Fil_ste_2000} A vector $s \in \mathbb R^{n}$  is said to be a $\tilde{K}$-descent direction of the  objective function ${\widetilde{f}}^{a}: \mathbb R^{n} \to \mathbb R^{m \bar{\omega}}$ of  (\ref{vec_pr}) at $ \bar x$ if 
 there exists $t_{0} >0$ such that 
 \begin{align*}
     {\widetilde{f}}^{a}(\bar x+ ts) \prec_{\widetilde{ K}} {\widetilde{f}}^{a}(\bar x) ~\text{for all}~ t \in (0, t_{0}],
 \end{align*}
 i.e., if there exists $t_{0} >0$ such that for all $ j \in  [\bar{\omega}]$, 
 \begin{align*}
     f^{a_{j}}(\bar x+ ts) \prec_{ {K}} 
      f^{a_{j}}(\bar x) ~\text{for all}~ t \in (0, t_{0}],
 \end{align*}
 i.e., if $ \nabla f^{a_{j}}(\bar x)^{\top} s \in -\text{int}(K)$ for all $ j \in  [\bar{\omega}]$.
\end{definition}
The following result shows that a local weakly $\preceq_{K}$-minimal point of \eqref{fgcx} is a $K$-critical point, and under a convexity assumption on the objective function, $K$-critical points and weakly $\preceq_{K}$-minimal points are identical.   
\begin{theorem} \label{nec_cri}
\begin{enumerate}[(i)]
\item \label{uitu_yrc} If $\bar x$ is a local weakly $\preceq^{l}_{K}$-minimal point of \eqref{fgcx}, then $\bar x$ is a $K$-critical point for \eqref{fgcx}. 
\item If each $f^i$, $i \in [p]$, is $K$-convex and $\bar x$ is a $K$-critical point for \eqref{fgcx}, then $\bar x$ is a weakly $\preceq^{l}_{K}$-minimal point of \eqref{fgcx}.
\end{enumerate}
\end{theorem}

\begin{proof}
\begin{enumerate}[(i)]
\item On contrary, let $\bar x$ be not a $K$-critical point of (\ref{fgcx}). By Lemma \ref{rtyrrsv}, there exists $\bar{a} := (\bar{a}_{1}, \bar{a}_{2}, \ldots, \bar{a}_{\bar{\omega}}) \in 
P_{\bar x}$ such that $\bar{x}$ is not a $\tilde{K}$-critical point of ($\mathcal{VOP}_{\bar a}$). Therefore, there exists $\bar s \in \mathbb R^{n}$ such that 
\begin{equation}\label{aux_dg_17_01} 
\nabla f^{\bar a_{j}}(\bar x)^{\top} \bar s \in -\text{int}(K)~ \text{for all}~ j \in [ \bar{\omega}].
\end{equation} 
 
Since $f^{\bar{a}_{j}} $ is continuously differentiable for all  $j \in [\bar{\omega}]$, we have 
\begin{align}\label{tylor_series}
    f^{\bar a_{j}}(x) = f^{\bar a_{j}}(\bar x) + \nabla f^{\bar a_{j}}(\bar x)^{\top}(x-\bar x)+ o(\lVert x-\bar x\rVert),
\end{align}
  where  $\lim\limits_{x \to \bar x} \frac{o(\lVert x -\bar x\rVert)}{\lVert x-\bar{x}\rVert}=0.$ Since $\bar x$ is a local weakly $\preceq^{l}_{K}$-minimal element of \eqref{fgcx}, there  exists a neighbourhood $\mathcal{N}(\bar x, \delta)$ of $\bar x$ such that 
  \[ \nexists ~x \in \mathcal{N}(\bar x, \delta) \text{ with } \{{{f}^{i}}(x)\}_{i\in [p]}\prec^{l}_{K}\{{f}^{i}(\bar x)\}_{i \in [p]}.\] 
Thus, by Lemma 3.1 in \cite{steepmethset}, $\bar x$ is a local weakly $\tilde{K}$-minimal point of (\ref{vec_pr}) with $a = \bar a$. This implies that there does not exist $ x \in \mathcal{N}(\bar x, \delta)$ with ${{f}^{\bar a_j}}(x) \prec_{{ K}}{{f}^{\bar a_{j}}}(\bar x)$ for all $j \in [\bar {\omega}]$, i.e., $  {{f}^{\bar a_j}}(x) - {{f}^{\bar a_{j}}}(\bar x) \notin -\text{int}(K)$ for any $x \in \mathcal{N}(\bar x, \delta)$ and for all $j \in [\bar{\omega} ]$. From (\ref{tylor_series}), there exists $\bar{\delta} < \delta $  such that for all $j \in [\bar \omega]$, 
\begin{align}\label{tyu}
& \nabla f^{\bar a_{j}}(\bar x)^{\top}(x-\bar x) \notin -\text{int} (K) \text{ for all } x \in \mathcal{N}(\bar x, \bar{\delta}). 
\end{align}
Take any $x'\in \mathcal{N}(\bar{x}, \bar{\delta})$ such that $ x'-\bar x = \bar s$. Then, from (\ref{tyu}), it implies that $\nabla f^{\bar a_{j}}(\bar x)^{\top} \bar s \not\in - \text{int}(K)$ for all $j \in [\bar{\omega}]$. This is contradictory to \eqref{aux_dg_17_01}. Therefore, $\bar x$ is a $K$-critical point for \eqref{fgcx}. \\

\item By the given hypothesis, it follows that for any $\bar a \in P_{\bar x}$, for all $j \in [\bar{\omega}], x \in \mathbb R^{n} $ and $ \mu \in (0,1]$,
\begin{align*}
&  \mu f^{\bar{a}_{j}}(x)+(1-\mu)f^{\bar{a}_{j}}(\bar x)- f^{\bar{a}_{j}}(\mu x+ (1-\mu)\bar x) \in { K}\\
\implies & (f^{\bar{a}_{j}}(x)-f^{\bar{a}_{j}}(\bar x))-\tfrac{1}{{\mu}} \left( {f^{\bar{a}_{j}}( \bar x+\mu(x-\bar x))-f^{\bar{a}_{j}}(\bar x) }\right) \in {K}.
\end{align*}
As $\mu \to 0^{+}$, we obtain for all $ j \in [\bar{\omega}]$ and $x \in \mathbb R^{n}$ that 
\begin{align} 
&  (f^{\bar{a}_{j}}(x)-f^{\bar{a}_{j}}(\bar x))- \nabla f^{\bar{a}_{j}}(\bar x)(x-\bar x)\in {K} \nonumber\\
\implies & \nabla f^{\bar{a}_{j}}(\bar x)(x-\bar x) \preceq_{K} (f^{\bar{a}_{j}}(x)-f^{\bar{a}_{j}}(\bar x)) \label{qwewrc} \\ 
\implies & \Delta_{- K}(\nabla f^{\bar{a}_{j}}(\bar x)(x-\bar x) ) \leq \Delta_{- K}( f^{\bar{a}_{j}}(x)-f^{\bar{a}_{j}}(\bar x),  \text{ by Lemma \ref{proori}(\ref{htytuttd})}.
\end{align}
Since $\bar x$ is a $K$-critical point for \eqref{fgcx}, it follows from Lemma \ref{rtyrrsv} that $\bar{x}$ is a $\tilde{K}$-critical for (\ref{vec_pr}). Hence, for any $s \in \mathbb R^{n}$, there exists $j_{0} \in [\bar {\omega}]$ such that $\Delta_{- K}(\nabla f^{\bar{a}_{j_{0}}}(\bar x)^{\top}s) \geq 0$. 
 In particular, taking $s = x-\bar x$, we have for any $x \in \mathbb R^{n}$ that 
 \begin{align*}
   & 0 \le \Delta_{- K}(\nabla f^{\bar{a}_{j_{0}}}(\bar x)( x- \bar x) ) \leq \Delta_{- K}( f^{\bar{a}_{j_{0}}}( x)-f^{\bar{a}_{j_{0}}}(\bar x)) \\
  \implies &  f^{\bar{a}_{j_{0}}}( x)-f^{\bar{a}_{j_{0}}}(\bar x) \notin - \text{int}{(K)}, ~\text{by Lemma \ref{proori}}~\\
 \implies & f^{\bar{a}_{j_{0}}}(x) \not\prec_{K}  f^{\bar{a}_{j_{0}}}(\bar x).
\end{align*}
Therefore, there exists $j_{0}\in [\bar{\omega}]$ for which there does not exist any $x \in \mathbb R^{n}$ such that
$f^{a_{j_{0}}}(x) \prec_{K}  f^{a_{j_{0}}}(\bar x)$. This implies $\bar x$ is a weakly $\tilde{K}$-minimal element of ($\mathcal{VOP}_{\bar a}$) for all $\bar a \in P_{\bar x}$. Hence, using Lemma 3.1 in \cite{steepmethset}, $\bar{x}$ is weakly $\preceq^{l}_{K}$-minimal element for \eqref{fgcx}. 
\end{enumerate}
\end{proof} 
 
 The following Theorem \ref{critopti}, restated from \cite{trust2025setopt}, shows how the function $\theta: \mathbb{R}^{n} \to \mathbb{R}$ and the map $s: \mathbb{R}^{n} \rightrightarrows \mathcal{B}(x)$ are used to characterize the $K$-critical points of our \eqref{fgcx}, where 
\begin{equation}\label{sdgftdu}
    \resizebox{.87\hsize}{!}{$\theta(x) := \min\limits_{ s \in  \mathcal B(x)}\max\limits_{j \in [ \omega(x)]}\left\{ \Delta_{-{K}}\left(\nabla f^{{a}_{j}}(x)^{\top}s+ \tfrac{1}{2}s^{\top}\nabla^{2}f^{{a}_{j}}(x)s\right) ,\Delta_{-{K}} \left( \nabla f^{{a}_{j}}(x)^{\top}s\right)\right\}$}, 
   \end{equation} 
   \begin{align}\label{hdhvsd} 
   \resizebox{.9\hsize}{!}{$s(x) := \underset{s\in \mathcal B(x)}{\operatorname{argmin}}\max\limits_{j \in [\omega(x)]}\bigg\{\Delta_{-{K}}\bigg(\nabla f^{{a}_{j}}(x)^{\top}s+ \tfrac{1}{2}s^{\top}\nabla^{2}f^{{a}_{j}}(x)s\bigg), \Delta_{-{K}} \bigg( \nabla f^{{a}_{j}}(x)^{\top}s\bigg)\bigg\},$}
\end{align}
and $\mathcal{B}(x) := \{s\in \mathbb{R}^n: \| s\| \le \Omega\}$ for some $\Omega>0$.
\begin{theorem}\textup{\cite{trust2025setopt}}\label{critopti}
For the functions $\theta$ and $s$ in (\ref{sdgftdu}) and (\ref{hdhvsd}), respectively, the following results hold. 
\begin{enumerate}[(a)]
\item The mapping $ \theta$ is continuous in $\mathcal{R}$, where $\mathcal{R}$ is the set of all regular point for \eqref{fgcx}, and well defined in $\mathbb R^{n}$.
\item\label{sucfr} 
The following three statements satisfy $(i) \Longleftrightarrow(ii) \Longrightarrow (iii)$: \begin{enumerate}[(i)]
    \item $\bar{x} \in \mathbb R^{n}$ is not a $K$-critical point for \eqref{fgcx};
    \item\label{gpourer} $\theta(\bar x)< 0$;
    \item $ s(\bar x) \neq {0}$.
\end{enumerate}
\end{enumerate}
\end{theorem}

\section{Non-Monotone Trust-Region Method for Set Optimization}\label{Section3}

In this section, we build on the TRM algorithm proposed in \cite{trust2025setopt} for ($\mathcal{SOP}^{l}_{k}$) and propose Max-NTRM and Avg-NTRM by modifying the reduction ratio in a way
such that it allows for non-monotonic relaxation for the acceptance of a step.

TRM for \eqref{fgcx} primarily involves three steps: vectorization, scalarization, and trust-region-based iterative updates. The vectorization step converts the set optimization problem into a class of vector optimization problems. At the $k$-th iterative point $x_{k}$, the partition set $P_{x_{k}}$ (denoted briefly as $P_{k}$) of weakly minimal solutions of the set objective is generated, from which a particular $a^{k}$ is selected based on the necessary condition of $K$-critical point of ($\mathcal{SOP}^{l}_{k}$). 
Also, denote $P_{x}$ as the partition set at $x$ and $\mathcal{B} := \{ s \in \mathbb R^{n}: \lVert s \rVert \le \Omega_{\max}\}$
with $\Omega_{\max}$ as the maximum allowed trust-region step-size.
Using these definitions, we restate the function $\theta: \mathbb R^{n} \to \mathbb R$,  which characterizes criticality as follows:
\begin{align}\label{rbdwbyew}
     \theta(x) := \min_{(a, s) \in P_{x} \times \mathcal {B}} \Theta_{x}(a,s),
 \end{align}
where $\Theta_{x}: P_{x} \times \mathcal{B} \rightarrow \mathbb R$, for any given $ x \in \mathbb R^{n}$, is given by 
\begin{align}\label{namfun}
   \Theta_{x}(a,s) := \max_{j \in [\omega(x)]} \left\{ \Delta_{-K} \left(\nabla f^{a_{j}}(x)^{\top}s+\tfrac{1}{2}s^{\top}\nabla^{2} f^{a_{j}}(x)s\right), \Delta_{-K}(\nabla f^{a_{j}}(x)^{\top}s) \right\}. 
\end{align} 
 Theorem  \ref{critopti} states that if a point $x_{k}$ is a $K$-critical point of \eqref{fgcx}, then $\theta(x_{k}) = 0$. 
 An ${a}^{k} := \left({a}^{k}_1, {a}^{k}_2, \ldots,  {a}^{k}_{\omega_{k}}\right)$ is selected from the partition set $P_{k}$ by solving
\begin{equation}\label{Sg_aux_17_01_2}
    \resizebox{0.91\hsize}{!}{$
   ({a}^{k},  s_{k}) \in \underset{(a, s)\in P_{k}\times \mathcal B_{k}}{\operatorname{argmin}} \max_{j \in [\omega_k]} \left\{\Delta_{-K}\left(\nabla f^{a_j}(x_{k})^{\top}s + \tfrac{1}{2} s^{\top} \nabla^{2} f^{a_j}(x_k) s \right), \Delta_{-K}\left( \nabla f^{a_j}(x_{k})^{\top}s\right) \right\}$
   },
   \end{equation}
such that $\theta({x_k}) = \Theta_{x_{k}}( {a^{k}},{s_{k}})$, where $\mathcal{B}_k := \mathcal{B}(x_{k}): = \{s\in \mathbb{R}^n: \|s\| \le \Omega_k\}$ with $\Omega_{k}>0$. Using ${a}^{k}$, a vector optimization problem is defined as
\begin{equation*}\label{vop_at_x_k}
(\preceq_{\widetilde{K}}) \text{---} \min_{x \in \mathbb R^{n}}  {\widetilde{f}}^{a^k}(x),~ a^{k} \in P_{k}
\tag{\text{$\mathcal{VOP}_{a^k}(x_k)$}}. 
\end{equation*}
Then, a model function approximation $\widetilde{m}^{a^{k}}$ of $\widetilde{f}^{a^{k}}$ is defined on $\mathcal{B}_{k}$ as follows: 
$${\widetilde{m}}^{{a}^{k}}: = (m^{{a}^{k}_{1}}, m^{{a}^{k}_{2}}, \ldots, m^{{a}^{k}_{\omega_k}})^\top,$$
where 
$m^{{a}^{k}_{j}}(s):= \nabla f^{{a}^{k}_{j}}(x_{k})^{\top}s+ \tfrac{1}{2}s^{\top}\nabla^{2}f^{{a}^{k}_{j}}(x_{k})s,~ s \in \mathcal{B}_k,~ \text{ for each}~ j \in [\omega_k]$,  with
\begin{align*} 
 & \nabla f^{{a}^{k}_{j}}(x_{k})^{\top}s := \left( 
\nabla f^{{a}^{k}_{j},1}( x_{k})^{\top}s, 
\nabla f^{{a}^{k}_{j},2}( x_{k})^{\top}s,  
\ldots, 
\nabla f^{{a}^{k}_{j},m}( x_{k})^{\top}s \right)^\top \text{ and }  \\ 
& s^{\top}\nabla^{2} f^{{a}^{k}_{j}}(x_{k})s := \left(   s^{\top}\nabla^{2} f^{{a}^{k}_j,1}(x_{k})s, 
s^{\top}\nabla^{2} f^{{a}^{k}_j,2}(x_{k})s, 
 \ldots, 
 s^{\top} \nabla^{2} f^{{a}^{k}_j,m}(x_{k})s \right)^\top,
\end{align*}
to compute the step $s_{k} \in \mathcal{B}_{k}$, satisfying (\ref{Sg_aux_17_01_2}). Subsequently, a minimization problem with an oriented distance scalarization of $\widetilde{m}^{a^{k}}$ as the objective is reformulated as follows:
\begin{equation}\label{ghtyu}
\left.\begin{aligned} 
&\min  & & t\\
   & \text{subject to} & &\Delta_{-{K}}\left(\nabla f^{{a}^{k}_j}(x_{k})^{\top}s+ \tfrac{1}{2}s^{\top}\nabla^{2}f^{{a}^{k}_j}(x_{k})s\right)-t\le 0, ~j = 1,2, \ldots, {\omega}_{k},\\
 & & &\Delta_{-{K}}\left(\nabla f^{{a}^{k}_{j}}(x_{k})^{\top}s\right)-t \le 0, ~j =1, 2, \ldots, {\omega}_{k},\\
 & & & \lVert s \rVert \le \Omega_{k},
 \end{aligned}
 \right\}.
\end{equation}
Next, in \cite{trust2025setopt}, an appropriate reduction ratios $\rho^{a^{k}_{j}}_{k}$ for all $j \in [\omega_{k}]$ is defined as 
\begin{align}\label{Ratios_Trm}
~\rho^{{a}^{k}_{j}}_{k} := & \frac{\text{actual function reduction}}{\text{predicted reduction}} =-\frac{\Delta_{- K}(f^{{a}^{k}_{j}}(x_{k}+s_{k})-f^{{a}^{k}_{j}}(x_{k}))}{\Delta_{-K}
({m}^{{a}^{k}_{j}}(0)-{m}^{{{a}^{k}_{j}}}(s_{k}))}.
\end{align}
This is used in the step acceptance criterion and the update rule of the trust-region radius. 

 In this paper, we modify the reduction ratio so that it takes into account not just the function values from the current iteration but also from previous iterations. Based on the way such information is utilized, we propose Max-NTRM and Avg-NTRM in the following. 

\subsection{Max-NTRM for SOP}
Drawing motivation from the work of Ghosh et al. \cite{trust2025setopt}, at the current iterate $x_{k}$, we define the reduction ratios $\rho^{a^{k}_{j}}_{k}$ for Max-NTRM, for all $j \in [\omega_{k}]$, as
\begin{align}\label{Ratios_NTR_Max}
\rho^{a^{k}_{j}}_{k} :=-\frac{\Delta_{-K}\left(f^{a^{k}_{j}}(x_{k}+s_{k})-(f^{a^{k}_{j},{r}}(x_{l^{j,r}(k)}))_{r \in [m]}\right)}{\Delta_{-K}(m^{a^{k}_{j}}_{k}(0)-m^{a^{k}_{j}}_{k}(s_{k}))}, 
\end{align}
where for all $i \in [p]$ and $r\in [m]$,  
\begin{align}\label{ghvt_ugh_huh}
f^{{i,r}}(x_{l^{i,r}(k)}) :=
\begin{cases}
\max_{0 \le q \le N_k} f^{{i,r}}(x_{k-q}) \quad \quad &\text{if} ~ a^{k}= a^{k-1}= \cdots = a^{k-q} = \cdots = a^{k-N_k} \\
f^{{i,r}}(x_{k}) &\text{otherwise}.
\end{cases}
\end{align}
Here, $ N_{0} = 0, N_k = \min\left\{N_{k-1}+1, \hat{N}\right\}$, and $\hat{N}$ is a sufficiently large estimate of the largest value that $N_k$ can take, i.e., the maximum number of past iterations that are considered. For each index \(r \in [m]\), the term \( f^{i,r}(x_{l^{i,r}(k)}) \) denotes the maximum value of the \( r \)-th component of the function \( f^i \) evaluated at some past iterate \( x_{l^{i,r}(k)} \), where the index \( l^{i,r}(k) \) corresponds to the point at which this maximum is attained.
Since $a^{k}_{j}$ is one of the indices $i$, the vector $(f^{a^{k}_{j}, r}(x_{l^{j,r}(k)}))_{r \in [m]}$ corresponds to one of the maximum value evaluations $(f^{i, r}(x_{l^{j,r}(k)}))_{r \in [m]}$. For $\hat{N} = 0$,  note that $f^{{i}}(x_{l^{i}(k)}) = f^{{i}}(x_{k})$ and the reduction ratio of Max-NTRM reduces to (\ref{Ratios_Trm}). \\
Here, we note that (\ref{Ratios_NTR_Max}) reduces to formula (7) in \cite{ramirez2022nonmonotone} when we set $p = 1$, so that $F(x) = f^{1}(x)$ and $K = \mathbb{R}^{n}_{+}$, and to equation (2.17) in \cite{sun2004nonmonotone} when we set $p = 1$ and $r = 1$, so that $F(x) = f^{1,1}(x)$ and $K = \mathbb{R}_{+}$. Then, (\ref{Ratios_NTR_Max}) is a true generalized of the conventional max-type NTRM for multiobjective and single objective optimization. 

Note that in \eqref{ghvt_ugh_huh}, $a^{k}$ of the last $N_{k}$ iterations is required to be the same. This requirement arises because the indices of the objective functions that are in $a^{k}$'s may change in the past  $N_{k}$ iterations.  If any $a^{k}$ in the past $N_{k}$ iterations differs, attainment of the maximum value by the same objective function would be impossible. Therefore, when the set $a^{k}$ of the current and last $N_k$ iterations are the same, we use the maximum of their function values over the last $N_k$ iterations to calculate the set of reduction ratios $\rho^{a^{k}_{j}}_{k}$; otherwise, we simply take the current function value $f^{i,r}(x_{k})$ in the numerator, as in \cite{trust2025setopt}.

Based on the formulation of reduction ratios  $\rho^{a^{k}_{j}}_{k}$ for Max-NTRM, one can determine whether the solution $s_{k}$ of (\ref{ghtyu}) is an acceptable step for $F: \mathcal{B}_k \rightrightarrows \mathbb{R}^m$ or not. As shown below in Proposition \ref{descveri}, for Max-NTRM, $\rho^{a^{k}_{j}}_{k}>0$ implies that  $f^{{a}^{k}_{j}}(x_{k}+s_{k}) \prec_K (f^{a^{k}_{j},{r}}(x_{l^{j,r}(k)}))_{r \in [m]}$. In other words, for acceptance, $s_k$ does not need to be along the descent direction of $f^{{a}^{k}_{j}}$ at $x_k$, as in TRM. This allows a non-monotonic freedom to the function values so that they can increase for some iterations if necessary. Also, if the step $s_{k}$ of $\eqref{vop_at_x_k}$ is accepted, it must also be an accepted step of \eqref{fgcx}, as demonstrated in the next subsection (see Corollary \ref{sg_1_ag1}).

This maximum-based scheme, however, can be disadvantageous in some cases. For example, when objective values at the current iteration are very good, it would be more sensible to utilise them; however, Max-NTRM completely ignores them to instead consider the maximum objective values over past iterations \cite{ahookhosh2012efficient}. To address such an issue, average-based non-monotone schemes have generally been studied
that take a weighted average of the function values from previous iterations. We next propose such an average-based non-monotone scheme for \eqref{fgcx}.

\subsection{Avg-NTRM for SOP}
At a given iterate $x_{k}$, we define the reduction ratios $\rho^{a^{k}_{j}}_{k}$, for all $j \in [\omega_{k}]$, of Avg-NTRM for \eqref{fgcx} by 
\begin{align}\label{NTR_Average_Ratios}
\rho^{a^{k}_{j}}_{k} := \frac{-\Delta_{-K}(f^{a^{k}_{j}}(x_{k}+s_{k}) - (C^{a^{k}_{j},{r}}_{k})_{r \in [m]})}{\Delta_{-K}(m^{a^{k}_{j}}(0)-m^{a^{k}_{j}}(s_{k}))},  
\end{align}
where for all $i \in [p]$,
\begin{align}\label{tbddv_uiet}
C^{{i,r}}_{k} :=
\begin{cases}
\frac{\mu_{k-1}q_{k-1}}{q_{k}} C^{{i,r}}_{k-1} + \frac{1}{q_{k}}f^{{i,r}}(x_{k}) \quad \quad &\text{if} ~ a^{k}= a^{k-1}= \cdots = a^{0} \\
f^{{i,r}}(x_{k}) &\text{otherwise},
\end{cases}
\end{align}
where 
\begin{align}
q_{k} := 
\begin{cases}
 1 \quad \quad &\text{if}~k =0\\
 \mu_{k-1} q_{k-1} +1 \quad \quad &\text{if} ~k \ge 1,
\end{cases}
\nonumber
\end{align} 
for $\mu_{k} \in [\mu_{\min}, \mu_{\max}]$, with $\mu_{\min} \in [0,1)$ and $\mu_{\max} \in [\mu_{\min},1)$. Since $a^{k}_{j}$ is one of the indices $i$, the vector $(C^{a^{k}_{j},{r}}_{k})_{r \in [m]}$ corresponds to one of the weighted average evaluations $C^{{i,r}}_{k}$. For $\hat{N} = 0$,  note that $f^{{i}}(x_{l^{i}(k)}) = f^{{i}}(x_{k})$ and the reduction ratio of Max-NTRM reduces to (\ref{Ratios_Trm}).
For $\mu_{k} = 0$ for all $k$, we get $C^{i}_{k} = f^{{i}}(x_{k})$ and the reduction ratio of Avg-NTRM reduces to (\ref{Ratios_Trm}).  

We note that, if we consider $p = 1$ and hence $F(x) = \{f^{1}(x)\}$, and if we consider $p=1, r = 1$ and hence $F(x) = \{f^{1,1}(x)\}$, (\ref{NTR_Average_Ratios}) and (\ref{tbddv_uiet}) reduces to (12) in \cite{ghalavand2023two} for multiobjective and to (2.2) in \cite{mo2007nonmonotone} for singleobjective.  Thus, (\ref{NTR_Average_Ratios}) is a true generalization.

 In \eqref{NTR_Average_Ratios}, $a^{k}$ is kept the same across all past iterations to avoid conflict in choosing $f^{i}$ for the weighted average calculation, as the index of the objective functions that are in $a^{k}$ may change over different past iterations. If any of the $a^{k}$'s from all previous iterations is different, then, we use the current function value $f^{i,r}(x_{k})$ to calculate $\rho^{a^{k}_{j}}_{k}$, as in \cite{trust2025setopt}.

From \eqref{NTR_Average_Ratios} of Avg-NTRM, it is observed that $\rho^{a^{k}_{j}}_{k}>0$ implies $f^{{a}^{k}_{j}}(x_k + s_k) \prec_K  (C^{a^{k}_{j},{r}}_{k}(x_{k}))_{r \in [m]}$, which means that the step $s_{k}$ is accepted for $\eqref{vop_at_x_k}$ (see Proposition \ref{descveri}) and consequently also for $\eqref{fgcx}$ (see Corollary \ref{sg_1_ag1}). Therefore, Proposition \ref{descveri} and Corollary \ref{sg_1_ag1} give rise to the statement of Theorem \ref{hfrhnye_uyhio}, which fulfills the requirements of Max-NTRM and Avg-NTRM. The only requirements for these two NTRMs are that the sequences $\{f^{i}(x_{l^{i}(k)})\}$ and $\{C^{i}_{k}\}$ must be monotonic. This, in turn, allows the sequence ${f^{i}(x_{k})}$ to be non-monotonic, with $\left\{f^{i}(x_{l^{i}(k)})\right\} = \left\{\left(f^{{i,r}}(x_{l^{i,r}(k)})\right)_{r \in [m]}\right\}$ and $\{C^{i}_{k}\}= \left\{ (C^{{i,r}}_{k})_{r \in [m]}\right\}$.

Next, similar to Proposition 3.2 of \cite{trust2025setopt}, we show a relation  between the values of reduction ratios and the condition of step acceptance.
\begin{proposition}\label{descveri}
Let $x_k$ be not a $K$-critical point of \eqref{fgcx}. Then, for any $j \in [\omega_k]$, $s_{k}$ satisfies $f^{{a}^{k}_{j}}(x_k + s_k) \prec_K (f^{a^{k}_{j},{r}}(x_{l^{r}_{j}(k)}))_{r \in [m]}$ for Max-NTRM or $f^{{a}^{k}_{j}}(x_k + s_k) \prec_K  (C^{a^{k}_{j},{r}}_{k}(x_{k}))_{r \in [m]}$ for Avg-NTRM if and only if $\rho^{{a}^{k}_{j}}_{k} >0$. 
\end{proposition}

\begin{proof}
 Let $s_{k}$ satisfy $f^{{a}^{k}_{j}}(x_k + s_k) \prec_K (f^{{a}^{k}_{j},{r}}(x_{l^{j,r}(k)}))_{r \in [m]}$ (Max-NTRM) or $f^{{a}^{k}_{j}}(x_k + s_k) \prec_K (C^{a^{k}_{j},r}_{k})_{r \in [m]}$ (Avg-NTRM). Then,  $\Delta_{-K}(f^{{a}^{k}_{j}}(x_{k}+s_{k})-(f^{{a}^{k}_{j},{r}}(x_{l^{j,r}(k)}))_{r \in [m]} ) < 0 $ (Max-NTRM) or $\Delta_{-K}(f^{{a}^{k}_{j}}(x_{k}+s_{k})-(C^{a^{k}_{j},r}_{k})_{r \in [m]}) < 0$ (Avg-NTRM). As $s_{k}$ is identified by solving (\ref{ghtyu}), we get 
\begin{align*}
\max_{ j \in [\omega_k]} \Delta_{-K}(m^{{a}^{k}_{j}}(s_{k})) ~&~ \le \max_{ j \in [\omega_k]} \left\{\Delta_{-K}(m^{{a}^{k}_{j}}(s_{k})),~\Delta_{-K}(\nabla f^{{a}^{k}_{j}}(x_{k})^{\top}s_k) \right\} \\ 
~&~ \le \max_{ j \in [\omega_k]} \left\{\Delta_{-K}(m^{{a}^{k}_{j}}(0)),~\Delta_{-K}(\nabla f^{{a}^{k}_{j}}(x_{k})^{\top}0) \right\} = 0, 
\end{align*}
i.e., $\Delta_{-K}(m^{{a}^{k}_{j}}(s_{k})) \le 0$. Thus,  
\begin{align}\label{dg_aux_19_01_1}
     \Delta_{-K}(m^{{a}^{k}_{j}}(0)-m^{{a}^{k}_{j}}(s_{k})) \ge \Delta_{-K}(m^{{a}^{k}_{j}}(0))-\Delta_{-K}(m^{{a}^{k}_{j}}(s_{k})) \ge 0.
 \end{align}
As $x_k$ is not $K$-critical, from Corollary \ref{tncdfd}, we get $\Delta_{-K}(m^{{a}^{k}_{j}}(0)-m^{{a}^{k}_{j}}(s_{k})) > 0$. Thus, 
 \begin{align}
     \rho^{{a}^{k}_{j}}_{k}=& ~~ -\frac{\Delta_{- K}(f^{{a}^{k}_{j}}(x_{k}+s_{k})-(f^{{a}^{k}_{j},{r}}(x_{l^{j,r}(k)}))_{r \in [m]})}{\Delta_{-K}
({m}^{{a}^{k}_{j}}(0)-{m}^{{{a}^{k}_{j}}}(s_{k}))} > 0~(\text{Max-NTRM})~\\
=& ~~ - \frac{\Delta_{- K}(f^{{a}^{k}_{j}}(x_{k}+s_{k})-((C^{a^{k}_{j},r}_{k})_{r \in [m]})}{\Delta_{-K}
({m}^{{a}^{k}_{j}}(0)-{m}^{{{a}^{k}_{j}}}(s_{k}))} > 0~(\text{Avg-NTRM}). 
\end{align}

Conversely, suppose $\rho^{{a}^{k}_{j}}_{k} > 0$. Then,  $\Delta_{-K}(f^{{a}^{k}_{j}}(x_{k} + s_{k})-(f^{{a}^{k}_{j},{r}}(x_{l^{j,r}(k)}))_{r \in [m]}) < 0$ (Max-NTRM) or $\Delta_{-K}(f^{{a}^{k}_{j}}(x_{k} + s_{k})- (C^{a^{k}_{j},r}_{k})_{r \in [m]}) < 0$ (Avg-NTRM), and hence 
$f^{{a}^{k}_{j}}(x_{k}+s_{k}) \prec_K (f^{{a}^{k}_{j},{r}}(x_{l^{j,r}(k)})_{r \in [m]}$ (Max-NTRM) or $f^{{a}^{k}_{j}}(x_{k}+s_{k}) \prec_K (C^{a^{k}_{j},r}_{k})_{r \in [m]}$ (Avg-NTRM). 
\end{proof}

Based on the reduction ratios (\ref{Ratios_NTR_Max}) and (\ref{NTR_Average_Ratios}), Max-NTRM and Avg-NTRM for \eqref{fgcx} are given in Algorithm \ref{alg_max} and Algorithm \ref{alg_avg}, respectively.

\begin{algorithm}
\caption{Max-NTRM algorithm for solving \eqref{fgcx}}
\label{alg_max}

\begin{enumerate}[1.]
\item  
\emph{Initialization and inputs}\\ 
Provide twice continuously differentiable functions  $f^i: \mathbb{R}^n \to \mathbb{R}^m$, $i = 1, 2, \ldots, p$, for the problem \eqref{fgcx} and the ordering cone $K \subseteq \mathbb R^{m}$.\\
Initialize parameters: iteration counter $k:=0$, initial point $x_{0}$, initial trust-region radius $\Omega_{0}$, maximum trust-region radius $\Omega_{\max}$, threshold parameters $\eta_{1}, \eta_{2} \in (0,1)$ to determine if the step is successful or very successful, fractions $\gamma_{1}, \gamma_{2}\in (0,1)$ to shrink the trust-region radius, tolerance value $\epsilon$ for stopping criterion, $N_{0} = 0$ is the initial number of past iterations to be considered in the maximum calculation, and $\hat{N}$ is an upper bound on the number of previous iterations over which the maximum value of $\{ f_{i}\}_{i\in [p]}$ is evaluated.
 
\item 
\emph{Find $K$-minimal elements and partition set} \label{gtesv_max} \\ 
Calculate $M_{k} := \text{Min}({F}(x_{k}), K) := \{r_1,r_2,\ldots,r_{w_k}\}$. \\
Calculate the partition set $P_{k}:= P_{x_{k}} := I_{r_1}\times I_{r_2}\times\cdots\times  I_{r_{\omega_k}}$. \\ 
Compute $p_k := \lvert P_{k} \rvert$ and $\omega_{k} := \lvert \text{Min}({F}(x_{k}), K)\rvert$. 

\item 
\label{step3_choice_of_a_max} \emph{Selection of an `$a^{k}$' from $P_{k}$} \\
Determine $\mathcal{B}_{k} = \{ s_{k} \in \mathbb{R}^{n} : \lVert s_{k} \rVert \le \Omega_{k}\}$.
Choose an element $a^{k}= (a^{k}_{1}, a^{k}_{2}, \ldots, a^{k}_{\omega_{k}}) \in P_{k}$ by 
\begin{align*}
(a^{k},s_{k}) \in \underset{(a,s)\in P_{k}\times \mathcal{B}_{k} }{\operatorname{argmin}} \max_{j \in [\omega_k]} \left\{ \Delta_{-K}\left(\nabla f^{a_j}(x_{k})^{\top}s + \tfrac{1}{2} s^{\top} \nabla^{2} f^{a_j}(x_k) s \right), \Delta_{-K}\left(\nabla f^{a_j}(x_{k})^{\top}s \right)\right\}.
\end{align*}
\end{enumerate}
\end{algorithm}

\begin{algorithm}   
\begin{enumerate}\setcounter{enumi}{3}  
\item \emph{Definition of model functions}\\ 
For all $j \in [{\omega}_{k}]$, calculate the model functions $m^{a^{k}_{j}}$ as
\begin{align*}
m^{a^{k}_{j}}_{k}(s) := \nabla f^{a^{k}_{j}}(x_{k})^{\top}s+\tfrac{1}{2}s^{\top} \nabla^{2}f^{a^{k}_{j}}(x_{k})s, ~~~\lVert s \rVert \le {\Omega}_{k}. 
\end{align*} 

\item \emph{Step computation}\label{yjuu_max} \\ 
Calculate $(t_{k}, s_{k})$ by solving the subproblem (\ref{ghtyu}).

\item \emph{Stopping criterion} \label{tocri_max} \\ 
If $\lvert t_{k} \rvert < \epsilon$, terminate the algorithm and output $x_{k}$ as a $K$-critical point of \eqref{fgcx}.\\ 
Otherwise, go to the next step.

\item \emph{Maximum value calculation}\\
For the maximum value of each $f^{i}$,  we first consider the set of indices used in past iterates, where $a^{k-q}$ is the same for $q = 0, 1, 2, \ldots, N_{k}$. Then, we calculate
\begin{align*}
   f^{i,r}(x_{l^{i,r}(k)}) = \max_{0 \le q \le N_{k}} f^{i,r}(x_{k-q}) ~\text{and}~ l^{i,r}(k) = \text{argmax} f^{i,r}(x_{k-q}).  
\end{align*}
Otherwise, 
 \begin{align}
    f^{i,r}(x_{l^{i,r}(k)}) = f^{i,r}(x_{k})  ~\text{and}~ l^{i,r}(k) = k.
 \end{align}
\item 
\emph{Reduction ratio}\\
Compute reduction ratio by (\ref{Ratios_NTR_Max}):\label{rati_ntr_nmax}
\begin{align*}
\rho^{a^{k}_{j}}_{k} :=-\frac{\Delta_{-K}\left(f^{a^{k}_{j}}(x_{k}+s_{k})-(f^{a^{k}_{j},{r}}(x_{l^{j,r}(k)}))_{r \in [m]}\right)}{\Delta_{-K}(m^{a^{k}_{j}}(0)-m^{a^{k}_{j}}(s_{k}))} ~\text{for all}~ j \in [\omega_{k}]. 
\end{align*}

\item \emph{Step acceptance criterion}
\label{getwqs_max}\\ 
If $\rho^{a^{k}_{j}}_{k}\ge \eta_{1}$ for all $j \in [{\omega}_{k}]$, then successful step. Set $x_{k+1} := x_{k}+ s_{k}.$\\
Else if $ \rho^{a^{k}_{j}}_{k}< \eta_{1}$ for at least one $j$, then unsuccessful step. Set $x_{k+1} := x_{k}$.

\item \emph{Update trust-region radius}\label{trust_region_radius_update_max} \\ 
Select 
\begin{equation}
\resizebox{\hsize}{!}{$
\Omega_{k+1}\in 
\begin{cases}
(\gamma_{2} \Omega_{k}, \Omega_{k}], & \text{if}~\eta_{1} \le \rho^{{a}^{k}_{j}}_{k} ~\forall~ j \in [{\omega}_{k}]~ \text{and}~ \exists~ l\in [\omega_{k}] ~\text{such that}~\rho^{{a}^{k}_{l}}_{k} < \eta_{2} ~ (\text{Successful}) \\
(\Omega_{k}, \infty), & \text{if}~ \eta_{2} \le \rho^{{a}^{k}_{j}}_{k} ~\forall ~j\in [{\omega}_{k}],~ (\text{Very successful})\\
[\gamma_{1}\Omega_{k}, \gamma_{2}\Omega_{k}], &\text{if}~ \exists~ l \in [\omega_{k}]~\text{such that}~ \rho^{{a}^{k}_{l}}_{k} ~\le \eta_{1}~ (\text{Unsuccessful}). 
\end{cases} $}
\end{equation}

\item \emph{Update $N_k$ and go to next iteration}\\
Set $N_{k+1} := \min\{N_k + 1,~\hat{N}\}$.\\
Set $k := k+1$ and go to Step \ref{gtesv_max}.

\end{enumerate}
\end{algorithm}

\newpage 
In Algorithm 1, by replacing the formula for the computation of the reduction ratio by (12), we get the algorithm for average-NTRM, which we refer to as Algorithm 2.

\setcounter{algorithm}{1} 
\begin{algorithm}[H]
\caption{Average-NTRM algorithm for solving \eqref{fgcx}}

\label{alg_avg}

\begin{list}{}{\leftmargin=2.5em \labelwidth=2.5em \labelsep=0.5em \itemindent=0em}
 \item[\text{1--6.}] Same as in Algorithm 1.
 \item[\text{7.}]\emph{Average value calculation}\\
 For the average value of each $f^{i}$,  we first consider the set of indices used in past iterates, where $a^{k-q}$ is the same for $q = 0, 1, 2, \ldots, N_{k}$. Then, we calculate
 \begin{align*}
C^{{i,r}}_{k} :=
\begin{cases}
\frac{\mu_{k-1}q_{k-1}}{q_{k}} C^{{i,r}}_{k-1} + \frac{1}{q_{k}}f^{{i,r}}(x_{k}) \quad \quad &\text{if} ~ a^{k}= a^{k-1}= \cdots = a^{0} \\
f^{{i,r}}(x_{k}) &\text{otherwise},
\end{cases}
\end{align*}
where 
\begin{align}
q_{k} := 
\begin{cases}
 1 \quad \quad &\text{if}~k =0\\
 \mu_{k-1} q_{k-1} +1 \quad \quad & \text{if} ~k \ge 1,
\end{cases}
\nonumber
\end{align} 
for $\mu_{k} \in [\mu_{\min}, \mu_{\max}]$, with $\mu_{\min} \in [0,1)$ and $\mu_{\max} \in [\mu_{\min},1)$.
 \item[\text{8.}] 
\emph{Reduction ratio}\\
Compute reduction ratio as per (\ref{NTR_Average_Ratios}):\label{rati_ntr_navg}
\begin{align*}
\rho^{a^{k}_{j}}_{k} = -\frac{\Delta_{-K}\left(f^{a^{k}_{j}}(x_{k}+s_{k}) - (C^{a^{k}_{j},{r}}_{k})_{r \in [m]}\right)}{\Delta_{-K}(m^{a^{k}_{j}}(0)-m^{a^{k}_{j}}(s_{k}))},
\end{align*}

 \item[\text{9--10.}] Same as in Algorithm 1.
\end{list}
\end{algorithm}

\subsection{Well-definedness of Algorithm \ref{alg_max} and Algorithm \ref{alg_avg}}\label{welldefine_non_monotone_trust_algori}

The well-definedness of the proposed Algorithm \ref{alg_max} and Algorithm \ref{alg_avg} depends on their respective Step \ref{step3_choice_of_a_max}, Step \ref{yjuu_max}, and Step \ref{rati_ntr_nmax}. Among them, Step \ref{step3_choice_of_a_max}, and Step \ref{yjuu_max} are the same as in TRM algorithm, whose well-definedness is given in Subsection 3.6 of \cite{trust2025setopt}.  

Step \ref{rati_ntr_nmax} calculates the respective reduction ratios of the proposed two algorithms.  

If we reach Step \ref{rati_ntr_nmax}, that means $x_{k}$ is  not a $K$-critical point because if $x_{k}$ is a $K$-critical point of \eqref{fgcx}, then by Theorem \ref{critopti}, $t_{k}= \theta(x_{k})=0,$ and algorithm should have stopped at Step \ref{tocri_max}. From Corollary \ref{tncdfd}, since $x_{k}$ is not $K$-critical, the solution $s_{k}$ is identified from the subproblem (\ref{ghtyu}), which ensures that predicted reduction $\Delta_{-K}(m^{{a}^{k}_{j}}(0)-m^{{a}^{k}_{j}}(s_{k}))$, in the denominator of reduction ratio, is positive. Also, for $s_{k}$ to be accepted, we have $f^{{a}^{k}_{j}}(x_k + s_k) \prec_K (f^{{a}^{k}_{j},{r}}(x_{l^{j,r}(k)}))_{r \in [m]}$ (Max-NTRM) and $f^{{a}^{k}_{j}}(x_k + s_k) \prec_K (C^{a^{k}_{j},r}_{k})_{r \in [m]}$ (Avg-NTRM). This, thus, creates an interdependence between reduction ratios $\rho^{a^{k}_{j}}_{k}$ and the condition of step acceptance as shown in Proposition \ref{descveri}. This interdependence causes Step \ref{rati_ntr_nmax} of Max-NTRM and Avg-NTRM for $\eqref{fgcx}$ to be analogous to the step acceptance criterion of the conventional Max-NTRM and Avg-NTRM, respectively.
Thus, the choice of $\rho^{a^{k}_{j}}_{k}$ for both the algorithms is well-defined. Hence, Step \ref{rati_ntr_nmax} is well-defined.\\

\section{Global Convergence Analysis}\label{global_conv_ana} 
In this section, we prove that both Max-NTRM and Avg-NTRM globally converge to a $K$-critical point for \eqref{fgcx}. For this, we first make a few assumptions about the set-valued map $F = \{f^{i}\}_{i \in [p]}$. These assumptions are very common in single and multi-objective non-monotone trust-region schemes.
 
\begin{assumption}\label{fun_bndbelow}
Function $f^{i}$, for all $i \in [p]$, is bounded below.
\end{assumption}

\begin{assumption}
    There exists $\mathcal{K}_{3}>0$ such that for each $ i \in [p]$, 
    \begin{align*}
        \lVert \nabla f^{{i}}(x) \rVert \le \mathcal{K}_{3} ~\text{for all}~x \in \mathbb{R}^{n}.
       \end{align*}
\label{grad_bnd}
\end{assumption}

\begin{assumption}\label{hess_bnd} 
Hessian of the function $f^{i}$, for all $i \in [p]$, is uniformly bounded, i.e., there exists a $\mathcal{K}_{1}>0$ such that
\begin{align*}
&\lVert \nabla^{2} f^{i}(x) \rVert \le \mathcal{K}_{1} ~\text{ for all } x \in \mathbb{R}^n.
\end{align*} 
\end{assumption}

\begin{assumption}\label{boun_term} \label{grwghcxdbn}
There exists a constant vector $M$ in $\mathbb R^{m}$ whose every coordinate is a positive real number, such that 
 \begin{align*}
      s^{\top}\nabla^{2} f^{i}(x)s \preceq_{K} 
 M \lVert s\rVert^{2} ~\text{for all}~ i \in [p] ~\text{and for all}~x \in \mathbb {R}^{n},
 \end{align*} 
 where $s^{\top}\nabla^{2} f^{i}(x)s := \left(   s^{\top}\nabla^{2} f^{i,1}(x_{k})s, 
s^{\top}\nabla^{2} f^{i,2}(x_{k})s, 
 \ldots, 
 s^{\top} \nabla^{2} f^{i,m}(x_{k})s \right)^\top$.
\end{assumption}

\begin{assumption}\label{label_boundd}  
The level set  ${\mathcal{L}_{0}} := \{ x \in \mathbb R^{n}: F(x) \preceq^{l}_{K} F(x_{0})\}$, is bounded. 
\end{assumption}

With these assumptions, we present a series of lemmas, corollaries, and theorems that ultimately lead to the proof of global convergence of the proposed algorithms. First, we recall Theorem \ref{fsgcsb} and Corollary \ref{tncdfd} from \cite{trust2025setopt}. Then, we present Lemma \ref{WeY_rjt}, to be used in Corollary \ref{utrw_kkjx}, to show a bound on the sufficient decrease of the model function $m^{a^{k}}$. Then, in Theorem \ref{hfrhnye_uyhio}, we show that the sequence of function values generated by the two algorithms monotonically decreases. Next, we present Corollary \ref{sg_1_ag1} that connects the accepted step of \eqref{vop_at_x_k} to the accepted step of \eqref{fgcx}. Then, Theorem \ref{rtyue_ouit} ensures that the algorithms always generate a successful step after a finite number of unsuccessful steps. Next, in Lemma \ref{admits_limit_max_avg}, we show that the sequence of functions generated by the two algorithms admits a limit point. Finally, in Theorem \ref{gjhg_uyhugy}, we show that the sequence of iterates generated by the two algorithms converges to the $K$-critical point of $\eqref{fgcx}$.

\begin{theorem}\textup{\cite{trust2025setopt}}\label{fsgcsb}
Let $x_{k}$ be not a $K$-critical point of \eqref{fgcx}, and $v \in \mathbb R^{n}$ be a descent direction of $f^{a^{k}_j}$ at $x_{k}$ for all $j \in [\omega_k]$. Then, for each $j \in [\omega_{k}]$, there exists $\bar{t}_j >0$ satisfying $\lVert {\bar{t}_j} v \rVert \le \Omega_{k}$ such that 
\begin{align*}
\Delta_{-K}(m^{{a}^{k}_{j}}\left({\bar{t}_j}v\right)) \le  \Delta_{-K}(m^{{a}^{k}_{j}}\left(tv\right)) \text{ for all }  {t} \ge 0 ~\text{satisfying}~\lVert {t} v\rVert \le \Omega_{k}.  
\end{align*}
 In addition, for each $j \in [\omega_{k}]$, 
 \begin{align*}
 & \Delta_{-K}(m^{{a}^{k}_{j}}(0)- m^{{a}^{k}_{j}}\left({\bar{t}_j}v\right))\ge -\tfrac{1}{2} \frac{\Delta_{-K}(\nabla f^{a^{k}_{j}}(x_{k})^{\top} v) }{\lVert v \rVert} \min\left\{-\frac{\Delta_{-K}(\nabla f^{a^{k}_{j}}(x_{k})^{\top}v)}{\lVert  v \rVert \mathcal{K}_{1}}, ~\Omega_{k} \right\}.
\end{align*}
\end{theorem}

\begin{cor}\cite{trust2025setopt}\label{tncdfd} 
If $s_{k}$ is a solution of the subproblem \eqref{ghtyu}, and $x_{k}$ is a not a $K$-critical point of \eqref{fgcx}, then there exists a positive constant $\beta$ such that  for all $j \in [\omega_{k}],  \Delta_{-K}(m^{{a}^{k}_{j}}_{k}(0)) \ge \Delta_{-K} (m^{{a}^{k}_{j}}_{k}(s_{k}))$, and
 \begin{equation*}
 \Delta_{-K}(m^{{a}^{k}_{j}}_{k}(0)- m^{{a}^{k}_{j}}_{k}\left(s_{k}\right))  \ge -\tfrac{\beta}{2} \frac{\Delta_{-K}(\nabla f^{{a}^{k}_{j}}(x_{k})^{\top} s_{k})}{\lVert s_{k} \rVert} \min\left\{-\frac{\Delta_{-K}(\nabla f^{{a}^{k}_{j}}(x_{k})^{\top}s_{k})}{\lVert  s_{k} \rVert \mathcal{K}_{1}}  , \Omega_{k}  \right\}.
 \end{equation*}
\end{cor}

Next, we present a lemma that will be used in Corollary \ref{utrw_kkjx} to show the sufficient decrease of the model $m^{a^{k}}$ in terms of the optimal value $\theta(x_k)$ of the subproblem (\ref{ghtyu}).

\begin{lemma}\label{WeY_rjt}
Let $x_{k}$ be a not a $K$-critical point for \eqref{fgcx}. Further, let $f^{i,r}\in C^{2}(\mathbb R^{n}, \mathbb R)$, and let $\nabla^{2} f^{i,r}(x)$ be positive definite for all $i\in [p]$, and $ r \in [m]$. Additionally, set $ T := - \Delta_{-K}(-M) >0$, where $M$ fulfills Assumption \ref{grwghcxdbn}. Then, $s_{k}$ given by solving the subproblem (\ref{ghtyu}) satisfies
\begin{align}
    & \max_{j \in [\omega_{k}]} \{\Delta_{-K}(f^{a^{k}_{j}}(x_k))^{\top} s_{k}) \} \le - \lvert \theta(x_{k}) \rvert \label{ufxvvnn}\\
  ~\text{and}~  & \lVert s_{k} \rVert^{2} \le \frac{4}{T} \lvert \theta(x_{k})\rvert.\label{hgjkg_arin}
\end{align}
\end{lemma}

\begin{proof}
Since $x_{k}$ is a not a $K$-critical point of \eqref{fgcx}, we have $s_{k} := s(x_{k}) \neq 0$ and $\theta(x_{k}) \neq 0$ from Theorem \ref{critopti}. Then, from the assumption that $\nabla^{2} f^{i,r}(x)$ is positive definite for all $i \in [p]$ and  $r \in [m]$, and from (\ref{sdgftdu}), we have

\allowdisplaybreaks
\resizebox{\linewidth}{!}{
  \begin{minipage}{\linewidth}
\begin{align*}
 &  0 < s_{k}^{\top} \nabla^{2} f^{i,r}(x_{k}) s_{k} \\
\implies & s_{k}^{\top} \nabla^{2} f^{i,r}(x_{k}) s_{k} \in \mathbb{R}_{+} \\
\implies & s_{k}^{\top} \nabla^{2} f^{i}(x_{k}) s_{k} \in \mathbb{R}^{n}_{+} ~\text{as}~ s^{\top}\nabla^{2} f^{i}(x)s = \left(   s^{\top}\nabla^{2} f^{i,1}(x_{k})s, 
s^{\top}\nabla^{2} f^{i,2}(x_{k})s, 
 \ldots, 
 s^{\top} \nabla^{2} f^{i,m}(x_{k})s \right)^\top\\
\implies & s_{k}^{\top} \nabla^{2} f^{i}(x_{k}) s_{k} \in \{\mathbb{R}^{n}_{+}\}\subseteq K  ~\text{for all}~i \in [p] \\
  \implies  & 0 \prec_{K} s_{k}^{\top}\nabla^{2} f^{i}(x_{k})s_{k} ~\text{for all}~i \in [p]\\
  \implies  &  0 \prec_{K} s_{k}^{\top} \nabla^{2} f^{a^{k}_{j}}(x_{k})s_{k}~\text{for all}~ j \in [\omega_{k}]\\
  \implies & \nabla f^{a^{k}_{j}}(x_{k})^{\top} s_{k} \prec_{K} \nabla f^{a^{k}_{j}}(x_{k})^{\top} s_{k} + \frac{1}{2} s_{k}^{\top} \nabla^{2} f^{a^{k}_{j}}(x_{k}) s_{k}~\text{for all}~j \in [\omega_{k}] \\
  \implies & \Delta_{-K}(\nabla f^{a^{k}_{j}}(x_{k})^{\top}s_{k}) < \Delta_{-K}(\nabla f^{a^{k}_{j}}(x_{k})^{\top} s_{k}+ \frac{1}{2}s_{k}^{2}\nabla^{2}f^{a^{k}_{j}}(x_{k})s_{k})~\text{for all}~j \in [\omega_{k}]\\
  \implies & \max_{j \in [\omega_{k}]} \{\Delta_{-K}(\nabla f^{a^{k}_{j}}(x_{k})^{\top}s_{k})\}\\&< \max_{j \in [\omega_{k}]} \{\Delta_{-K}(\nabla f^{a^{k}_{j}}(x_{k})^{\top} s_{k}+ \frac{1}{2}s_{k}^{2}\nabla^{2}f^{a^{k}_{j}}(x_{k})s_{k})\}\\
&< \max_{j \in [\omega_{k}]} \{\Delta_{-K}(\nabla f^{a^{k}_{j}}(x_{k})^{\top} s_{k}+ \frac{1}{2}s_{k}^{2}\nabla^{2}f^{a^{k}_{j}}(x_{k})s_{k}), \Delta_{-K}(\nabla f^{a^{k}_{j}}(x)^{\top} s_{k})\}\\
&= \theta(x_{k}).
\end{align*}
  \end{minipage}
}

Now, since $x_{k}$ is not $K$-critical, we have $\theta(x_{k}) < 0$, i.e., $\theta(x_{k}) = -\lvert \theta(x_{k}) \rvert$. This proves (\ref{ufxvvnn}). 

For proving the second inequality (\ref{hgjkg_arin}), the Lagrangian of problem (\ref{ghtyu}) is given by
\begin{align}\label{wer_opty}
L((\tau_{1},s_{k}),(\lambda^{1},\lambda^{2}, \lambda^{3})) := ~&\tau_{1} + \sum_{j =1}^{\omega_{k}} \lambda^{1}_{j}(\Delta_{-K}(\nabla f^{a^{k}_{j}}(x_{k})^{\top} s_{k}+\frac{1} {2}s^{\top}_{k}\nabla^{2}f^{a^{k}_{j}}(x_{k})s_{k})-\tau_{1})) + \sum_{j=1}^{\omega_{k}} \lambda^{2}_{j} \nonumber\\&(\Delta_{-K}(\nabla f^{a^{k}_{j}}(x_{k})^{\top}s_{k})-\tau_{1}))+ \lambda^{3}\left(\frac{\lVert s_{k} \rVert^{2}}{2}- \frac{\Omega^{2}_{k}}{2}\right)~\text{for all}~j \in [\omega_{k}],
    \end{align}
   with $\lambda^{1}=(\lambda^{1}_{1}, \lambda^{1}_{2}, \ldots, \lambda^{1}_{\omega_{k}}) \in \mathbb{R}^{\omega_{k}}, \lambda^{2}= (\lambda^{2}_{1},\lambda^{2}_{2}, \ldots, \lambda^{2}_{\omega_{k}} )\in \mathbb R^{\omega_{k}}$ and $ \lambda^{3} \in \mathbb R$.
The Karush-Kuhn-Tucker conditions for the problem (\ref{ghtyu}) are given by
\begin{align}
\begin{split}
&\sum_{j=1}^{\omega_{k}}\lambda^{1}_{j} \Delta_{-K}\left((\nabla f^{a^{k}_{j,r}}(x_{k})^{\top}s_{k})_{r \in [m]} + (s^{\top}_{k}\nabla^{2} f^{a^{k}_{j,r}}(x_{k})s_{k})_{r \in [m]}\right) \\
&\quad+ \sum_{j=1}^{\omega_{k}} \lambda^{2}_{j} \Delta_{-K}\left((\nabla f^{a^{k}_{j,r}}(x_{k})^{\top}s_{k})_{r \in [m]}\right) + \lambda^{3} s_{k} = 0 \quad \forall j \in [\omega_{k}]
\end{split} \label{jghutfvre}\\
& \sum_{j=1}^{\omega_{k}} \lambda^{1}_{j} =1, \sum_{j=1}^{\omega_{k}} \lambda^{2}_{j} = 1 ~\text{and}~\lambda^{1}_{j}, \lambda^{2}_{j} \ge 0 ~\forall ~j \in [\omega_{k}],  \nonumber\\
&  \Delta_{-K}((\nabla f^{a^{k}_{j,r}}(x_{k})^{\top} s_{k} + \frac{1}{2}s_{k}^{\top} \nabla^{2} f^{a^{k}_{j,r}}(x_{k})s_{k})_{r \in [m]})-\tau_{1} \le 0 ~\forall ~j \in [\omega_{k}], \label{active_inactive _constraint_1}\\
& \Delta_{-K}((\nabla f^{a^{k}_{j,r}}(x_{k})^{\top} s_{k})_{r \in [m]})-\tau_{1} \le 0~\forall~ j \in [\omega_{k}], \label{active_inactive _constraint_2} \\
&\lambda^{1}_{j}(\Delta_{-K}((\nabla f^{a^{k}_{j,r}}(x_{k})^{\top}s_{k} + \frac{1}{2}s_{k}^{\top} \nabla^{2} f^{a^{k}_{j,r}}(x_{k})^{\top}s_{k})_{r \in [m]})-\tau_{1}) =0 \quad \forall~ j \in [\omega_{k}], \label{vbn_hjhj_jkj}\\
& \lambda^{2}_{j}(\Delta_{-K}((\nabla f^{a^{k}_{j,r}}(x_{k})^{\top}s_{k})_{r \in [m]})-\tau_{1})=0 ~\forall ~j \in [\omega_{k}], \label{nbh_ui_iop}\\
& \frac{\lVert s_{k} \rVert^{2}}{2} \le \frac{\Omega^{2}_{k}}{2}, ~\lambda^{3}\left(\frac{\lVert s_{k} \rVert^{2}}{2}- \frac{\Omega^{2}_{k}}{2}\right)=0, ~\lambda^{3} \ge 0.
\end{align}

Here, we have two cases based on if the trust-region constraints (\ref{active_inactive _constraint_1})  and (\ref{active_inactive _constraint_2}) are active. If active, the trust-region constraints in the Lagrangian function disappears. When inactive, i.e., $\lVert s_{k} \rVert < \Omega_{k}$, the same can be obtained using complementary conditions (\ref{vbn_hjhj_jkj}) and (\ref{nbh_ui_iop}), which reduces the Lagrangian to
\begin{align}\label{ytirc_oyrf}
    L((\tau_{1}, s_{k}),(\lambda^{1}, \lambda^{2})) & = \tau_{1} + \sum_{j=1}^{\omega_{k}} \lambda^{1}_{j}(\Delta_{-K}(\nabla f^{a^{k}_{j}}(x_{k})^{\top}s_{k} + \frac{1}{2} s_{k}^{\top}\nabla^{2} f^{a^{k}_{j}}(x_{k})s_{k})-\tau_{1})\nonumber\\&~~~~~~~+ \sum_{j =1}^{\omega_{k}}\lambda^{2}_{j} (\Delta_{-K}(\nabla  f^{a^{k}_{j}}(x_{k})^{\top}s_{k})- \tau_{1})) 
     = \tau_{1}. 
\end{align}
Next, from (\ref{jghutfvre}) and $ \lambda^{3} =0 $, we obtain, for all $j \in [\omega_{k}]$, that
 \begin{align}
\sum_{j=1}^{\omega_{k}} \lambda^{2}_{j} \Delta_{-K}((\nabla f^{a^{k}_{j,r}}(x_{k})^{\top}s_{k})_{r \in [m]}) =& -\sum_{j=1}^{\omega_{k}} \lambda^{1}_{j} ((\Delta_{-K}((\nabla f^{a^{k}_{j,r}}(x_{k})^{\top}s_{k})_{r\in [m]} \nonumber\\&+(s^{\top}_{k}\nabla^{2} f^{a^{k}_{j,r}}(x_{k})s_{k})_{r \in [m]})\label{hvcye_ugfx}.
 \end{align} 
Then, using (\ref{ytirc_oyrf}), (\ref{hvcye_ugfx}), and (\ref{tewyueui}), we have for all $j\in [\omega_{k}]$ that  
\begin{align}
   2 \tau_{1}= &\sum_{j=1}^{\omega_{k}} \lambda^{1}_{j} (\Delta_{-K}(\nabla f^{a^{k}_{j}}(x_{k})^{\top} s_{k})+ \frac{1}{2}(s_{k}^{\top}\nabla^{2}f^{a^{k}_{j,r}}(x_{k})s_{k})_{r\in [m]}) \nonumber\\
   &- \sum_{j=1}^{\omega_{k}} \lambda^{1}_{j}(\Delta_{-K}(\nabla f^{a^{k}_{j,r}}(x_{k})^{\top}s_{k}+ s_{k}^{\top} \nabla^{2} f^{a^{k}_{j,r}}(x_{k})s_{k})_{r \in [m]}) \nonumber\\
  \le & \sum_{j=1}^{\omega_{k}} \lambda^{1}_{j}\Delta_{-K}\left(\left(-\frac{1}{2} s_{k}^{\top}\nabla^{2}f^{a^{k}_{j,r}}(x_{k})s_{k}\right)_{r \in [m]}\right) \nonumber\\=& \frac{1}{2} \sum_{j=1}^{\omega_{k}} \lambda^{1}_{j} \Delta_{-K}\left(\left(- s_{k}^{\top}\nabla^{2}f^{a^{k}_{j,r}}(x_{k})s_{k}\right)_{r \in [m]}\right). 
\end{align}
Now, from Assumption \ref{boun_term}, we have 
\begin{align}
  & -(s_{k}^{\top} \nabla^{2}f^{a^{k}_{j}}(x_{k})s_{k})_{r \in [m]} \le -M \lVert s_{k} \rVert^{2} \nonumber ~\forall~ j \in [\omega_{k}]\\
 ~\text{and}~ & \Delta_{-K}(-(s_{k}^{\top} \nabla^{2}f^{a^{k}_{j}}(x_{k})s_{k})_{r \in [m]}) \le \Delta_{-K}(-M) \lVert s_{k} \rVert^{2} ~\forall~ j \in [\omega_{k}]\label{weyr_qui},
\end{align}
and using (\ref{weyr_qui}), we obtain 
\begin{align}\label{bbbhhhuuu_hygf}
   2 \tau_{1} \le \frac{1}{2} \sum_{j=1}^{\omega_{k}} \lambda^{1}_{j} \Delta_{-K}(-M) \lVert s_{k} \rVert^{2} =\frac{1}{2} \Delta_{-K}(-M)\lVert s_{k}\rVert^{2} <0. 
    \end{align}
Finally, taking $T = - \Delta_{-K}(-M)$, we have from (\ref{bbbhhhuuu_hygf}) that 
   \begin{align*}
    & -4 \tau_{1} \ge -\frac{1}{2} \Delta_{-K}(-M)\lVert s_{k} \rVert^{2} = \frac{1}{2} T\lVert s_{k}\rVert^{2} \\
   \implies & 4\lvert\tau_{1}\rvert \ge T \lVert s_{k} \rVert^{2}\\
   \implies & \lVert s_{k} \rVert^{2} \le \frac{4}{T}\lvert \tau_{1}\rvert.
   \end{align*}
This proves (\ref{hgjkg_arin}). 
\end{proof}

Next, we present an auxiliary lemma that will be useful for proving global convergence. \begin{lemma}\label{tyuio_pret} 
If $x_{k}+ s_{k} \in \Omega_{k}$, then for each $j \in [\omega_{k}]$, we have
\begin{align*}
    \lvert -\Delta_{-K}(f^{a^{k}_{j}}(x_{k}+s_{k})- 
    f^{a^{k}_{j}}(x_{k}) )+ \Delta_{-K}( m^{a^{k}_{j}}_{k}(s_{k}))\rvert \le \mathcal{K}\lVert s_{k}\rVert^{2} ~\text{for some}~ \mathcal{K} > 0.
\end{align*}
\end{lemma}

\begin{proof} 
For each $j \in [\omega_{k}]$, as $f^{{a}^{k}_{j}}$ is twice continuously differentiable (Assumption \ref{hess_bnd}), we get 
\begin{align}\label{ytehf_1}
    f^{{a}^{k}_{j}}(x_{k}+s_{k}) = f^{{a}^{k}_{j}}(x_{k}) + \nabla f^{{a}^{k}_{j}}(x_{k})^{\top}s_{k} + o(\|s_{k}\|^2). 
\end{align}
Using (\ref{ytehf_1}) and $m^{a^{k}_{j}}$ given in (3.7) of \cite{trust2025setopt}, we obtain
\begin{align}
  & \lvert -\Delta_{-K}(f^{a^{k}_{j}}(x_{k}+s_{k})- 
    f^{a^{k}_{j}}(x_{k}) )+ \Delta_{-K}( m^{a^{k}_{j}}_{k}(s_{k}))\rvert\\
      = & \lvert -\Delta_{-K}( 
  f^{a^{k}_{j}}(x_{k}+s_{k}))-
  f^{a^{k}_{j}}(x_{k}))+ \Delta_{-K}(-m^{a^{k}_{j}}_{k}(0)+  m^{a^{k}_{j}}_{k}(s_{k})) \rvert \nonumber \\
   = ~& \bigg \lvert-\Delta_{-K}(\nabla f^{a^{k}_{j}}(x_{k})^{\top}s_{k}+ o(\|s_{k}\|^{2})+\Delta_{-K}(\nabla f^{a^{k}_{j}}(x_{k})^{\top} s_{k} +  \tfrac{1}{2}s_{k}^{\top}\nabla^{2}f^{{a}^{k}_{j}}(x_{k})s_{k})\bigg\rvert  \nonumber\\
  \le~ &\left \lVert (\nabla f^{a^{k}_{j}}(x_{k})^{\top}s_{k}+\tfrac{1}{2}s_{k}^{\top}\nabla^{2}f^{{a}^{k}_{j}}(x_{k})s_{k}) -(\nabla f^{a^{k}_{j}}(x_{k})^{\top} s_{k} 
   + o( \lVert s_{k} \rVert^{2} )\right\rVert \nonumber\\
   =~& \mathcal{K}\|s_{k}\|^{2}\label{qwtyu_opf} \text{ for some } \mathcal{K} > 0, ~\text{since}~ s_{k} \in \mathcal{B}_{k}.
\end{align}
This concludes the proof.
\end{proof}

Next, using Lemma \ref{WeY_rjt}, we present Corollary \ref{utrw_kkjx}.
\begin{cor}\label{utrw_kkjx} 
Suppose that all the hypotheses of Theorem \ref{fsgcsb}, Lemma \ref{WeY_rjt} and Corollary \ref{tncdfd} hold. Then, there exists a positive constant $\beta$ such that, for all $j \in [\omega_{k}]$,
\begin{align}\label{gghgftfd_gugyfh}
 \Delta_{-K}(m^{{a}^{k}_{j}}_{k}(0)- m^{{a}^{k}_{j}}_{k}\left(s_{k}\right))  \ge \frac{\beta}{2} \frac{ \lvert \theta(x_{k})\rvert}{\lVert s_{k} \rVert}\min\left\{ \frac{\lvert \theta(x_{k}) \rvert}{\lVert s_{k}\rVert \mathcal{K}_{1}}, \Omega_{k} \right\}.     
\end{align}
\end{cor}
\begin{proof} 
From (\ref{ufxvvnn}) of Lemma \ref{WeY_rjt}, we have 
\begin{align}\label{hhjgjhg_b}
 &\Delta_{-K}(f^{a^{k}_{j}}(x_k))^{\top} s_{k})  \le - \lvert \theta(x_{k}) \rvert  ~\forall~ j \in [\omega_{k}] \nonumber\\ 
\implies &-\Delta_{-K}(f^{a^{k}_{j}}(x_k))^{\top} s_{k}) \ge  \lvert \theta(x_{k}) \rvert   ~\forall~ j \in [\omega_{k}]. 
\end{align}
Using Corollary \ref{tncdfd} and (\ref{hhjgjhg_b}), we obtain (\ref{gghgftfd_gugyfh}).
\end{proof}

Next, we show that the sequence of function values generated by the two algorithms, $ F(x_{l(k)})$ for Max-NTRM and $\{(C^{i,r}_{k})_{r \in [m]}\}$ for Avg-NTRM, are monotonically decreasing. 


\begin{theorem}\label{hfrhnye_uyhio}
Let $\{x_{k}\}$ be a sequence generated by Algorithm \ref{alg_max} (Max-NTRM). Then, $\{(f^{i,r}(x_{l^{i,r}(k)}))_{r \in [m]}\}$, for all $i\in [p]$, is a non-increasing sequence, i.e., for all $k$,
\begin{align}\label{gercyuk_ob}
(f^{i,r}(x_{l^{i,r}(k+1)}))_{r \in [m]} \preceq_{K} (f^{i,r}(x_{l^{i,r}(k)}))_{r \in [m]}. 
\end{align}
\end{theorem}

\begin{proof}
 From the definition of $(f^{i,r}(x_{l^{i,r}(k)}))_{r \in [m]}$, we have
\begin{align}\label{jugr_uior}
(f^{i,r}(x_{k}))_{r \in [m]} \preceq_{K} (f^{i,r}(x_{l^{i,r}(k)}))_{r \in [m]}~\text{for all}~ i\in [p].
\end{align}  
For Algorithm \ref{alg_max}, the generated sequence $\{x_{k}\}$ contains only successful iterates.
 For $x_{k}$ to be a successful iterate, we must have, from Step \ref{getwqs_max} of Algorithm \ref{alg_max}, that
\begin{align}\label{pwfhsteui}
-\frac{\Delta_{-K}\left(f^{a^{k}_{j}}(x_{k}+s_{k})-(f^{a^{k}_{j},{r}}(x_{l^{j,r}(k)}))_{r \in [m]}\right)}{\Delta_{-K}(m^{a^{k}_{j}}(0)-m^{a^{k}_{j}}(s_{k}))} \ge \eta_{1} \ge 0~\text{for all}~j \in [\omega_{k}]. 
\end{align}
Applying property (\ref{fyf}) of Lemma \ref{proori} to the inequality (\ref{pwfhsteui}), we obtain for the case when $a^{k}=a^{k'}$ with $k> k'\ge k-N_k$ that
\begin{align}\label{lijk_first_case}
 f^{a^{k}_{j}}(x_{k+1})  \preceq_{K} \max_{0< q< N_k}(f^{a^{k}_{j},{r}}(x_{k-q}))_{r \in [m]} = (f^{a^{k}_{j},r}(x_{l^{j,r}(k)}))_{r \in [m]}
\end{align}
for all $j \in [\omega_{k}]$. For the case when $a^{k}\neq a^{k'}$ for some $k>k'\ge k-N_k$, we have from (\ref{jugr_uior}), for all $j \in [\omega_{k}]$, that 
\begin{align}\label{lijk_second_case}
    f^{a^{k}_{j}}(x_{k+1}) \preceq_{k}   f^{a^{k}_{j}}(x_{k})= (f^{a^{k}_{j},{r}}(x_{l^{j,r}(k)}))_{r \in [m]}. 
\end{align}
Thus, from (\ref{lijk_first_case}), we have 
\begin{align}\label{ghggh_first_case}
F(x_{l(k)}) & = \{ (f^{i,r}(x_{l^{i,r}(k)}))_{r \in [m]}\}_{i \in [p]} \nonumber\\
& \subseteq \{(f^{a^{l^{j,r}(k)}}(x_{l^{j,r}(k)}))_{r \in [m]}\}_{j \in [\omega_k]} + K ~\text{by Proposition 2.1 in \cite{steepmethset}} \nonumber\\
& =\{(f^{a^k_{j}}(x_{l^{j,r}(k)}))_{r \in [m]}\}_{j \in [\omega_k]} + K \nonumber \\
&\overset{(\ref{lijk_first_case})}{\subseteq}
\{(f^{a^{k}_{j},{r}}(x_{k+1}))_{r \in [m]}\}_{j \in [\omega_k]} + K  \nonumber\\&\subseteq \{(f^{i,r}(x_{k+1}))_{r\in [m]}\}_{i\in [p]}+K \nonumber \\
&= F(x_{k+1})+K.
\end{align}
Similarly, from (\ref{lijk_second_case}),
we have 
\begin{align}\label{ghggh_second_case}
 F(x_{l(k)}) = & \{ (f^{i,r}(x_{l^{i,r}(k)}))_{r \in [m]}\}_{i \in [p]} \nonumber\\ 
 \overset{(\ref{jugr_uior})}{\subseteq} & \{ (f^{i,r}(x_{k}))_{r \in [m]}\}_{i \in [p]} + K \nonumber\\
 \subseteq & ~\{ (f^{a^{k}_{j},{r}}(x_{k}) )_{r \in [m]}\}_{j \in [\omega_{k}]} + K~\text{by Proposition 2.1 in \cite{steepmethset}} \nonumber\\
 \subseteq &~\{(f^{a^{k}_{j},{r}}(x_{k+1}))_{r \in [m]}\}_{j \in [\omega_{k}]} + K \nonumber\\
 \subseteq & ~\{(f^{i,r}(x_{k+1}))_{r\in [m]}\}_{i\in [p]}+K \nonumber\\
 =& ~F(x_{k+1})+ K.
 \end{align}
Combining (\ref{ghggh_first_case}) and (\ref{ghggh_second_case}), we have 
\begin{align}\label{chyon_soup}
   (f^{i,r}(x_{k+1}))_{r \in [m]} \preceq_{K} f^{i,r}(x_{l^{i,r}(k)})_{r \in [m]}~\text{for all}~ i \in [p].
\end{align}
Now, using (\ref{chyon_soup}) and $N_{k+1}= \min\{N_k+1,\hat{N}\}$, we have for all $i \in [p]$ that
\begin{align}\label{hghgg_ftdff}
 (f^{i,r}(x_{l^{i,r}(k+1)}))_{r \in [m]}& = \left(\max_{0 \le b \le m(k+1)} f^{i,r}(x_{k+1-b}) \right)_{r \in [m]} \nonumber\\
    & \preceq_{K} \left(\max_{0 \le b \le m(k)+1} f^{i,r}(x_{k+1-b})\right)_{r \in [m]} \nonumber\\
    & =\left (\max_{0 \le b-1 \le m(k)} \{f^{i,r}(x_{k+1})), f^{i,r}(x_{k-(b-1)})\} \right)_{r \in [m]} \nonumber\\ 
    & = \left (\max\{ f^{i,r}(x_{k+1})), \max_{0 \le b \le m(k)} f^{i,r}(x_{k-b})\}\right)_{r \in [m]} \nonumber\\
    & = \left(\max \{ f^{i,r}(x_{k+1})), f^{i,r}(x_{l^{i,r}(k)})\}\right)_{r \in [m]} \nonumber\\
    & \overset{(\ref{chyon_soup})}{\preceq_{K}} (f^{{i,r}}(x_{l^{i,r}(k)}))_{r \in [m]}.
\end{align}
Finally, using (\ref{hghgg_ftdff}), we get 
\begin{align*}
    F(x_{l(k)}) = \{(f^{{i,r}}(x_{l^{i,r}(k)}))_{r \in [m]} \}_{i \in [p]} \subseteq  \{(f^{i,r}(x_{l^{i,r}(k+1)}))_{r \in [m]} \}_{i \in [p]} + K \subseteq F(x_{l(k+1)}) + K.
\end{align*}
Therefore, $ F(x_{l(k+1)}) \preceq^{l}_{K} F(x_{l(k)})$ for all $k$ and  $\{ F(x_{l(k)})\}$ is a non-increasing sequence.
\end{proof}
From Theorem \ref{hfrhnye_uyhio}, Algorithm \ref{alg_max} and Algorithm \ref{alg_avg} generates the sequence $\{x_{k}\}$ in such a way that $ F(x_{l(k+1)}) \preceq^{l}_{K} F(x_{l(k)})$, where $ F(x_{l(k)})$ is defined as
\begin{align}
 F(x_{l(k)}) = \left\{f^{i}(x_{l^{i}(k)})\right\}_{i \in [p]} &= \left\{\left(f^{{i,r}}(x_{l^{i,r}(k)})\right)_{r \in [m]}\right\}_{i \in [p]} \nonumber\\
 & = \left\{\left(f^{{i,1}}(x_{l^{i,1}(k)}),  f^{{i,2}}(x_{l^{i,2}(k)}), \ldots, f^{{i,m}}(x_{l^{i,m}(k)})\right)\right\}_{i \in [p]}.
\end{align} Similarly, the inequality 
$ \{ (C^{i,r}_{k+1})_{r \in [m]}\}_{i \in [p]} \preceq^{l}_{K} \{ (C^{i,r}_{k})_{r \in [m]}\}_{i \in [p]}$ also holds.  Therefore, Algorithm \ref{alg_max} and  Algorithm \ref{alg_avg} are well-defined.
\begin{theorem}\label{hfrhnye_uyhio_avg}
Let $\{x_{k}\}$ be a sequence generated by Algorithm \ref{alg_avg} (Avg-NTRM). Then,  $\{(C^{i,r}_{k})_{r \in [m]}\}$,  for all $i\in [p]$, is a non-increasing sequence, i.e., for all $k$,
\begin{align}\label{truhg_wdd}
(f^{i,r}(x_{k+1}))_{r \in [m]} \preceq_{K}  (C^{i,r}_{k+1})_{r \in [m]}  \preceq_{K}(C^{i,r}_{k})_{r \in [m]}. 
\end{align}
\end{theorem}

\begin{proof}
For Avg-NTRM, $C^{a^{k}_{j}}$ gets updated even for an unsuccessful iterate. Therefore, for Avg-NTRM, we first divide the sequence of iterations into two sets: 
\begin{align}\label{gutpyr_utdr}
    I_{1} := \{k : \rho^{a^{k}_{j}} \ge \eta_{1} ~\text{for all}~ j \in [\omega_{k}] \} ~\text{and}~ I_{2} := \{ k:  \exists l \in [\omega_{k}]~\text{such that}~ \rho^{a^{k}_{l}} < \eta_{1}\}.
\end{align}
$I_{1}$ and $I_{2}$ represent the set of indices of successful and unsuccessful iterations, respectively. Then, for $k \in I_{1}$, from the definition of $\rho^{a^{k}_{j}}_{k}$, Corollary \ref{utrw_kkjx}, and the fact that $\rho^{a^{k}_{j}}_{k} \ge \eta_{1}$ for each $j \in [\omega_{k}]$, we have
\begin{align}\label{hff_uydfv}
&  -\Delta_{-K}(f^{a^{k}_{j}}(x_{k}+s_{k}) - (C^{a^{k}_{j},{r}}_{k})_{r \in [m]}) \ge \eta_{1}\frac{\beta}{2}\frac{\lvert \theta(x_{k})\rvert}{\lVert s_{k} \rVert}\min\left\{ \frac{\lvert \theta(x_{k}) \rvert}{\lVert s_{k}\rVert \mathcal{K}_{1}}, \Omega_{k} \right\} > 0.
\end{align}
Applying property (\ref{fyf}) of Lemma \ref{proori} to (\ref{hff_uydfv}), for the case when $a^{k}=a^{k'} $ for all $k'$ such that $k>k'\ge 0$, we have
\begin{align}\label{cijk_first_case}
(f^{a^{k}_{j},{r}}(x_{k+1}))_{r \in [m]} \prec_{K} \frac{\mu_{k-1}q_{k-1}}{q_{k}}(C^{a^{k}_{j},{r}}_{k-1})_{r \in [m]} + \frac{1}{q_{k}} (f^{a^{k}_{j,r}}(x_{k}))_{r \in [m]} = (C^{a^{k}_{j},{r}}_{k})_{r\in [m]}. 
\end{align}
For the case when $a^{k} \neq a^{k'}$ for any $k'$ such that $k>k'\ge 0$, we have
\begin{align}\label{cijk_second_case}
    (f^{a^{k}_{j},{r}}(x_{k+1}))_{r \in [m]} \prec_{K} (f^{a^{k}_{j
},{r}}(x_{k}))_{r \in [m]} = (C^{a^{k}_{j},{r}}_{k})_{r \in [m]}.
\end{align}
From (\ref{cijk_first_case}) and (\ref{cijk_second_case}), we have  
\begin{align}\label{hjh_uhv_g}
     \{(C^{i,r}_{k})_{r \in [m]} \}_{i \in [p]} 
  & \subset \{(C^{a^{k}_{j},{r}}_{k})_{r \in [m]} \}_{j \in [\omega_{k}]} \nonumber \\
  & \subset   
  \{ (f^{a^{k}_{j},{r}}(x_{k+1}))_{r \in [m]} \}_{i \in [p]} + K \nonumber\\
  & \subset  F(x_{k+1}) + K.
\end{align}
Thus, from (\ref{hjh_uhv_g}), we have for all $i \in [p]$ that
\begin{align}\label{yhhfop_ugv}
   (f^{i,r}(x_{k+1}))_{r \in [m]} \prec_{K} (C^{i,r}_{k})_{r \in [m]}. 
\end{align}
Next, from (\ref{tbddv_uiet}) of $(C^{i,r}_{k})_{r \in [m]}$, for all $i \in [p]$, we can write
\begin{align}\label{gcgncd}
(C^{i,r}_{k})_{r \in [m]}-(C^{i,r}_{k+1})_{r \in [m]} &= (C^{i,r}_{k})_{r \in [m]}-\left(\frac{\mu_{k} q_{k}}{q_{k+1}} (C^{i,r}_{k})_{r \in [m]} + \frac{1}{q_{k+1}} (f^{i,r}(x_{k+1}))_{r \in [m]}\right) \nonumber\\
& = \frac{1}{q_{k+1}}((C^{i,r}_{k})_{r \in [m]}-(f^{i,r}(x_{k+1}))_{r \in [m]}) \in \text{int}K
\end{align}
and 
\begin{align}\label{wqtyrer_uffharin}
(C^{i,r}_{k+1})_{r \in [m]}-(f^{i,r}(x_{k+1}))_{r \in [m]} &= \left(\frac{\mu_{k} q_{k}}{q_{k+1}} (C^{i,r}_{k})_{r \in [m]} + \frac{1}{q_{k+1}} (f^{i,r}(x_{k+1}))_{r \in [m]}\right) \nonumber\\&\qquad-(f^{i,r}(x_{k+1}))_{r \in [m]} \nonumber\\
& = \frac{\mu_{k}q_{k}}{q_{k+1}}((C^{i,r}_{k})_{r \in [m]}-(f^{i,r}(x_{k+1}))_{r \in [m]}) \in \text{int}K.
\end{align}
Combining (\ref{gcgncd}) and (\ref{wqtyrer_uffharin}), we get   
\begin{align}\label{uuvpyt_igfd}
 F(x_{k+1}) = \{(f^{i,r}(x_{k+1}))_{r \in [m]} \} _{i \in [p]} \prec^{l}_{K}  \{(C^{i,r}_{k+1})_{r \in [m]} \}_{i \in [p]} \prec^{l}_{K}\{ (C^{i,r}_{k})_{r \in [m]}\}_{i \in [p]}.
\end{align} 
Also note that if $\mu_{k} = 0$, from Definition \ref{tbddv_uiet} of $(C^{i,r}_{k})_{r \in [m]}$, we have the equality
\begin{align}\label{gutree}
F(x_{k+1}) =\{(f^{i,r}_{k+1})_{r \in [m]}\}_{i \in [p]} = 
\{(C^{i,r}_{k+1})_{r \in [m]}\}_{i\in [p] }.
\end{align}
Therefore, for $\mu_{k} \ge 0$ and $q_{k}\ge 1$, by assembling (\ref{uuvpyt_igfd}) and (\ref{gutree}) and applying property (\ref{htytuttd}) of Lemma \ref{proori}, we get
\begin{align}\label{efbtp}
F(x_{k+1}) =  \{ (f^{i,r}_{k+1})_{r \in [m]})\}_{i \in [p]} \preceq^{l}_{K} \{(C^{i,r}_{k+1})_{r \in [m]}\}_{i \in [p]}\preceq^{l}_{K}  \{(C^{i,r}_{k})_{r \in [m]}\}_{i \in [p]}. 
\end{align}

Next, we consider the case of unsuccessful iteration, i.e., $k \in I_{2}$. Here, since $x_{k+1} = x_{k}$, we have $(f^{a^{k}_{j},{r}}_{k+1})_{r \in [m]} = (f^{a^{k}_{j},{r}}_{k})_{r \in [m]}$, and 
$F_{k+1} =\{(f^{i,r}_{k+1})_{r \in [m]}\}_{i \in [p]} = \{(f^{i,r}_{k})_{r \in [m]}\}_{i \in [p]} = F_{k}$.
Now, to prove that (\ref{efbtp}) holds when $k \in I_{2}$, we have two subcases to consider: $k-1 \in I_{1}$ and $k-1 \in I_{2}$.\\
Case 1: For $k-1 \in I_{1}$, according to (\ref{uuvpyt_igfd}), we have
$ F(x_{k}) = \{(f^{i,r}_{k})_{r \in [m]})\}_{i \in [p]} \preceq^{l}_{K} \{(C^{i,r}_{k})_{r \in [m]}\}_{i \in [p]}$ for all $i \in [p]$. Then, from (\ref{gcgncd}) and using  $ (f^{i,r}_{k+1})_{r \in [m]} = (f^{{i,r}}_{k})_{r \in [m]}$, we have for all $i \in [p]$ that
\begin{align}\label{rtyu_ptdc}
(f^{{i,r}}(x_{k+1}))_{r \in [m]}& =\frac{\mu_{k}q_{k}(f^{i,r}(x_{k+1}))_{r \in [m]}+(f^{{i,r}}(x_{k+1}))_{r \in [m]}}{q_{k+1}} \nonumber\\
& =
\frac{\mu_{k}q_{k}(f^{{i,r}}(x_{k}))_{r \in [m]}+(f^{{i,r}}(x_{k+1}))_{r \in [m]}}{q_{k+1}} \nonumber\\
& \preceq_{K} \frac{\mu_{k}q_{k} (C^{{i,r}}_{k})_{r \in [m]} + (f^{i,r}(x_{k+1}))_{r \in [m]}}{q_{k+1}} \nonumber \\
& =  (C^{{i,r}}_{k+1})_{r \in [m]} \nonumber\\
& = \frac{\mu_{k}q_{k} (C^{{i,r}}_{k})_{r \in [m]} + (f^{i,r}(x_{k+1}))_{r \in [m]}}{q_{k+1}} \nonumber\\
& =  \frac{\mu_{k}q_{k} (C^{{i,r}}_{k})_{r \in [m]} + (f^{i,r}(x_{k}))_{r \in [m]}}{q_{k+1}} \nonumber\\
& \preceq_{K} \frac{\mu_{k}q_{k} (C^{{i,r}}_{k})_{r \in [m]} + (C^{{i,r}}_{k})_{r \in [m]}}{q_{k+1}} = (C^{{i,r}}_{k})_{r \in [m]}.
\end{align}
Therefore, we have $ F(x_{k+1}) =\{(f^{{i,r}}(x_{k+1}))_{r \in [m]}\}_{i \in [p]} \preceq^{l}_{K} 
  \{(C^{i,r}_{k+1})_{r \in [m]}\}_{i \in [p]} \preceq^{l}_{K} \{(C^{i,r}_{k})_{r \in [m]}\}_{i \in [p]}  $.\\
Case 2: For $k-1 \in I_{2}$, suppose that $K = \{j': 1 < j' \le k, k-j' \in I_{1}\}$. If $K = \emptyset$, from the unsuccessful case of Step \ref{getwqs_max} in  Algorithm \ref{alg_avg}, we have for all $i \in [p]$ and for all $p' \in \{0,1, 2, \ldots, k\}$ that 
\begin{align*}
& (f^{{i,r}}(x_{0}))_{r \in [m]} = (f^{i,r}(x_{k-p'}))_{r \in [m]} = (f^{i,r}(x_{k+1}))_{r \in [m]}) \end{align*}
~\text{because}~ $(f^{a^{k}_{j},{r}}(x_{0}))_{r \in [m]} = (f^{a^{k}_{j},{r}}(x_{k-p'}))_{r \in [m]} = (f^{a^{k}_{j},{r}}(x_{k+1}))_{r \in [m]}), ~\text{for all}~ j \in [\omega_{k}].$\\
As a consequence, from the definition (\ref{tbddv_uiet}) of $(C^{i,r}_{k})_{r \in [m]}$, we get for all $i \in [p]$ that
\begin{align}
&(C^{i,r}_{k+1})_{r \in [m]} = (C^{i,r}_{k})_{r \in [m]} = (f^{i,r}_{k+1})_{r \in [m]} \nonumber
\end{align}
~\text{because}~ \begin{align}\label{opuytr_tubv}& (C^{a^{k}_{j},{r}}_{k+1})_{r \in [m]} = (C^{a^{k}_{j},{r}}_{k})_{r \in [m]} = (f^{a^{k}_{j},{r}}_{k+1})_{r \in [m]}  ~\text{for all}~ j \in [\omega_{k}].
\end{align} 
Next, for $K \neq \emptyset$, setting $v := \min\{j':j'\in K\}$, we get for all $i \in [p]$ that
\begin{align}\label{hjuior_yres}
&  (f^{{i,r}}_{k-p'})_{r \in [m]} = (f^{{i,r}}_{k})_{r \in [m]} = (f^{{i,r}}_{k+1})_{r \in [m]} ~\text{for all}~p' \in \{0,1,2, \ldots,v-1\} \nonumber \\
~\text{as}~& (f^{a^{k}_{j},{r}}_{k-p'})_{r \in [m]} = (f^{a^{k}_{j},{r}}_{k})_{r \in [m]} = (f^{a^{k}_{j},{r}}_{k+1})_{r \in [m]} ~\text{for all}~p' \in \{0,1,2, \ldots,v-1\}.
\end{align}
From Definition \ref{tbddv_uiet} of $\{(C^{{i,r}}_{k})_{r\in [m]}\}_{i \in [p]}$, we have for all $i \in [p]$ that  
\begin{align}\label{ureb_uitr}
   &\mu_{k} q_{k} (C^{{i,r}}_{k})_{r \in [m]} + (f^{{i,r}}(x_{k+1}))_{r \in [m]} \nonumber\\
  = ~& \prod_{j=1}^{v-1}\mu_{k-j} q_{k-v+1} (C^{{i,r}}_{k-v+1})_{r \in [m]} + \sum_{i=0}^{v-2} \prod_{j=0}^{i} \mu_{k-j} (f^{{i,r}}_{k-i})_{r \in [m]} + (f^{{i,r}}_{k+1})_{r \in [m]}.
\end{align}
Since  $k-v \in I_{1}$, we must have from (\ref{uuvpyt_igfd}) that $\{(f^{{i,r}}_{k-v+1})_{r \in [m]}\}_{i \in [p]} \preceq^{l}_{K} \{(C^{{i,r}}_{k-v+1})_{r \in [m]}\}_{i \in [p]}$. Next, using (\ref{opuytr_tubv}) and (\ref{ureb_uitr}), we have for all $i \in [p]$ that 
\begin{align}
  &  q_{k+1} (f^{{i,r}}(x_{k+1}))_{r \in [m]}\\ =~& (\prod_{j=0}^{v-1} \mu_{k-j}q_{k-v+1}+ \sum_{i = 0}^{v-2} \prod_{j=0}^{i} \mu_{k-j} + 1) (f^{{i,r}}_{k+1})_{r \in [m]}\nonumber\\
   =~& ( \prod_{j=0}^{v-1} \mu_{k-j} q_{k-v+1}(f^{{i,r}}_{k-v+1})_{r \in [m]} + \sum_{i=0}^{v-2} \prod_{j=0}^{i} \mu_{k-j} (f^{{i,r}}_{k-j})_{r \in [m]}) + (f^{{i,r}}(x_{k+1}))_{r \in [m]}\nonumber \\
   \preceq_{K} ~&~ \mu_{k}q_{k} (C^{{i,r}}_{k})_{r \in [m]} + (f^{{i,r}}(x_{k+1}))_{r \in [m]}\label{wrty_opuy} = q_{k+1}(C^{i,r}_{k+1})_{r \in [m]}.
\end{align}
Hence, we have $\{(f^{{i,r}}(x_{k+1}))_{r \in [m]}\}_{i \in [p]} \preceq^{l}_{K} \{(C^{i,r}_{k+1})\}_{i \in [p]}$, and along the similar lines of (\ref{rtyu_ptdc}), we have 
 \begin{align}\label{hjknm_jhjk}
     \{(f^{{i,r}}(x_{k+1}))_{r \in [m]}\}_{i \in [p]} \preceq^{l}_{K} \{(C^{i,r}_{k+1})\}_{i \in [p]} \preceq^{l}_{K} \{(C^{i,r}_{k+1})\}_{i \in [p]}~\text{for all}~ k \in I_{2}, i \in [p].
 \end{align}
If $\mu_{k} \neq 0$, from (\ref{hjknm_jhjk}), we get 
  \begin{align}\label{hhjhkjknm_jhjk}
     \{(f^{{i,r}}(x_{k+1}))_{r \in [m]}\}_{i \in [p]} \preceq^{l}_{K} \{(C^{i,r}_{k+1})\}_{i \in [p]} \preceq^{l}_{k} \{(C^{i,r}_{k+1})\}_{i \in [p]}~\text{for all}~ k \in I_{2}, i \in [p].
 \end{align}
 If $\mu_{k} = 0$, from  Definition \ref{tbddv_uiet} of $(C^{{i,r}}_{k})_{i \in [p]}$ and $k \in I_{2}$, we have for all $i \in [p]$ that 
\begin{align*}
(f^{{i,r}}_{k+1})_{r \in [m]} = (C^{{i,r}}_{k+1})_{r \in [m]} = (f^{{i,r}}_{k})_{r \in [m]}.  
\end{align*} 
Therefore, from (\ref{hhjhkjknm_jhjk}), we see that $\{(f^{i,r}_{k})\}_{i\in [p]} \preceq^{l}_{K} \{(C^{i,r}_{k})\}_{i \in [p]}$ when $k-1 \in I_{2}$. Thus, (\ref{efbtp}) and (\ref{truhg_wdd}) hold for all $k \in I_{2}$. This completes the proof.
\end{proof}


\begin{cor}\label{sg_1_ag1}
If a non-monotone trust-region step $s_k$ of Algorithm \ref{alg_max} or Algorithm \ref{alg_avg} at $x_k$ is accepted for the objective function ${\widetilde{f}}^{{a}^{k}}_{k} = (f^{{a}^{k}_{1}}_{k}, f^{{a}^{k}_{2}}_{k}, \ldots, f^{{a}^{k}_{\omega_k}}_{k})^\top$ of \eqref{vop_at_x_k}, then $s_k$ is also accepted for the objective map of \eqref{fgcx} at $x_k$.
\end{cor}

\begin{proof}
Using (\ref{ghggh_first_case}) and (\ref{ghggh_second_case}), we get 
\begin{align}\label{max_acc_step}
   F(x_{k+1}) \preceq^{l}_{K}F(x_{l(k)}),
\end{align}
from which the conclusion follows for Algorithm \ref{alg_max}. Similarly, using (\ref{hjh_uhv_g}), we get
\begin{align}\label{avg_acc_step}
   F(x_{k+1}) \preceq^{l}_{K} \{(C^{i,r}_{k})_{r \in [m]} \}_{i \in [p]},
\end{align}
from which the same conclusion follows for Algorithm \ref{alg_avg}.
\end{proof}

The following result ensures that for a sequence $\{x_{k}\}$ generated by Algorithm \ref{alg_max} or Algorithm \ref{alg_avg}, number of unsuccessful steps is always finite and we will definitely see a successful step eventually.

\begin{theorem}\label{rtyue_ouit}
Let $\{x_{k}\}$ be a sequence generated by Algorithm \ref{alg_max} or Algorithm \ref{alg_avg}. Suppose Assumptions \ref{fun_bndbelow}--\ref{label_boundd} hold true, and $\lvert \theta(x_{k}) \rvert > \epsilon $ for some $\epsilon>0$. Then, for any $k$, there is an integer $p \ge 0$ such that $x_{k+p}$ is a successful iterate.
\end{theorem}

\begin{proof} 
We prove the result by the method of  contradiction. Assume that there exists a $k$ such that $x_{k+p}$ is unsuccessful iterate for all $p>0$. Then,  
\begin{align}\label{gjnjkut_pouyt}
\exists~ l \in [\omega_{k}] ~\text{such that}~ \rho^{a^{k+p}_{l}}_{k+p} < \eta_{1}, ~\text{for}~ p = 0, 1, 2, \ldots.
\end{align}
From (3.22) in \cite{trust2025setopt}, we then obtain
\begin{align}\label{fwet_kvg_uin}
    \Delta_{-K}(f^{a^{k+p}_{l}}(x_{k+p}+s_{k+p})-f^{a^{k+p}_{l}}(x_{k+p})) >0.
\end{align}
Thus, from the unsuccessful case in Step \ref{getwqs_max} and Step \ref{trust_region_radius_update_max} of Algorithm \ref{alg_max}, we have
\begin{align}\label{ghddf_uuhvvb}
    & x_{k+p} = x_{p},~\text{for}~ p =0,1,2, \ldots
\end{align}
\text{and} 
\begin{align}\label{jhugg_ohuh} 
    \lim_{p \to \infty}\Omega_{k+p} =  0.
\end{align}
Next, from Lemma \ref{WeY_rjt}, $\lvert \theta(x) \rvert > \epsilon$, and the intermediate step (4.11) in \cite{trust2025setopt} to prove Corollary \ref{utrw_kkjx}, we get
\begin{align}\label{eutyr_uihrg}
    \lVert s_{k} \rVert^{2} &\le \frac{4}{T} \lvert \theta(x_{k}) \rvert,
\end{align}
and for all $j \in [\omega_{k+p}]$,
\begin{align}  
\Delta_{-K}(m^{{a}^{k+p}_{j}}(0)-m^{{a}^{k+p}_{j}}\left(s_{k+p}\right)) & \ge \Delta_{-K}(m^{{a}^{k+p}_{j}}(0))-\Delta_{-K}(m^{{a}^{k+p}_{j}}\left(s_{k+p}\right) \label{ewrt_iop} \\
& \ge \frac{\beta}{2} \frac{\lvert\theta(x_{k+p})\rvert}{\lVert s_{k+p} \rVert}\min\left\{ \frac{\lvert \theta(x_{k+p}) \rvert}{\lVert s_{k+p}\rVert \mathcal{K}_{1}}, \Omega_{k+p} \right\} \nonumber \\
 & \ge \frac{\beta}{4} \frac{\sqrt{T}\lvert \theta(x_{k+p})\rvert}{\lvert \theta(x_{k+p})\lvert^{\frac{1}{2}}} \min \left\{\frac{\sqrt{T}\lvert \theta(x_{k+p})\rvert}{2\lvert \theta(x_{k+p})\rvert^\frac{1}{2} \mathcal{K}_{1}}, \Omega_{k+p} \right\} \nonumber\\
 & \ge \frac{\beta}{4} \sqrt{T} \epsilon^\frac{1}{2} \min\left\{\frac{\sqrt{T}\epsilon^{\frac{1}{2}}}{2 \mathcal{K}_{1}}, \Omega_{k+p}\right\} >0 \label{bvhgjgj_gjygfhj} . 
\end{align}
 Next, we need to prove that 
 \begin{align}
 \lvert \rho^{a^{k+p}_{l}}_{k+p}-1\rvert = &\left\lvert\frac{-\Delta_{-K}(f^{a^{k+p}_{l}}(x_{k+p}+ s_{k+p})-f^{a^{k+p}_{l}}(x_{k+p}))}{\Delta_{-K}(m^{a^{k+p}_{l}}(0)-m^{a^{k+p}_{l}}_{k+p}(s_{k+p}))} -1\right\rvert\nonumber\\
 \le &\frac{\mathcal{K}\lVert s_{k+p} \rVert^2}{\frac{\beta}{4} \sqrt{T}\epsilon^{\frac{1}{2}} \min \left\{ \frac{\sqrt{T} \epsilon^{\frac{1}{2}}}{2 \mathcal{K}_{1}}, \Omega_{k+p}\right\}}\label{hjbf_khv}, \text{for some} ~\mathcal{K} >0.
 \end{align} 
For this, from (\ref{gjnjkut_pouyt}), recall that
\begin{align}\label{uyiear_pi}
    \rho^{a^{k+p}_{l}} < \eta_{1} < 1.
\end{align}
From Lemma \ref{tyuio_pret}, 
(\ref{fwet_kvg_uin}), (\ref{bvhgjgj_gjygfhj}) and (\ref{uyiear_pi}), for all $j \in [\omega_{k+p}]$, we have  
\begin{align}
0 < ~& \lvert \rho^{a^{k+p}_{l}}-1\rvert \\
=~& 1 - \rho^{a^{k+p}_{l}}\nonumber\\
 =~&  1+\frac{\Delta_{-K}(f^{a^{k+p}_{j}}(x_{k+p}+ s_{k+p})-f^{a^{k+p}_{j}}(x_{k+p}))}{\Delta_{-K}(m^{a^{k+p}_{j}}(0)-m^{a^{k+p}_{j}}(s_{k+p}))}  \nonumber\\
\overset{(\ref{ewrt_iop})}{\le} ~& 1 + \frac{\Delta_{-K}(f^{a^{k+p}_{j}}(x_{k+p}+s_{k+p})-f^{a^{k+p}_{j}}(x_{k+p}))}{\Delta_{-K}(0)-\Delta_{-K}(m^{a^{k+p}_{j}}(s_{k+p}))} \nonumber\\
 \le ~& \frac{-\Delta_{-K}(m^{a^{k+p}_{j}}(s_{k+p}))+\Delta_{-K}(f^{a^{k+p}_{j}}(x_{k+p}+s_{k+p})-f^{a^{k+p}_{j}}(x_{k+p}))}{\Delta_{-K}(0)-\Delta_{-K}(m^{a^{k+p}_{j}}(s_{k+p}))} \nonumber\\
\le ~& \frac{ \lvert -\Delta_{-K}(f^{a^{k+p}_{j}}(x_{k+p}+s_{k+p})- 
    f^{a^{k}_{j}}(x_{k+p}) )+ \Delta_{-K}(-m^{a^{k+p}_{j}}(0)+ m^{a^{k+p}_{j}}(s_{k+p}))\rvert}{\Delta_{-K}(0)-\Delta_{-K}(m^{a^{k+p}_{j}}(s_{k+p}))} \nonumber\\
 {\le}~& \frac{\mathcal{K}\lVert s_{k+p} \rVert^2}{\frac{\beta}{4} \sqrt{T}\epsilon^{\frac{1}{2}} \min \left\{ \frac{\sqrt{T} \epsilon^{\frac{1}{2}}}{2 \mathcal{K}_{1}}, \Omega_{k+p}\right\}} \label{tey_op} \text{by Lemma 
   \ref{tyuio_pret}}.
\end{align}
Thus, from (\ref{tey_op}), we have that (\ref{hjbf_khv}) is true. 
For sufficiently large $p$, from (\ref{jhugg_ohuh}) and $\lVert s_{k+p}\rVert \le \Omega_{k+p}$, we have, for all $j \in [\omega_{k+p}]$, that
\begin{align}\label{tyuop_yde}
    \lim_{p \to \infty} \frac{-\Delta_{-K}(f^{a^{k+p}_{j}}(x_{k}+s_{k})-f^{a^{k+p}_{j}}(x_{k+p}))}{\Delta_{-K}(0-m^{a^{k+p}_{j}}(s_{k+p}))} = 1.
\end{align}
Here, considering only Algorithm \ref{alg_max}, we recall the relation (\ref{jugr_uior}) to get 
for all $j \in [\omega_{k+p}]$ that  
\begin{align}\label{teryup_efd}
&(f^{a^{k+p}_{j,r}}(x_{k+p}))_{r \in [m]} \preceq_{K} (f^{a^{k+p}_{j,r}}(x_{l^{r}_{j}(k+p)}))_{r \in [m]}\nonumber\\
  \implies   &  (f^{a^{k+p}_{j,r}}(x_{k+p}))_{r \in [m]} - (f^{a^{k+p}_{j,r}}(x_{l^{r}_{j}(k+p)}))_{r \in [m]} \in -K \nonumber\\
  \implies & \Delta_{-K}((f^{a^{k+p}_{j,r}}(x_{k+p}))_{r \in [m]} - (f^{a^{k+p}_{j,r}}(x_{l^{r}_{j}(k+p)}))_{r \in [m]}) \le 0 \nonumber \\
\overset{\text{Lemma \ref{proori} (\ref{fyf})}}{\implies} & \Delta_{-K}((f^{a^{k+p}_{j}}(x_{k+p}+s_{k+p})-(f^{a^{k+p}_{j,r}}(x_{l^{r}_{j}(k+p)}))_{r \in [m]}) \nonumber\\&-(f^{a^{k+p}_{j}}(x_{k+p}+s_{k+p})-(f^{a^{k+p}_{j,r}}(x_{k+p}))_{r \in [m]}))  \le 0 \nonumber\\
  \overset{\text{Lemma \ref{proori} (\ref{tewyueui})}}{\implies} & \Delta_{-K}(f^{a^{k+p}_{j}}(x_{k+p}+s_{k+p})-(f^{a^{k+p}_{j,r}}(x_{l^{r}_{j}(k+p)}))_{r \in [m]}) \nonumber\\  ~~~~~~\le~~&\Delta_{-K}((f^{a^{k+p}_{j}}(x_{k+p}+s_{k+p})- (f^{a^{k+p}_{j,r}}(x_{k+p}))_{r \in [m]})) . 
  \end{align}
From (\ref{teryup_efd}), we can conclude, for all $j \in [\omega_{k+p}]$, that
\begin{align}\label{ggyu_uhugyu}
&\frac{-\Delta_{-K}(f^{a^{k+p}_{j}}(x_{k+p}+s_{k+p})-(f^{a^{k+p}_{j,r}}(x_{k+p}))_{r \in [m]})}{\Delta_{-K}(0-m^{a^{k+p}_{j}}(s_{k+p}))}\nonumber \\
 \le & \frac{-(\Delta_{-K}(f^{a^{k+p}_{j}}(x_{k+p}+s_{k+p})-(f^{a^{k+p}_{j,r}}(x_{l^{r}_{j}(k+p)}))_{r \in [m]}))}{\Delta_{-K}(0-m^{a^{k+p}_{j}}(s_{k+p}))}.
\end{align}

Therefore, according to (\ref{tyuop_yde}), (\ref{ggyu_uhugyu}) and $\eta_{1}\in (0,1)$, when $p$ is sufficiently large, we have ${\rho^{{a^{k+p}_{j}}}} \ge \eta_{1}$ that contradicts (\ref{gjnjkut_pouyt}). This concludes that for any $k$, there is a non-negative $p >0$ such that $x_{k+p}$ is a successful iterate. 

Next, considering Algorithm \ref{alg_avg}, we recall the relation (\ref{truhg_wdd}) to similarly have
\begin{align}\label{huyu_uyty}
  (f^{a^{k+p}_{j,r}}(x_{k+p}))_{r \in [m]}   \prec_{K}(C^{a^{k+p}_{j,r}}_{k+p})_{r \in [m]} ~\text{for all}~ j \in [\omega_{k+p}].  
\end{align}
Then, it follows from (\ref{huyu_uyty}) that for all $j \in [\omega_{k+p}]$,
\begin{align}\label{hgfcbf_yuio}
  & ~\Delta_{-K}(f^{a^{k+p}_{j,r}}(x_{k+p}+s_{k+p})-(C^{a^{k+p}_{j,r}}_{k+p})_{r \in [m]})\nonumber\\ \le &~\Delta_{-K}(f^{a^{k+p}_{j}}(x_{k+p}+s_{k+p})- (f^{a^{k+p}_{j,r}}(x_{k+p}))_{r \in [m]}). 
\end{align}
From (\ref{hgfcbf_yuio}), we have for all $j \in [\omega_{k+p}]$ that
\begin{align}\label{tqwyup_outv}
  & \frac{-\Delta_{-K}(f^{a^{k+p}_{j,r}}(x_{k+p}+s_{k+p})-(f^{a^{k+p}_{j,r}}(x_{k+p}))_{r \in [m]})}{\Delta_{-K}(0-m^{a^{k+p}_{j,r}}(s_{k+p}))} \nonumber\\ \le& \frac{-\Delta_{-K}(f^{a^{k}_{j}}(x_{k+p}+s_{k+p})-(C^{a^{k+p}_{j,r}}_{k+p})_{r \in [m]})}{\Delta_{-K}(0-m^{a^{k+p}_{j}}(s_{k+p}))}.  
\end{align}
Therefore, according to (\ref{hgfcbf_yuio}), (\ref{tqwyup_outv}) and $\eta_{1}\in (0,1)$, when $p$ is sufficiently large we have $\rho^{a^{k+p}_{j}} \ge \eta_{1}$. This again contradicts  
(\ref{gjnjkut_pouyt}) and we conclude that, for any $k$, there is a non-negative $p >0$ such that $x_{k+p}$ is a successful iterate.
\end{proof}

Next, we show that the sequences ${ F(x_{l(k)})}$ and $\{ (C^{i,r}_{k})_{r \in [m]}\}_{i \in [p]}$ are convergent, and that the sequence $\{x_{k}\}$ generated by Algorithm~\ref{alg_max} or Algorithm~\ref{alg_avg} is contained in $\mathcal{L}_{0}$ and hence, by Assumption~\ref{label_boundd}, $\{x_{k}\}$ generated by these algorithms is bounded. 

\begin{lemma}\label{admits_limit_max_avg}
Let $\{x_{k}\}$ be a sequence generated by Algorithm \ref{alg_max} or Algorithm \ref{alg_avg}. Suppose that Assumption \ref{label_boundd} holds. Then, $((f^{i,r}(x_{l^{i,r}(k)}))_{r \in [m]})$ or $((C^{{i,r}}_{k}))_{r \in [m]}$ are convergent.    \end{lemma}
\begin{proof}
From Assumption \ref{label_boundd}, note that the level set $\mathcal{L}_{0}$ of $F$ is bounded. Then, first, we consider Algorithm \ref{alg_max} and prove by mathematical induction,  that if $x_{p'} \in \mathcal{L}_{0}$, then $x_{p'+1}\in \mathcal{L}_{0}$ for all $p' = 1, 2, \ldots, k$. For this, by Definition \ref{ghvt_ugh_huh}, we have $ \{f^{{i,r}}(x_{l^{i,r}(0)})_{r \in [m]}\}_{i \in [p]} = F_{0}= \{(f^{{i,r}}(x_{0}))_{r \in [m]}\}_{i \in [p]}$, for all $i \in [p]$. Additionally, since $\{x_{k}\}$ is a successful iterate, (\ref{chyon_soup}) also holds. Then, using these and (\ref{gercyuk_ob}) of Theorem \ref{hfrhnye_uyhio}, we have for all $i \in [{p}]$ and for all $k$ that

\begin{align}\label{hvhhhhhhhhhhhv_gvv}
(f^{{i,r}}(x_{k+1}))_{r \in [m]} \preceq_{K} \left(f^{{i,r}}(x_{l^{i,r}(k+1)})\right)_{r \in [m]} \preceq_{K} \left(f^{{i,r}}(x_{l^{i,r}(k)})\right)_{r \in [m]}\preceq_{K} f^{i}(x_{0}), 
\end{align}
which can be further expressed as 
\begin{align}\label{gjgpuhu_uy}
    F_{k+1} =~& F(x_{k+1}) = \{ (f^{{i,r}}(x_{k+1})_{r \in [m]}\}_{i \in [p]} \preceq^{l}_{K} \{(f^{i,r}(x_{l^{i,r}(k+1)}))_{r \in [m]} \}_{i \in [p]} \preceq^{l}_{K} ((f^{i,r}({x_{0}})_{r \in [m]})_{i \in [p]} \nonumber\\= ~&F(x_{0}). 
\end{align}
This shows that the sequence $\{x_{k} \}$ is contained in $\mathcal{L}_{0}$. Now, from Theorem \ref{hfrhnye_uyhio} and (\ref{gjgpuhu_uy}), we have that $ F(x_{l(k)})$, i.e., $\{ (f^{i,r}(x_{l^{i,r}(k+1)}))_{r \in [m]} \}_{i \in [p]}$ is a non-increasing sequence and bounded. Therefore, $\{ F(x_{l(k)})\}$ is convergent.

Next, we consider Algorithm \ref{alg_avg}. Similar to above, by mathematical induction, we prove that if $x_{p'}\in {\mathcal L}_{0}$, then $x_{p'+1}\in \mathcal{L}_{0}$ for all $p'=1,2, \ldots, k$. For this, by Definition \ref{tbddv_uiet}, we have 
$\{(C^{i,r}_{0})_{r \in [m]}\}_{i \in [p]} =\{(f^{i,r}(x_{0}))_{r \in m}\}_{i \in [p]} = F_{0}$. Additionally, since $\{ x_{k}\}$ is a successful iterate, (\ref{yhhfop_ugv}) holds. Then, using these and (\ref{truhg_wdd}) of Theorem \ref{hfrhnye_uyhio}, we have
\begin{align}\label{hggh_ojh}
  F_{k+1} =& \{ (f^{i,r}(x_{k+1}))_{r \in [m]} \}_{i \in [p]} \preceq_{K} \{(C^{i,r}_{k+1})_{r \in [m]} \}_{i \in [p]} \preceq_{K} \{ (C^{i,r}_{k})_{r \in [m]} \}_{i \in [p]} \preceq_{K} \{ f^{i}(x_{0})\}_{i \in [p]}.
\end{align}
Here again, the sequence $\{x_{k} \}$ is contained in $\mathcal{L}_{0}$. Since, from Theorem \ref{hfrhnye_uyhio_avg} and (\ref{hggh_ojh}),  the sequence $\{(C^{i,r}_{k})_{r \in [m]}\}_{i \in [p]}$ is non-increasing and bounded, we conclude that $\{(C^{i,r}_{k})_{r \in [m]}\}_{i \in [p]}$ is convergent.
\end{proof}

Finally, in Theorem \ref{gjhg_uyhugy}, we show that the sequence $\{x_{k}\}$ generated by the two algorithms converges to a $K$-critical point of \eqref{fgcx}.

\begin{theorem}\label{gjhg_uyhugy}
Let $\{x_{k}\}$ be a sequence of regular iterative points for $F$ generated by Algorithm \ref{alg_max} or Algorithm \ref{alg_avg}, and let any limit point of $\{x_{k}\}$ be a regular point of $F$. Suppose that Assumptions \ref{hess_bnd}-\ref{grad_bnd} hold. Then, we have 
\begin{enumerate}[(i)]
    \item $ \liminf_{k \to \infty} \lvert \theta(x_{k}) \rvert = 0$\label{ygiu_iie}, and 
    \item every limit point of $\{x_{k}\}$ is a $K$-critical point of \eqref{fgcx}. 
\end{enumerate}
\end{theorem}

\begin{proof}
First, we consider Algorithm \ref{alg_max}. We have two possibilities: (a) a finite number of successful iterations followed by unsuccessful iterations for large $k$, or (b) an infinite number of successful iterations. In the first case, suppose that $k_{0}$ is the index of the last successful iteration. At this point, if $\lvert \theta(x_{k_{0}+1}) \rvert$ is still greater than $\epsilon$, then, from Theorem \ref{rtyue_ouit}, we can find an index larger than $k_{0}$ that is successful. This is in direct contradiction with the statement of the first case; thus,  case (a) never holds. For case (b), we prove (\ref{ygiu_iie}) by the method of contradiction. Assume that there exists a constant $\epsilon > 0$ and $\mathcal{M} \subseteq \{0,1,2, \ldots\}$ such that for all $k \in \mathcal{M}$, 
\begin{align}\label{tyur_yubfv}
     \lvert \theta(x_{k})\rvert > \epsilon.
\end{align}
First, we need to prove that 
\begin{align}\label{gthitv_obthe}
    \lim\limits_{k \to \infty} \Omega_{k} = 0.
\end{align}
For this, we consider the set of successful iterations $I_{1}$ and the set of unsuccessful iterations $I_{2}$ from (\ref{gutpyr_utdr}). First, we show that (\ref{gthitv_obthe}) holds for $k \in I_{1}$, i.e., 
\begin{align}\label{hff_ioog}
    \lim\limits_{k \to \infty, ~k \in I_{1}} \Omega_{k} =0.
\end{align}
From Lemma \ref{rtyue_ouit}, it is clear that $I_{1}$ is an infinite set. Then, from Lemma \ref{WeY_rjt}, Lemma \ref{utrw_kkjx} and (\ref{tyur_yubfv}), we have
\begin{align}\label{fgio_piytf}
    \Delta_{-K}(m^{a^{k}_{j}}(0)-m^{a^{k}_{j}}(s_{k})) \ge \beta \frac{\sqrt{T}\epsilon^\frac{1}{2}}{4 \mathcal{K}_{1}} \min \left\{ \frac{\sqrt{T}\epsilon^\frac{1}{2}}{2\mathcal{K}_{1}}, \Omega_{k}\right\} ~\forall~j \in [\omega_{k}],
\end{align}
and, from the definition of $\rho^{a^{k}_{j}}_{k}$, (\ref{fgio_piytf}) and Lemma \ref{proori} (\ref{tewyueui}), we have
\begin{align}\label{hgfrt_uyiih}
 &\Delta_{-K} (f^{a^{k}_{j,r}}(x_{l^{r}_{j}(k)}))_{r \in [m]})-\Delta_{-K}(f^{a^{k}_{j}}(x_{k}+s_{k})) \ge \eta_{1}  \beta \frac{\sqrt{T}\epsilon^\frac{1}{2}}{4 \mathcal{K}_{1}} \min \left\{ \frac{\sqrt{T}\epsilon^\frac{1}{2}}{2 \mathcal{K}_{1}}, \Omega_{k}\right\}\\
\implies & \Delta_{-K}(f^{a^{k}_{j}}(x_{k}+s_{k})) \le \Delta_{-K}(f^{a^{k}_{j,r}}(x_{l^{r}_{j}(k)}))_{r\in [m]}) - \eta_{1} \beta \frac{\sqrt{T}\epsilon^\frac{1}{2}}{4 \mathcal{K}_{1}} \min \left\{ \frac{\sqrt{T}\epsilon^\frac{1}{2}}{2 \mathcal{K}_{1}}, \Omega_{k}\right\}.  
\end{align}
Then, applying $\Delta_{-K}$ to the inequality
(\ref{chyon_soup}) and taking limit as $k \to \infty$ and $k \in I_{1}$ on both sides of (\ref{hgfrt_uyiih}), we obtain (\ref{gthitv_obthe}) from Lemma \ref{admits_limit_max_avg}. 

Next, for $k \in I_{2}$, if it is a finite set, it can be shown from (\ref{hff_ioog}) that (\ref{gthitv_obthe}) holds. If 
$I_{2}$ is an infinite set, we define
$\bar{K}= \{d_{k}: k = 1,2, \ldots\}$ as a subset of $I_{2}$, where
\begin{align*}
    & d_{1} = \min \{ d': d' \in I_{2}\},
\end{align*}   
~\text{and}~ 
\begin{align*}
 & d_{k+1} = \min \{d' \in I_{2}:d'-1 \in I_{1}, d'-1 > d_{k}  \}, ~ k \ge 1.
\end{align*}
For $k \ge 1$, according to the definition of $d_{k+1}$, we have $d_{k+1}-1\in I_{1}$. As per Case 3 of  Step \ref{trust_region_radius_update_max} of Algorithm \ref{alg_max}, we then have
\begin{align*}
   \Omega_{d_{k}} \le \gamma_{2} \Omega_{d_{k}-1}. 
\end{align*}
Next, according to the definition of $d_{k+1}$, an integer $h$ can be found which satisfies
\begin{align}\label{yhy_op}
    d_{k} + h < d_{k+1}-1~\text{and}~d_{k}+h \in I_{2}.
\end{align}
Then, let an integer $h_{k}$ be the maximum that satisfies (\ref{yhy_op}). Thus, as per Case 3 of Step \ref{trust_region_radius_update_max} of Algorithm \ref{alg_max}, we have 
\begin{align*}
\Omega_{d_{k}+h_{k}+1} \le \Omega_{d_{k}+h} \le \Omega_{d_{k}} \le \gamma_2 \Omega_{d_{k}-1}, ~~h = 0,1, \ldots,h_{k},
\end{align*}
and consequently 
\begin{align}\label{qwur_uioop}
\Omega_{d_{k}+h_{k}+1} \le \Omega_{d_{k}+h} \le \gamma_{2} \Omega_{d_{k}-1}, ~~h =0,1,\ldots, h_{k}.  
\end{align}
Finally, since $d_{k}+h_{k}+1$ and $d_{k}-1$ lie in $I_{1}$, from (\ref{hff_ioog}) and (\ref{qwur_uioop}) we get that 
\begin{align}\label{eutr_hgjk}
    \lim_{k \to \infty, ~k \in I_{2}} \Omega_{k} =0.
\end{align}
Thus, combining (\ref{hff_ioog}) and (\ref{eutr_hgjk}), we get (\ref{gthitv_obthe}).

Next, for all $j \in [\omega_{k}]$, we have to show that 
\begin{align}\label{typ_yue}
 \lvert \rho^{a^{k}_{j}}_{k}-1\rvert = \left\lvert \frac{-\Delta_{-K}(f^{a^{k}_{j}}(x_{k}+ s_{k})-f^{a^{k}_{j}}(x_{k}))}{\Delta_{-K}(m^{a^{k}_{j}}_{k}(0)-m^{a^{k}_{j}}_{k}(s_{k}))} -1\right\rvert   \le~& \frac{\mathcal{K}\lVert s_{k} \rVert^2}{\frac{\beta}{2} \frac{\sqrt{T}\epsilon^\frac{1}{2}}{2} \min \left\{ \frac{\sqrt{T}\epsilon^\frac{1}{2}}{2 \mathcal{K}_{1}}, \Omega_{k}\right\}}.
\end{align}
Here, two possible cases arise:
\begin{enumerate}[{Case} 1.]
\item \label{case_max_sg_1} 
$\lvert \rho^{a^{k}_{j}}_{k}-1\rvert = \rho^{a^{k}_{j}}_{k}-1 $ 
\item \label{case_max_sg_2}$
\lvert \rho^{a^{k}_{j}}_{k}-1\rvert = 1-\rho^{a^{k}_{j}}_{k} $ 
\end{enumerate}
Without loss of generality, some of the $j$'s in [$\omega_{k}$] are included in Case \ref{case_max_sg_1} and the rest are included in Case \ref{case_max_sg_2}.
In Case \ref{case_max_sg_1}, for all such $j$'s where  
$\rho^{a^{k}_{j}}_{k} >1$ hold, we have from Lemma \ref{tyuio_pret} and (\ref{fgio_piytf}) that 
\begin{align} \label{gvgu_opt_urg}
&\rho^{a^{k}_{j}}_{k}-1 \\
=~ & \frac{-\Delta_{-K}(f^{a^{k}_{j}}(x_{k}+ s_{k})-f^{a^{k}_{j}}(x_{k}))}{\Delta_{-K}(m^{a^{k}_{j}}(0)-m^{a^{k}_{j}}(s_{k}))} -1 \nonumber\\
 \le  ~&  \frac{-\Delta_{-K}(f^{a^{k}_{j}}(x_{k}+ s_{k})-f^{a^{k}_{j}}(x_{k}))-\Delta_{-K}((m^{a^{k}_{j}}_{k}(0)-m^{a^{k}_{j}}_{k}(s_{k}))}{\Delta_{-K}(m^{a^{k}_{j}}(0)-m^{a^{k}_{j}}(s_{k}))} \nonumber\\
  \overset{(\ref{ytehf_1})}{=}  ~&
  \frac{-\Delta_{-K}(\nabla f^{{a}^{k}_{j}}(x_{k})^{\top}s_{k} + \|s_{k}\|^2 \mathcal{K})-\Delta_{-K}(-\nabla f^{a^{k}_{j}}(x_{k})^\top s_{k}-\frac{1}{2} s^\top_{k} \nabla^{2} f^{a^{k}_{j}}(x_{k})s_{k})}{\Delta_{-K}(m^{a^{k}_{j}}(0)-m^{a^{k}_{j}}(s_{k}))} \nonumber\\ 
{\le}~& \frac{\Delta_{-K}(\frac{1}{2} s^\top_{k} \nabla^{2} f^{a^{k}_{j}}(x_{k})s_{k}-\lVert s_{k}\rVert^{2}\mathcal{K})}{\Delta_{-K}(m^{a^{k}_{j}}(0)-m^{a^{k}_{j}}(s_{k}))} ~\text{by Lemma~\ref{proori}(\ref{tewyueui})}~\nonumber\\
{\le}~& \frac{1}{2}\frac{\left\lVert 
 s^\top_{k} \nabla^{2} f^{a^{k}_{j}}(x_{k})s_{k}-\lVert s_{k}\rVert^{2}\mathcal{K}\right\rVert}{\Delta_{-K}(m^{a^{k}_{j}}(0)-m^{a^{k}_{j}}(s_{k}))} ~\text{by Lemma ~\ref{proori}(\ref{aeruin})}~\nonumber\\
 \overset{(\ref{qwtyu_opf})}{\le} &\frac{\mathcal{K}(\lVert s_{k} \rVert)^2}{\frac{\beta}{2} \frac{\sqrt{T}\epsilon^\frac{1}{2}}{2} \min \left\{ \frac{\sqrt{T}\epsilon^\frac{1}{2}}{2\mathcal{K}_{1}}, \Omega_{k}\right\}}.
  \end{align}
Also, in Case \ref{case_max_sg_2}, for all such $j$'s where $\rho^{a^{k}_{j}}_{k}< 1$ holds, we have similarly from Lemma \ref{tyuio_pret} and (\ref{fgio_piytf}) that
\begin{align}\label{tybjc_uoo}
    & 1-\rho^{a^{k}_{j}}_{k} \nonumber\\
  = ~&1 + \frac{\Delta_{-K}( f^{a^{k}_{j}}(x_{k}+ s_{k})-f^{a^{k}_{j}}(x_{k})))}{\Delta_{-K}(m^{a^{k}_{j}}(0)-m^{a^{k}_{j}}(s_{k}))}& \nonumber\\
\overset{(\ref{fgio_piytf})}{\le}~ & 1 + \frac{\Delta_{-K}(f^{a^{k}_{j}}(x_{k}+ s_{k})-f^{a^{k}_{j}}(x_{k})))}{\Delta_{-K}(m^{a^{k}_{j}}(0))-\Delta_{-K}(m^{a^{k}_{j}}(s_{k}))} \nonumber\\
 =~&  \frac{ \Delta_{-K}(0)-\Delta_{-K}(m^{a^{k}_{j}}_{k}(s_{k})) + \Delta_{-K}(f^{a^{k}_{j}}(x_{k}+ s_{k})-f^{a^{k}_{j}}(x_{k}))}{\Delta_{-K}(m^{a^{k}_{j}}(0))-\Delta_{-K}(m^{a^{k}_{j}}(s_{k}))} \nonumber\\  \le ~& \frac{\lvert \Delta_{-K}(m^{a^{k}_{j}}_{k}(s_{k})) - \Delta_{-K}(f^{a^{k}_{j}}(x_{k}+ s_{k})-f^{a^{k}_{j}}(x_{k})) \rvert}{\Delta_{-K}(m^{a^{k}_{j}}(0))-\Delta_{-K}(m^{a^{k}_{j}}(s_{k}))} \nonumber\\
 \overset{(\ref{qwtyu_opf})}{\le} &\frac{{\mathcal{K}}\lVert s_{k} \rVert^2}{\frac{\beta}{2} \frac{\sqrt{T}\epsilon^\frac{1}{2}}{2} \min \left\{ \frac{\sqrt{T} \epsilon^{\frac{1}{2}}}{2 \mathcal{K}_{1}}, \Omega_{k}\right\}}.
  \end{align}
From (\ref{gvgu_opt_urg}) and (\ref{tybjc_uoo}), we observe that (\ref{typ_yue}) is true.
Then, from (\ref{gthitv_obthe}), (\ref{typ_yue}) and $\lVert s_{k} \rVert \le \Omega_{k}$, we have for all $j \in [\omega_{k}]$ that 
\begin{align*}
  \lim_{k \to \infty} \frac{-\Delta_{-K}(f^{a^{k}_{j}}(x_{k}+ s_{k})-f^{a^{k}_{j}}(x_{k}))}{\Delta_{-K}(m^{a^{k}_{j}}(0)-m^{a^{k}_{j}}(s_{k}))}  = 1.   
\end{align*}
Further, from the definition of $\rho^{a^{k}_{j}}_{k}$ of Algorithm \ref{alg_max} and (\ref{jugr_uior}), we get, for all $j$, that
\begin{align*}
   &\frac{-\Delta_{-K}(f^{a^{k}_{j}}(x_{k}+s_{k})-(f^{a^{k}_{j,r}}(x_{k}))_{r \in [m]})}{\Delta_{-K}(0-m^{a^{k}_{j,r}}(s_{k}))} \nonumber\\
\le~&\frac{-\Delta_{-K}(f^{a^{k}_{j}}(x_{k}+s_{k})-(f^{a^{k}_{j,r}}(x_{l^{r}_{j}(k)})_{r \in [m]})}{\Delta_{-K}(0-m^{a^{k}_{j}}(s_{k}))}.   
\end{align*}
Therefore, for sufficiently large $k$, we have 
\begin{align*}
\rho^{a^{k}_{j}}_{k} \ge \eta_{1}, j \in [\omega_{k}], 
\end{align*}
which shows that $\Omega_{k+1} \ge \Omega_{k}$ for sufficiently large $k$. This leads to a contradiction to (\ref{gthitv_obthe}). This shows that the assumption (\ref{tyur_yubfv}) is false and  (\ref{ygiu_iie}) is proved.

Next, for the second part of this theorem, from (\ref{ygiu_iie}), we have 
\begin{align}\label{limit_inf_theta}
\lim_{k \to \infty}\inf~\lvert\theta(x_{k})\rvert = 0.
\end{align}
Since the sequence $\{x_{k}\}$ belongs to $\mathcal{L}_{0}$ that is bounded, $\{x_{k}\}$ is also bounded and has an accumulation point $x^{\ast}$, with $\{x_{k}\}_{k\in \mathcal{M}}$, a subsequence of $\{x_{k}\}$, converging to $x^\ast$.
From part (b) of Theorem \ref{critopti}, $\theta$ is a continuous function on the set of regular points for \eqref{fgcx}. Thus, from (\ref{limit_inf_theta}), we have  
$\theta(x^{\ast})=0$. Hence, 
$x^{\ast}$ is a $K$-critical point for \eqref{fgcx}. 

For Algorithm \ref{alg_avg}, we can prove the same by following a similar approach.    
\end{proof}

\section{Numerical Experiment}\label{Numerical_Experiment}

In this section, we study the performance of our proposed Max-NTRM and Avg-NTRM by comparing with the TRM for SOP \cite{trust2025setopt} and two other methods for SOP: Steepest Descent (SD) \cite{steepmethset} and Conjugate Gradient  (CG) \cite{kumar2024nonlinear} methods. For this, we consider $22$ test SOPs (as defined in Appendix \ref{appendix_ind}) and employ the performance profiling proposed by Dolan and Mor{\'e} \cite{dolan2002benchmarking},
for four performance metrics: average number of non-convergences, average number of iterations, average CPU time, and reciprocal of average step size.

\subsection{Experiment setup}
Similar to \cite{ghalavand2023two}, we constrain our SOPs to the box $\mathcal S$, of the form $[x_{L}, x_{U}]$, by enforcing the next iterate $x_{k}+s_{k}$ to remain within $\mathcal S$, i.e., $x_{L} \le  x_{k}+s_{k} \le x_{U}$, and choose the initial point from $\mathcal S$. Therefore, instead of (3.11), we consider the following modified subproblem: 
\begin{equation}\label{ghtyu_issue_sub}
\left.\begin{aligned} 
&\min  & & t\\
   & \text{subject to} & &\Delta_{-{K}}\left(\nabla f^{{a}^{k}_j}(x_{k})^{\top}s+ \tfrac{1}{2}s^{\top}\nabla^{2}f^{{a}^{k}_j}(x_{k})s\right)-t\le 0, ~j = 1,2, \ldots, {\omega}_{k},\\
 & & &\Delta_{-{K}}\left(\nabla f^{{a}^{k}_{j}}(x_{k})^{\top}s\right)-t \le 0, ~j =1, 2, \ldots, {\omega}_{k},\\
 & & & \lVert s \rVert \le \Omega_{k}, \\
 & & & x_{L}-x_{k}\le s \le x_{U}-x_{k}.
 \end{aligned}
 \right\}
 \end{equation}

For experimentation, we use MATLAB on an Apple M2 system with $8$ CPU cores and $8$ GB of RAM. The parameters associated with the algorithms are given in Table \ref{tab2}, where $\rho$, $\sigma$ are for CG; $\beta_\text{SD}$ and $\nu$ are for SD; and $\Omega_{0}$, $\Omega_\text{max}$, $\eta_{1}$, $\eta_{2}$, $\gamma_{1}$, and $\gamma_{2}$ are for TRM, Max-NTRM and Avg-NTRM. Additionally, $M$ is for Max-NTRM and $\mu_{k}$, $\mu_\text{max}$, and $\mu_\text{min}$ are for Avg-NTRM, with $\mu_{k}=0.5$ fixed across all iterations $k$. For all algorithms, we use the cone $K$ given in the table below and set $\epsilon$ as the tolerance for the stopping criterion, also given in the table below. For TRM, Max-NTRM, Avg-NTRM, the stopping criterion is given by $\lvert t_{k} \rvert < \epsilon$, where $t_{k}$ is the solution of the subproblem (3.11), and for CG and SD, it is given by $\Vert v_{k} \rVert < \epsilon$, where $v_{k}$ is the steepest descent direction obtained as the solution of (4) in \cite{kumar2024nonlinear}. For CG, we use the basic conjugate descent rule $\beta^{\text{CG}}_{k}= 0.99(1-\sigma)\beta^{\text{CD}}_{k}$, where $\beta^{\text{CD}}_{k}$ is given by (49) in \cite{kumar2024nonlinear}.

\begin{table}[htbp]
  \centering
  \resizebox{\textwidth}{!}{\begin{minipage}{\textwidth}
        \caption{Parameters of different algorithms used in all the experiments}
        \label{tab2}
        \begin{tabular}{c c c c c c c c c c c c c c c c c c}
          \hline
          $\rho $ & $\sigma$ & $\beta_\text{SD}$ & $\nu$& $\Omega_0$ & $\Omega_\text{max}$ & $\epsilon$ & $\eta_1$ & $\eta_2$ & $\gamma_1$ & $\gamma_2$ & $K$ & $\bar{M}$ & $\mu_{\min}$ & $\mu_{\max}$ & $\mu_{k}$\\
          \hline
          $0.0001$ & $0.1$ & $0.0001$ &$0.5$& $1$ & $20$ & $10^{-3}$ & $0.001$ & $0.75$ & $0.4$ & $0.9$ & $\mathbb{R}^m_+$ & $10$ & $0$ & $1$ & $0.5$ \\
          \hline
        \end{tabular}
      \end{minipage}}
\end{table}

\begin{figure}[htbp]
\centering
\subfigure[Max-NTRM with cone $K_1$]{%
    \includegraphics[width=0.45\textwidth]{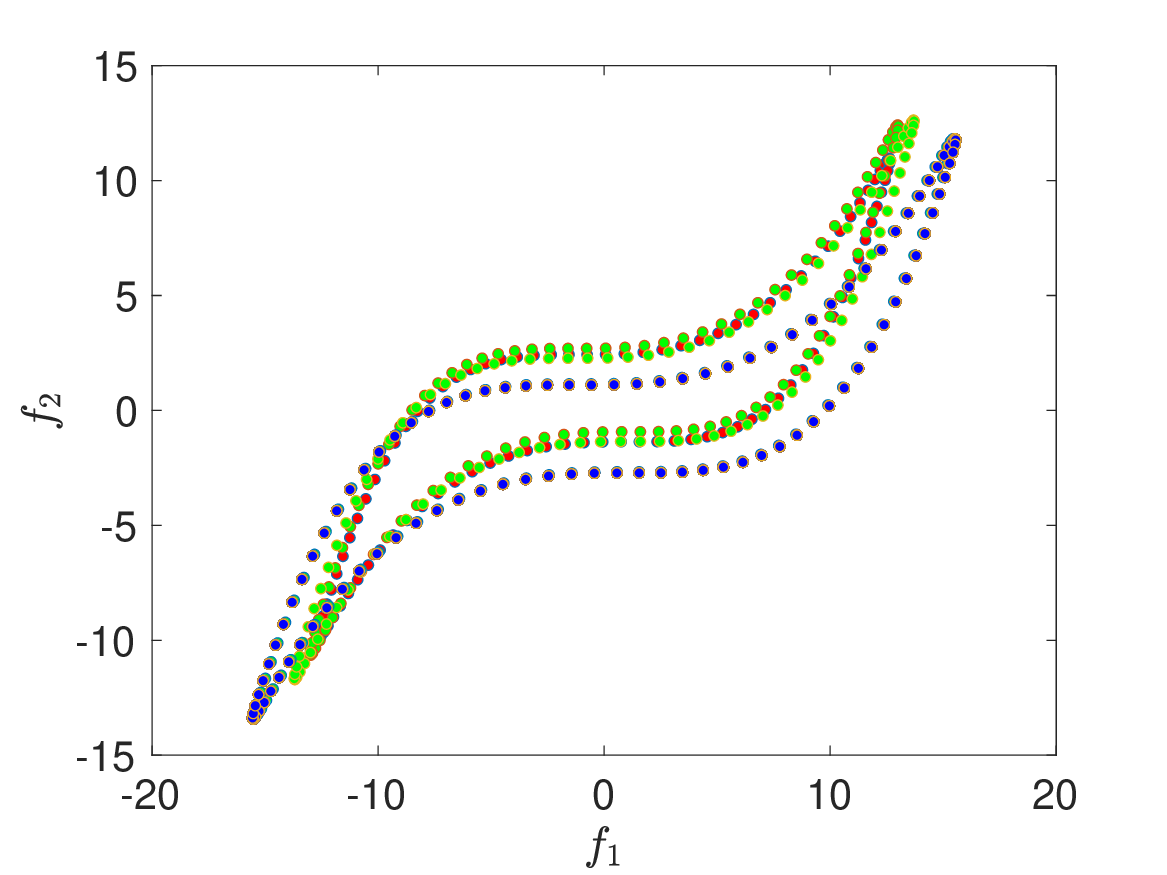}%
    \label{objective_space_max_k1}%
}
\hfill
\subfigure[Max-NTRM with cone $K_2$]{%
    \includegraphics[width=0.45\textwidth]{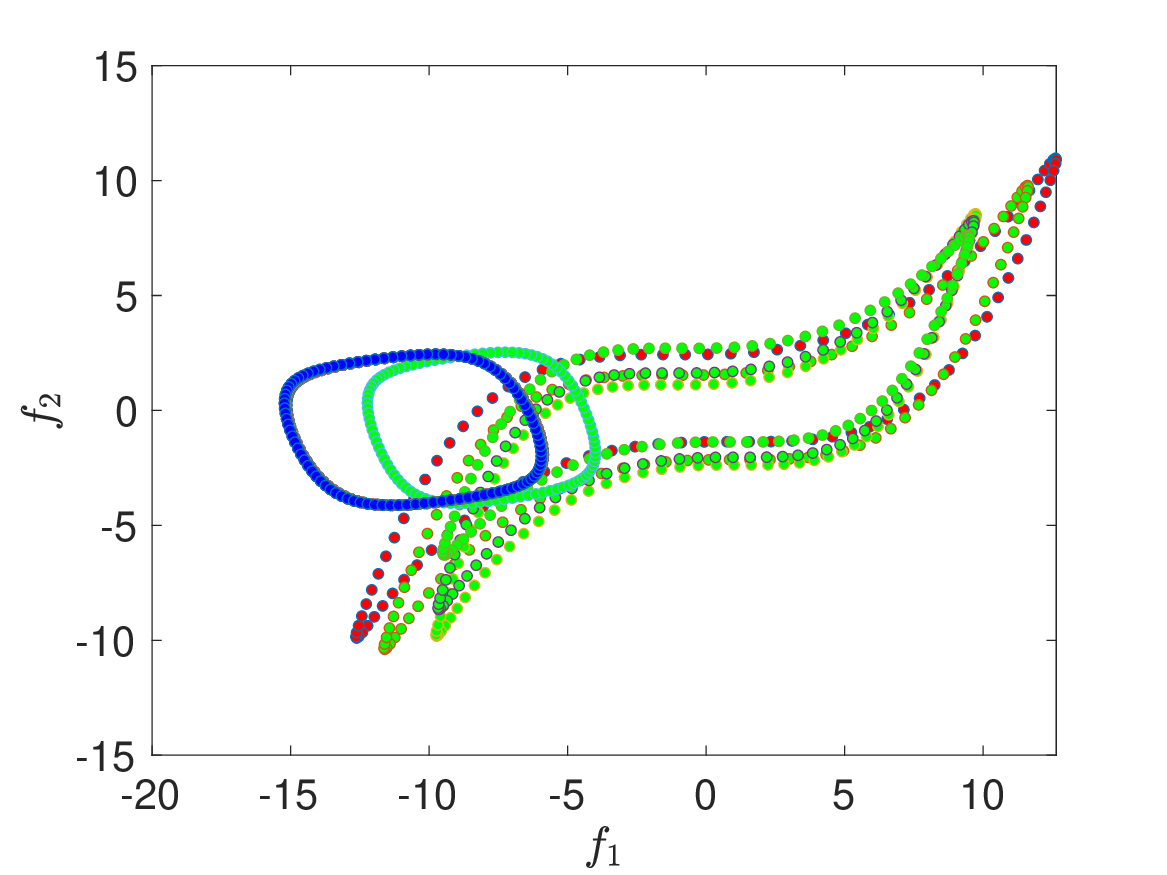}%
\label{objective_space_max_k2}%
}
\vspace{0.5cm} 
\subfigure[Avg-NTRM with cone $K_1$]{%
    \includegraphics[width=0.45\textwidth]{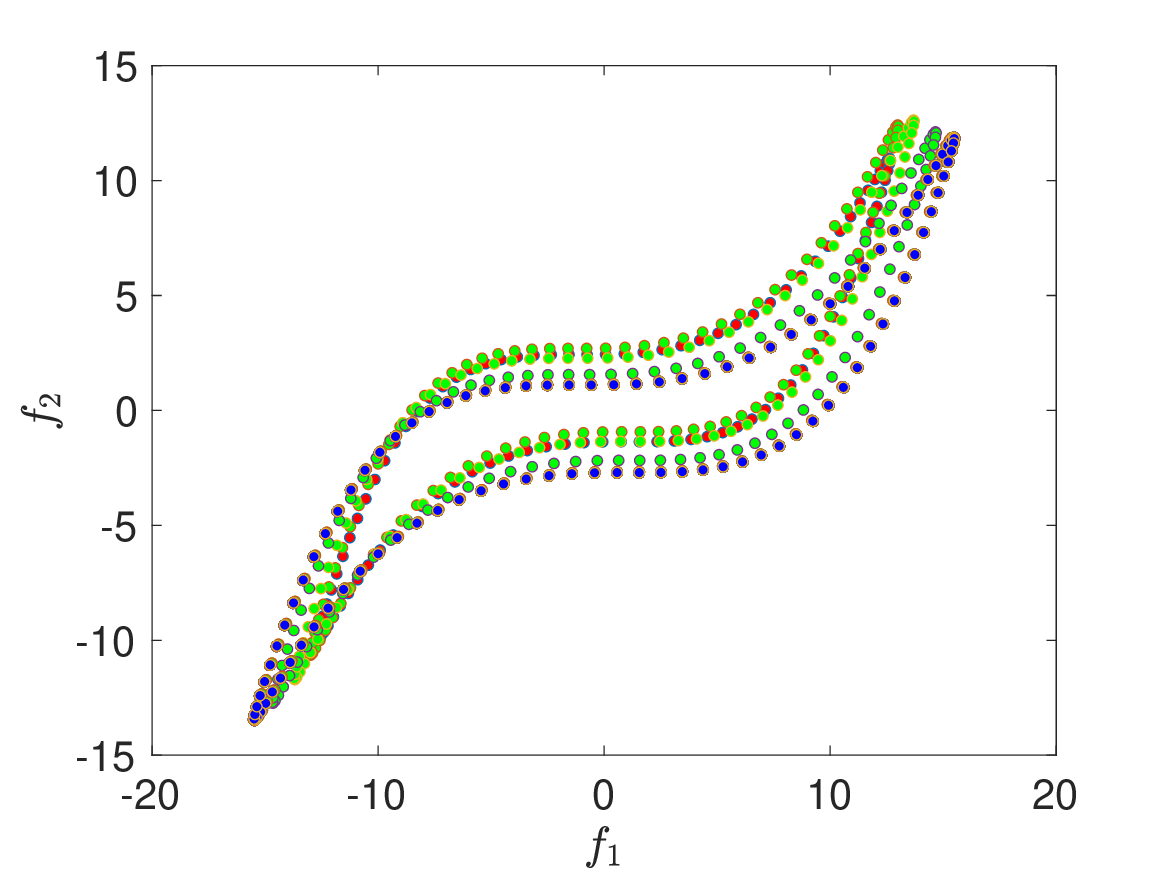}%
    \label{objective_space_avg_k1}%
}
\hfill
\subfigure[Avg-NTRM with cone $K_2$]{%
    \includegraphics[width=0.45\textwidth]{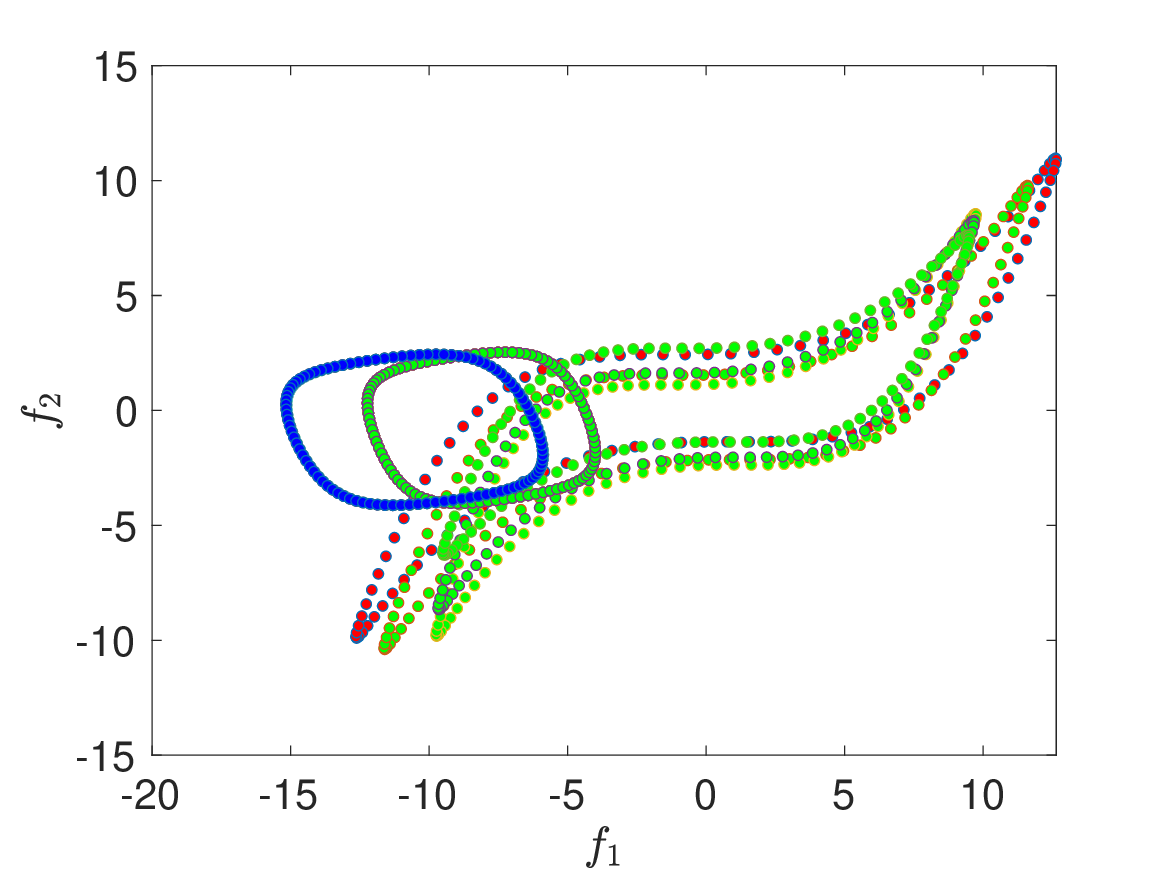}%
    \label{objective_space_avg_k2}%
}
\caption{Iterative points in the objective space generated by Max-NTRM (top row) and Avg-NTRM (bottom row) with cones $K_{1}$ and $K_{2}$ for Problem \ref{Example_5.3_stee} starting from the initial point $x_{0}=(-16.355461, -2.454201)$. The red, green and blue images represent the $F$-values at the initial, intermediate, and final points, respectively.}
\label{cone_changes_intermediate_images}
\end{figure} 

First, we explore whether changing the ordering cone $K$ changes the behavior of Max-NTRM and Avg-NTRM. For this, we consider Problem \ref{Example_5.3_stee} with two different cones $K_{1}:= \mathbb{R}^{2}_{+}$ and $K_{2} := \{(y_{1}, y_{2}) \in \mathbb{R}^{2}_{+} : y_{1} \ge 3 y_{2}, 3 y_{1} \le y_{2} \}$ and run the two algorithms starting from the initial point $x_{0} = (-16.355461, -2.454201)$, as shown in Figure \ref{cone_changes_intermediate_images}. From Figures \ref{objective_space_max_k1} and \ref{objective_space_avg_k1}, it seems that the algorithms converge under the cone $K_{1}$, but they actually do not.  In these two figures, the blue images corresponds to the $F$-value at the points obtained at the maximum number of iterations, not at a convergent point. In contrast, NTRMs converge under $K_{2}$ and in Figures \ref{objective_space_max_k2} and \ref{objective_space_avg_k2}, the  blue images represent the $F$-values at the convergent point. 
Comparing Figures \ref{objective_space_max_k1} and \ref{objective_space_max_k2}, it is clear that changing the cone from $K_{1}$ to $K_{2}$ alters the behavior of Max-NTRM, leading to different iterative points (intermediate points are shown in green and the final point is shown in blue). A similar pattern is observed for Avg-NTRM when comparing Figures \ref{objective_space_avg_k1} and \ref{objective_space_avg_k2}. Therefore, the choice of cone plays a crucial role in the convergence behavior of NTRMs.

Next, as per the performance profiling in \cite{dolan2002benchmarking}, given a set of solvers $S$ (TRM, Max-NTRM, Avg-NTRM, SD, CG) and a set of problems $P$ ($22$ problems in this paper), the performance ratio $r_{p,s}$ of algorithm $s \in S$ in solving problem $p \in P$ is defined as $r_{p,s}: = \frac{t_{p,s}}{\min\{t_{p,s}:s \in S\}}$, 
where $t_{p,s}$ is the metric of interest (e.g., number of iterations). For problem $p$, the best performing algorithm $s$ has $r_{p,s}=1$ and for the rest $r_{p,s}>1$. Then, the performance profile $\rho_{s}$ for algorithm $s$ is given as the cumulative distribution function of $r_{p,s}$, i.e.,
\begin{align*}
  \rho_{s}(\tau) : = \frac{1}{\lvert P \rvert}\left\lvert \{p \in P: r_{p,s} \le \tau \}\right\rvert.  
\end{align*}
Note that $\rho_{s}(\tau)$ is the probability that $r_{p,s}$ of algorithm $s$ lies in $[1, \tau]$. The value $\rho_{s}(1)$ gives the probability that algorithm $s$ performs the best. For example, $\rho_{s}(1)=0.7$ means algorithm $s$ performs the best for 70 percent of the problems. However, for a complete judgment, we also look at $\tau > 1$ where better performing algorithms should have a higher $\rho_{s}(\tau)$. \\ 

To compute $t_{p,s}$, we first sample $100$ initial points for each problem $p\in P$ and then run all the algorithms in $\mathcal{S}$ starting from the same initial points, for a maximum of $it_{\max}$ ($=100$) iterations. We call the initial point for which an algorithm converges within $it_{\max}$ as \emph{convergent} for that algorithm, and otherwise, it is called \emph{non-convergent}. Based on this, we compute $t_{p,s}$ under the scenario of considering convergent points that are common to all algorithms. For each problem $p$ and algorithm $s$, $t_{p,s}$ corresponding to different metrics obtained from our experiments for the scenario are given in Table \ref{uipuy_ytf}.


\begin{figure}[htbp]
\centering
\subfigure[No. of non-convergence]{%
    \includegraphics[width=0.45\textwidth]{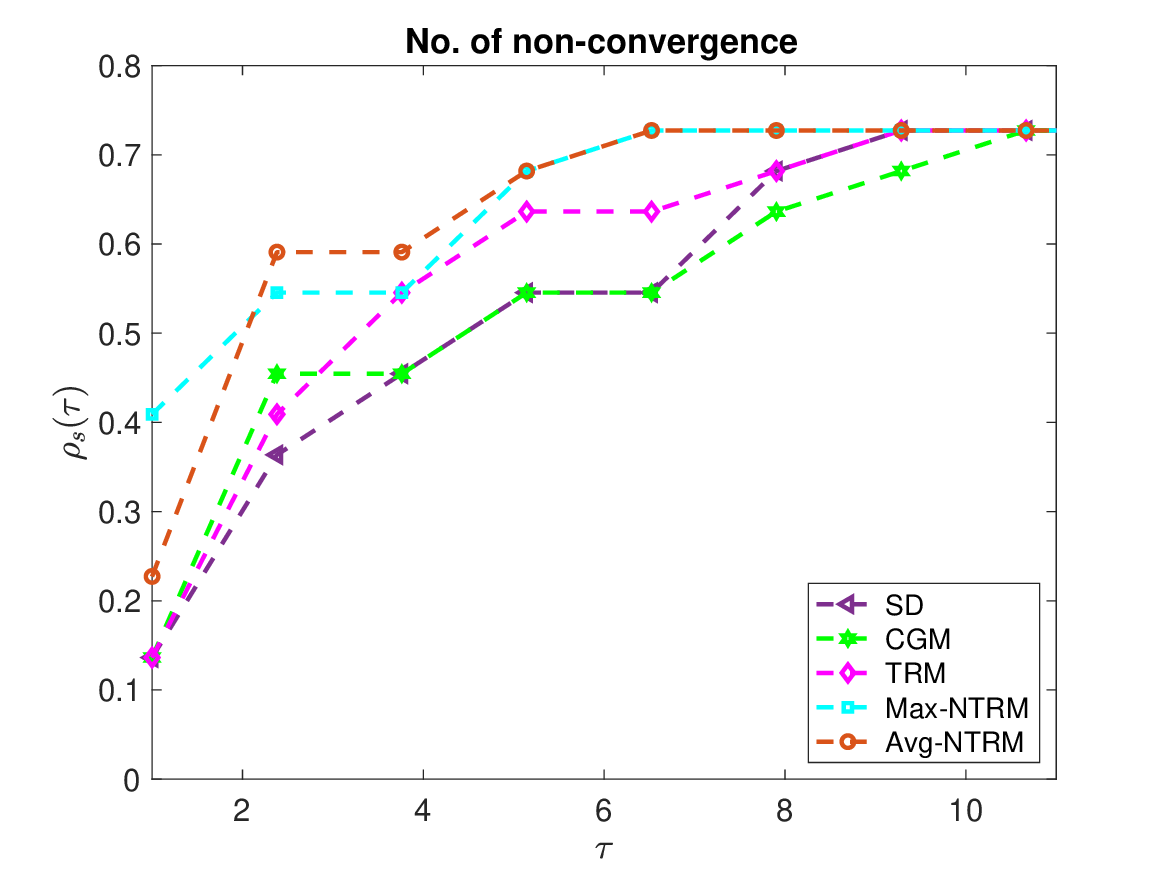}%
    \label{per_prof_no_of_failure}%
}
\hfill
\subfigure[No of iterations]{%
    \includegraphics[width=0.45\textwidth]{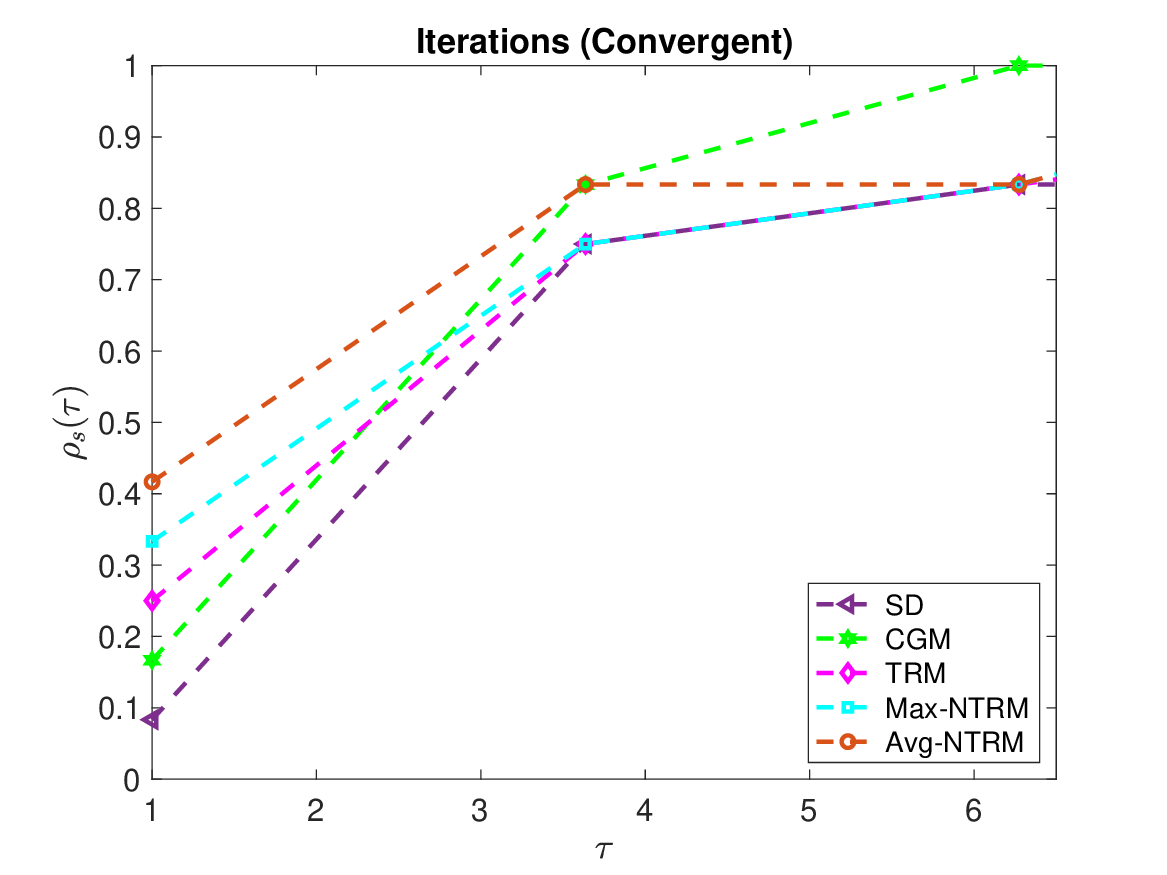}%
    \label{Iterations_Common_conv}%
}

\vspace{0.5cm} 

\subfigure[CPU time (sec)]{%
    \includegraphics[width=0.45\textwidth]{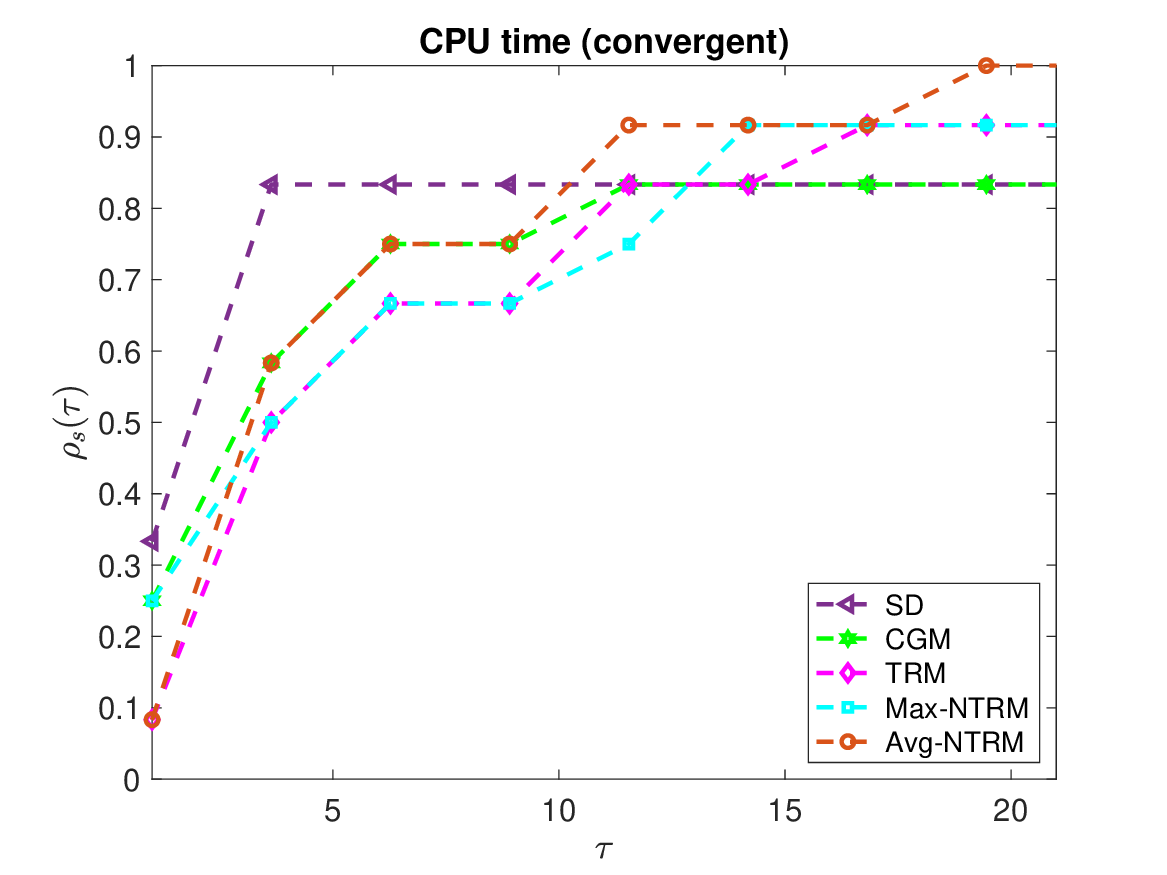}%
    \label{CPU_time_comm_conv_point}%
}
\hfill
\subfigure[Reciprocal of Step size]{%
    \includegraphics[width=0.45\textwidth]{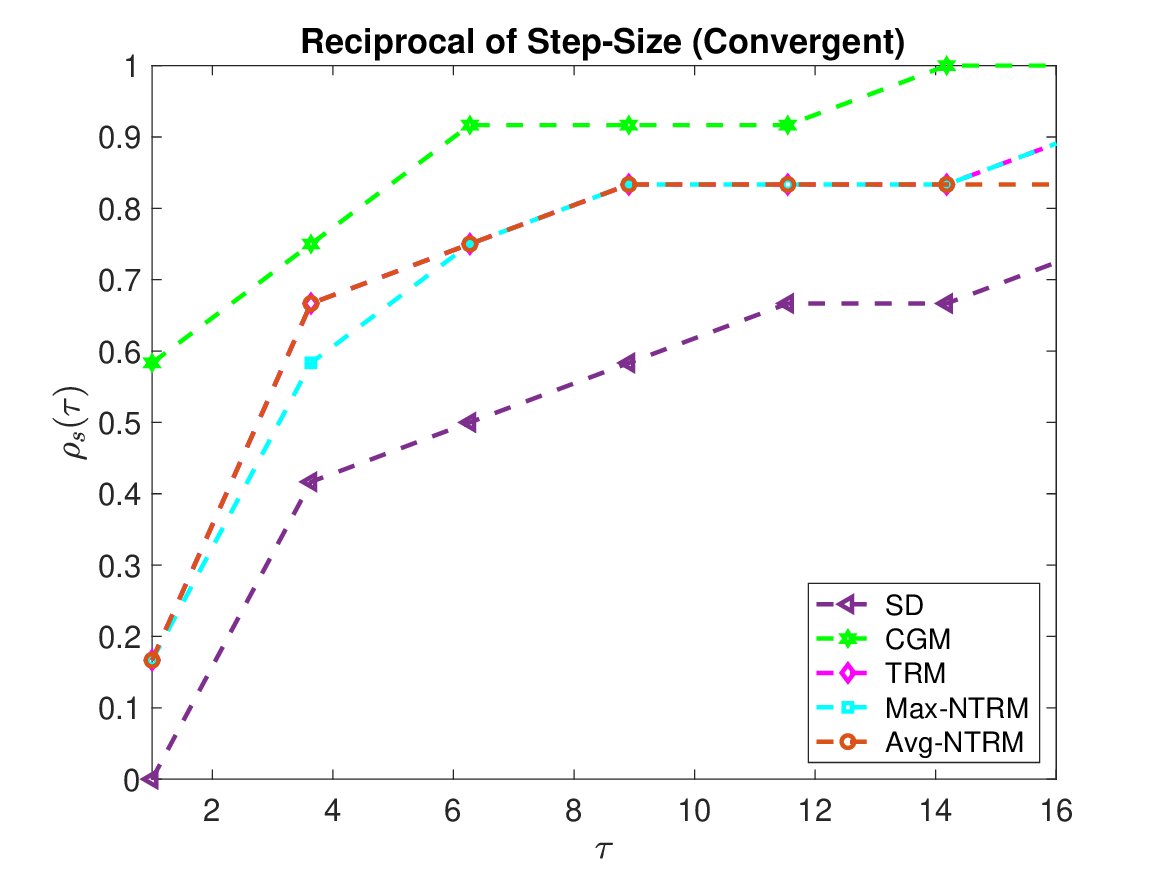}%
    \label{Step_sixe_avg}%
}
\caption{Performance profile of SD, CG, TRM, Max-NTRM and Avg-NTRM using four metrics: (a) no. of non-convergent points, (b) no. of iterations to converge, (c) CPU time taken to converge, and (d) reciprocal of average step-size taken every iteration. For these plots, we consider only the case of a common convergent point, i.e., the metrics are reported over only those initial points for which all three algorithms converged.}
\label{all_points}
\end{figure}

\subsection{Results}
First, Figure \ref{per_prof_no_of_failure} shows the performance profile of the five algorithms in terms of the probability of non-convergence (failure to converge) within $it_{\max}$ iterations. Within the three trust-region-based methods, Max-NTRM and Avg-NTRM outperform TRM. Specifically, Max-NTRM and Avg-NTRM show similar performance across most $\tau$ values, except in the intervals $[1, 2.3]$ and $[2.3, 5]$, where Max-NTRM and Avg-NTRM are, respectively, better. Compared with line-search methods, NTRMs consistently outperform SD and CG. 

Next, we analyze the speed of convergence of five algorithms using three metrics: number of iterations, CPU time, and reciprocal of average step-size. To ensure a fair comparison of speed of convergence, we exclude the effect of non-convergence and consider the scenario that takes into account the common convergent points of all algorithms.

In this scenario, for $10$ out of $22$ problems, neither CG nor SD converged for any initial point out of $100$. Therefore, this analysis is limited to the remaining $12$ problems. In Figure \ref{Iterations_Common_conv}, among the three TRMs, Avg-NTRM converges the fastest, outperforming both Max-NTRM and TRM. Between Max-NTRM and TRM, Max-NTRM is superior when $\tau$ is approximately in the range $[1, 3.63]$ (approx.); beyond this point, both methods require a similar number of iterations to converge. 
Upon comparing all algorithms, Avg-NTRM outperforms the others, especially the two line search methods, CG and SD, for $\tau$ value in $[1, 3.63]$ (approx.); beyond this point, CG performs better.

In terms of CPU time, as depicted in Figure \ref{CPU_time_comm_conv_point}, Avg-NTRM and Max-NTRM are more time-efficient. Notably, for $\tau$ in the range $[9, 11.5]$, TRM slightly outperforms Max-NTRM. Comparing Avg-NTRM and Max-NTRM, Max-NTRM is preferable for $\tau$ values up to approximately $3.63$, but beyond this, Avg-NTRM consistently outperforms Max-NTRM. When contrasting the three trust-region-based methods with the two line-search methods, SD requires the least amount of time for $\tau$ values approximately up to $9$. This efficiency is primarily because trust-region-based methods involve computationally intensive Hessian computations, e.g., in Step \ref{step3_choice_of_a_max} (selection of an $a^{k}$ from $P_{k}$) and Step \ref{yjuu_max} (step computation), whereas SD does not have such overhead.  For higher $\tau$ values, Avg-NTRM outperforms in time efficiency.

\begin{landscape}
\begin{center}
\begin{table}[] 
 \resizebox{1.5\textwidth}{!}{\begin{minipage}{\textwidth}
 \caption{Performance comparison of SD, CG, TRM, Max-NTRM, Avg-NTRM for $22$ SOPs based on (a) \textbf{No. of non-convergences}: 
 no of initial points for which an algorithm failed to converge within $100$ iterations, (b) \textbf{Iterations} (Convergent): average number of iterations taken by an algorithm, (c) \textbf{CPU Time} (Convergent): average time taken (in seconds) by an algorithm, (d) 
\textbf{Step Size} (Convergent): average step size taken by an algorithm, calculated over only common convergent points.
For each metric, the best-performing algorithm's value is marked in bold. Each problem, defined on its domain $S$, is identified by its name, dimension $n$, and $m$ in argument and image spaces, respectively. $\uparrow$ means higher values are better, and $\downarrow$ means lower values are better.$`-$' means that we did not get even a single initial point for which all five algorithms converged.}
\label{uipuy_ytf}
  \begin{tabular}{c c|c|c c c c c|c c c c c|c c c c c |c c c c c}
    \hline
   \multicolumn{3}{c|}{\textbf{Type of SOP}}   &
      \multicolumn{5}{c|}{\textbf{No. of non-convergences $\downarrow$}} &
      \multicolumn{5}{c|}{\textbf{CPU Time} (Convergent) $\downarrow$} &
      \multicolumn{5}{c|}{\textbf{Iterations} (Convergent) $\downarrow$} & \multicolumn{5}{c}{\textbf{Step Size} (Convergent) $\uparrow$}\\ 
      \hline
 Name& $n,m$ & Domain ($S$) & \textbf{SD} & \textbf{CG} & \textbf{TRM} & \textbf{Max} & \textbf{Avg}& \textbf{SD} & \textbf{CG} & \textbf{TRM} & \textbf{Max} & \textbf{Avg}&\textbf{SD} & \textbf{CG} & \textbf{TRM} &\textbf{Max} & \textbf{Avg}  &\textbf{SD} & \textbf{CG} & \textbf{TRM} & \textbf{Max} & \textbf{Avg} \\
  \hline  
  Problem \ref{test_inst_ZDT1_22}&$2,2$& $ [0,1]^{n}$&$73$ & $100$ & ${15}$ & $\bf{10}$  & $12$  & $-$ & $-$ & $-$ & $-$ & $-$ & $-$ & $-$ & $-$ & $-$ & $-$ &$-$ & $-$ & $-$ & $-$ & $-$ \\
  
 & $5,2 $ &  &$98$ & $100$ & ${2}$ & $\bf{0}$ & $7$&$-$& $-$ & $-$ & $-$ & $-$ &$-$ & $-$ & $-$ & $-$&$-$ & $-$ & $-$ & $-$ & $-$  &$-$\\

 &$8,2$ & & $100$ &$100$ & ${12}$ & $\bf{0}$ & $15$& $-$& $-$ & $-$ & $-$& $-$&$-$ & $-$ & $-$ & $-$ & $-$ &$-$ & $-$ & $-$ &$-$ & $-$\\
  
 &$10,2$ &  &$100$ & $100$ & ${17}$ &$\bf{0}$& $ 9$& $-$ & $-$ & $-$& $-$ & $-$ & $-$ & $-$ & $-$ & $-$& $-$ & $-$ & $-$ & $-$ & $-$ & $-$\\
 \hline
  
 Problem \ref{test_inst_ZDT4}  &$10,2$& $[0.01,1]\times$ &$100$ & $100$ & $\bf{0}$ &$\bf{0}$ & $\bf{0}$ & $-$ & $-$ & $-$ & $-$ & $-$ & $-$ & $-$ & $-$ & $-$ & $-$ & $-$ & $-$ & $-$ & $-$ & $-$\\
  && $ [-5,5]^{n-1}$ & & & & &  &  &  &  & &&&  &  &&  & \\
 
 \hline
 Problem 
\ref{test_inst_DTLZ1_n100} & $6, 4$ &$[0,1]^{n}$& $94$ & $100$ & ${33}$ & $\bf{14}$& $\bf{14}$&$-$& $-$& $-$& $-$ & $-$& $-$ & $-$ & $-$ & $-$& $-$ & $-$ & $-$ & $-$ & $-$ & $-$\\
  \hline
  Problem 
 \ref{test_inst_DTLZ3_54} & $5,4$ & $[0,1]^{n}$ &$100$ & $100$ & $\bf{55}$ &$62$ & $ $63$ $ & $-$ & $-$ & $-$ & $-$ &$-$& $-$ &$ -$ & $-$ &$-$ &$-$&$-$ & $-$ & $-$ & $-$ & $-$ \\
 \hline 
  Problem
\ref{test_inst_DTLZ5} & $3,3$ & $[0,1]^{n}$&$27$ & ${17}$ & $49$  & $\bf{16}$ & $32$& $\bf{4.51}$  & $7.45$ & $66.61$ & $63.45$ & $40.27$& ${5.76}$ & $7.24$ & ${ 5.41}$ &$6.14$& $\bf{3.79}$ &${0.09}$ & $0.119$ & $\bf{0.230}$ & $0.211$ & $        0.217$\\
  
 &$5, 3$ & & $56$ & $60$ & ${30}$ & $\bf{8}$& $\bf{8}$ & ${5.53}$ & $23.14$  & $176.9$ & $301.94$& $95.18$& ${5.45}$ & $17.33$ & $8$  &$15.55$ & $\bf{4.83}$ & $0.171$ & ${0.057}$ & $\bf{0.298}$ &${0.274}$& $0.293$\\
  
 &$7, 5$ & &$97$ & $100 $ & ${28}$ & $\bf{12}$& $15$& $-$ & $-$ & $-$ &$-$& $-$& $-$ & $-$ & $-$ &$-$&$-$& $-$ & $-$ & $-$& $-$& $-$\\
 \hline
 
Problem \ref{test_inst_Hil} & $2,2$ & $[0,5]^{n}$ & $\bf{0}$ &$35$& $40$& $33$& $33$& ${7.87}$ & $29.04$ & $\bf{5.217}$ & $9.35$ & $15.22$ & $\bf{19.84}$ & $88.125$ & $   60.85$ & $92.74$ & $159.20$ & $  0.07$& $\bf{1.05}$ & ${0.07}$& $ 0.07$ & $0.06$\\
 \hline

Problem \ref{test_inst_DGO1}  & $1,2$ & $[-10,13]$ &$\bf{20}$ & $\bf{20}$ & $93$ & $92$ & $91$& $3.02$ & $\bf{2.93}$ & $33.76$ & $33.10$ & $27.18$ & $2.17$ & $\bf{1.17}$ & $2.83$ & $3.83$ & $2.66$& $0.574$ & $\bf{1}$ & $0.323$ & $   0.323$ & $0.323$\\
 Problem
\ref{test_inst_DGO2} & $ 1,2$ & $[-9,9]$ &$\bf{1}$ &$\bf{1}$ &$2$ & $2$& $2$& $50.29$ & $\bf{1.59}$ & $6.45$ & $7.43$ & $7.13$ & $61.28$ & $\bf{6.13}$ & $44.90$ & $44.19$& $46.96$ & $0.09$ & $\bf{6.016}$ & $0.741$ & $  0.745$ & $0.773$\\
 \hline
 Problem
  \ref{test_inst_JOS1a} & $5, 2$ & $[-2,2]^{n}$ & $21$ & $\bf{7}$ & $58$ & $41$ & $41$ &$344.2$ & $    1097.57$ & $98.11$ & $\bf{8.57}$ & $ 8.81$ & $14.07$ & $2.29$ & $18.57$ & $\bf{1.21}$& $\bf{1.21}$ & $0.19$ & $  \bf{1.96}$ & $ 0.54$ & $0.66$ & $0.66$\\
 \hline
  Problem
 \ref{test_inst_FDSa} & $2,3$ & $[-2,2]^{n}$ &$\bf{21}$ & $46$& $97$ &$97$& $97$& $10.84$ & $\bf{8.06}$ & $32.71$ & $32.51$
& $48.13$ &  $4$ &  $3$
 & $\bf{2.5}$ &  $3.5$ &  $5$ & $0.069$ & $\bf{0.103}$ & $0.06$ & $0.06$ & $0.05$ \\
 \hline

Problem \ref{test_inst_Rosenbrock} & $4,3$& $[-2,2]^{n}$ & $100$ & $100$ & ${79}$ &$\bf{23}$ &$26$ & $-$ & $-$ & $-$ & $-$ & $-$ &$-$ & $-$ & $-$ & $-$ & $-$ & $-$ & $-$ & $-$& $-$ & $-$\\
 \hline

Problem \ref{test_inst_Brown_Dennis} 
    & $4,5$ & $[-25,25]$ & $39$ & $16$ & ${9}$ & ${1}$  & $\bf{0}$ & $141.90$ & ${130.31}$ & $267.93$ & $\bf{128.87}$ & $139.25$ &$26.83$ & ${9.17}$ & $10.67$ & $\bf{5.75}$ & $5.83$& $0.40$ & $2.05$ & $1.93$ &  $\bf{2.77}$ & $2.70$\\

  &  & $\times [-5,5]^{2}$ &  &  &  & & &  &  &  & & & &  &  &  &  &  \\

   &  & $\times [-1,1]$ &  &  &  & &&  & &  &  &  & & &  & &  &  & \\
 \hline

Problem \ref{test_inst_Trigonometric} & $4, 4$  & $[-1,1]^{n}$ &$13$ & $14$ & ${5}$ &$\bf{3}$& $\bf{3}$&$\bf{39.31}$ &  $8249.51$ & $84.79$ &  $87.74$& $
          100.74$ & $8.07$ & $4.85$ & $\bf{ 3.15}$ & $ 4.31$ & $ 4.69$& $ 0.218$ & $\bf{  0.679}$ & ${0.164}$ & $   0.164$ & $  0.164$  \\
 \hline

Problem \ref{test_inst_Das_Dennis}& $5 ,2$ & $[-20,20]^{n}$& $100$ &  $100$ & $\bf{75}$ & $76$ & $78$ & $-$ & $-$ & $-$ &$-$ &$-$&$-$ & $-$ & $-$ & $-$ & $-$ & $-$ & $-$ & $-$ & $-$ & $-$ \\
 \hline
Problem \ref{Example_5.1_stee} & $1,2$ &$[2,10]^{n}$ &$74$ & $74$ & $\bf{73}$&$73$& $73$ & $\bf{0.95}$ & $ 2.05$ &  $1.31$&  $13.15 $ & ${1.54 }$ & $2.81$ & $3.80$ & $3.80$&  $\bf{1.27}$ & ${1.34}$ & $0.110 $& $\bf{192608174.14}$ &  $0.212$  & $0.212$ & $0.212$\\
\hline
Problem \ref{Example_5.3_stee} & $2,2$ &$[-20,20]^{n}$ & $55$ & $56$ & ${52}$ &$\bf{48}$& $50$ &${33.35}$ & ${48.04}$ &  $36.72$ &$\bf{32.05}$& $
          34.79$ & $5.75$ & $ 6.25$ & $\bf{3} $ & $\bf{3}$ & $3$ &$0.023$ & $   0.126$ & $  0.550$ & $\bf{0.604}$ & $   \bf{0.604}$\\
\hline
Problem \ref{Sphere} & $3,3$ & $[0,1]^{3} $ &$23$ & ${14}$ & $66$& $40$  & $ $\bf{9}$ $ & $ 111.83$ & $614.00$ & $79.22$& $218.87$ & $\bf{60.29}$ & $ 4.7
$ & $8.75$ & $8.6$ & $14.3$ & $\bf{2.05}$& $0.126$ &  $0.046$ & $0.185$ & $0.167$ & $\bf{0.612}$ \\
 \thickhline
 \end{tabular} 
 \end{minipage}}
\end{table}
\end{center}
\end{landscape}

In Figure \ref{Step_sixe_avg}, we show the performance profile in terms of the average step-size taken by the five algorithms. As per Dolan and Mor{\'e} \cite{dolan2002benchmarking}, since the metric $t_{p,s}$ indicates better performance with lower values, we define the performance metric in terms of step-size as $t_{p,s} := \frac{1}{\text{step-size}}$. The step size is defined as the usual Euclidean $l_{2}$-norm of the difference between the consecutive iterates, calculated as $\lVert x_{k+1}-x_{k} \rVert$. To determine the average step size, we first compute an average across all iterations for each initial point, and then average the averages across all initial points.
The problem-wise values used to plot the performance profile in Figure \ref{Step_sixe_avg}, are given in the seventh column of Table \ref{uipuy_ytf}. From Figure \ref{Step_sixe_avg}, we see that Avg-NTRM, Max-NTRM and TRM take almost similar step except for $\tau$ in the range $[1, 6.27]$ (approx.), where Avg-NTRM and TRM are slightly better than Max-NTRM. Additionally, beyond the $\tau$ value of $14.16$, Max-NTRM and TRM takes similar steps, while Avg-NTRM takes smaller steps compared to them. Overall, among all algorithms, CG takes the largest step.

 Finally, if we analyze the performance of only the three TRM algorithms, then from Table \ref{uipuy_ytf}, we observe that in terms of the possibility of convergence, non-monotone schemes are far better than TRM for all problems except DTLZ3, Das and Dennis, and Modified Example 5.1 \cite{steepmethset}, where TRM is slightly better. On the other hand, if we look at the speed of convergence, then, based on iterations (Convergent), NTRM is better than TRM except for Hill and FDSa, where TRM is better. In terms of CPU time (Convergent), NTRM and TRM are both almost similar, where NTRM and TRM are better for $59$ percent and $41$ percent of test problems, respectively. Finally, in terms of the step size, from Table \ref{Step_sixe_avg}, NTRM is either similar to or slightly better than TRM, except for DTLZ with $(n,m) = (3,3)$ and $(5,3)$. Overall, we find that NTRM is better than TRM at finding a solution by taking fewer iterations and less CPU time.

Next, similar to the TRM paper \cite{trust2025setopt}, we observe that, in terms of convergence probability, any of the three TRM algorithms perform better for problems with a higher number of variables and function components, whereas SD and CG perform better for problems with fewer variables and function components. Specifically, for some problems like the ZDT1 type problem, SD and CG do not converge at all when the number of variables $n$ increases from $2$ to $10$. A similar observation can be made for DTLZ3 and DTLZ5, where TRM performs better with a higher probability of convergence.

Altogether, with the higher possibility of finding solution, NTRM is not only better than TRM but also better than SD and CG as well, shown in Figure \ref{per_prof_no_of_failure} and in the third column of Table \ref{uipuy_ytf}.   In terms of speed of convergence, using the iteration (convergent) metric, Avg-NTRM is consistently better for up to $\tau=3.63$; beyond that, CG performs better.  In terms of CPU time (convergent), SD performs better for certain values of $\tau$ because of the absence of Hessian computation.  
Therefore, the proposed NTRMs should be considered very strong alternatives for many types of problems.


Finally, we analyze the performance of the algorithms from the perspective of the error graph in Figure \ref{error_graph_analysis}. For this, we consider the problems Brown and Dennis, Modified Ex 5.3, DTLZ5 ($n=3, m=3$), and DTLZ5 ($n=5, m=3$) with their respective initial points $(-4.4685,-1.6190,-3.7467,0.5130)$, $(-10.3717, 14.4456)$, $(0.243525, 0.015403, 0.779052)$, and $(0.438645, 0.166460, 0.603468, 0.722923, 0.163512)$. As error, denoted by $\mathrm{error}_{k}$, we plot $|t_{k}|$ for TRM, Max-NTRM, and Avg-NTRM and $\lVert v_{k} \rVert$ for SD and CG versus the iteration number $k$.  
 
In Figure \ref{error_graph_Brown_dennis}, the TRM and NTRM variants show near-superlinear convergence, bending sharply downward near the optimum and reaching low error levels in few iterations. CG converges more slowly but still faster than SD, while SD exhibits nearly linear convergence and slows down significantly near the optimum.
In Figure \ref{error_graph_Modified_Ex5.3},  TRM variants converge the fastest, reaching low error levels in few iterations, with Average-NTRM performing best. CG shows near-linear convergence with a slower decline, while SD is the slowest, remaining at the highest error level. In Figures \ref{error_graph_DTLZ5_n3_m3} and \ref{error_graph_DTLZ5_n5_m3}, TRM and NTRM variants converge faster than linearly, showing sharp error reduction after a few steps. By contrast, SD and CG decrease error steadily but slowly, appearing nearly linear on the log scale. 



\begin{figure}[htbp]
\centering

\subfigure[Brown and Dennis]{
    \includegraphics[width=0.4\textwidth]{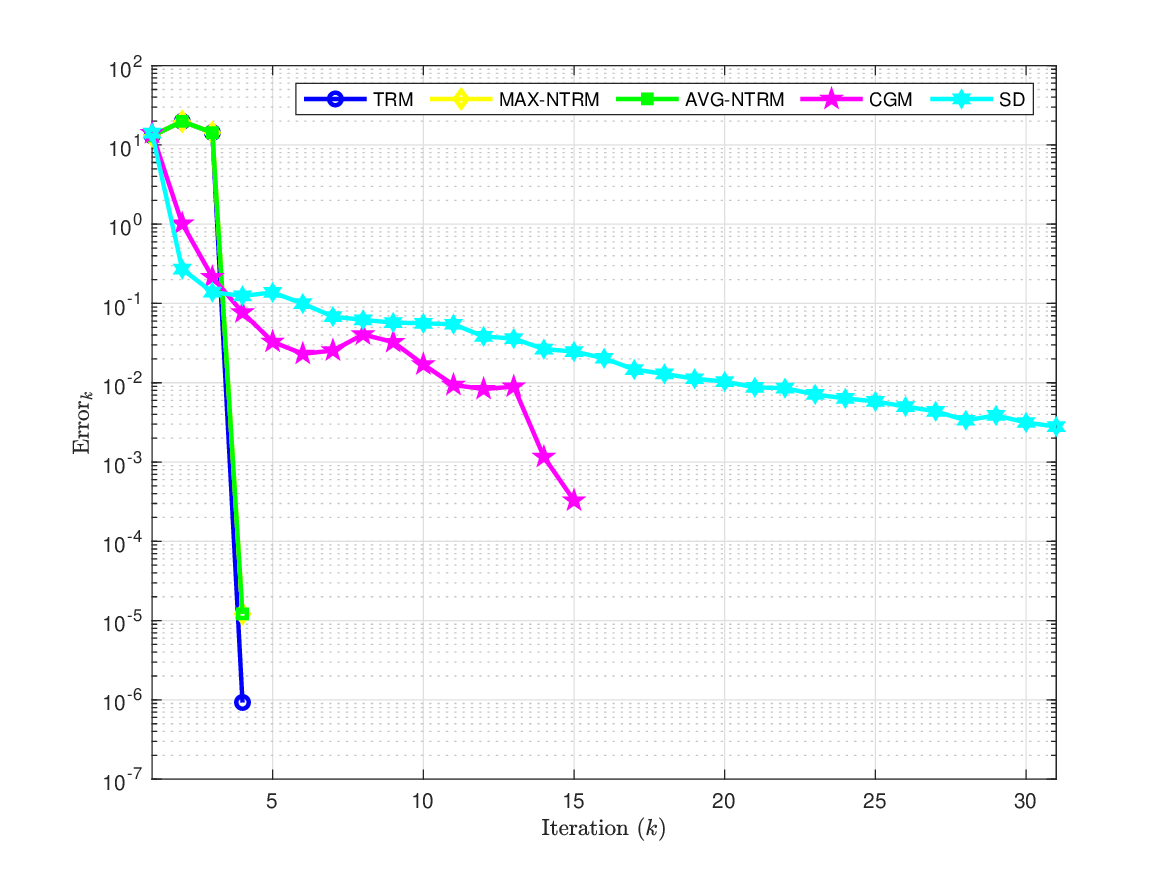}
    \label{error_graph_Brown_dennis}
}
\subfigure[Modified Ex5.3]{
    \includegraphics[width=0.4\textwidth]{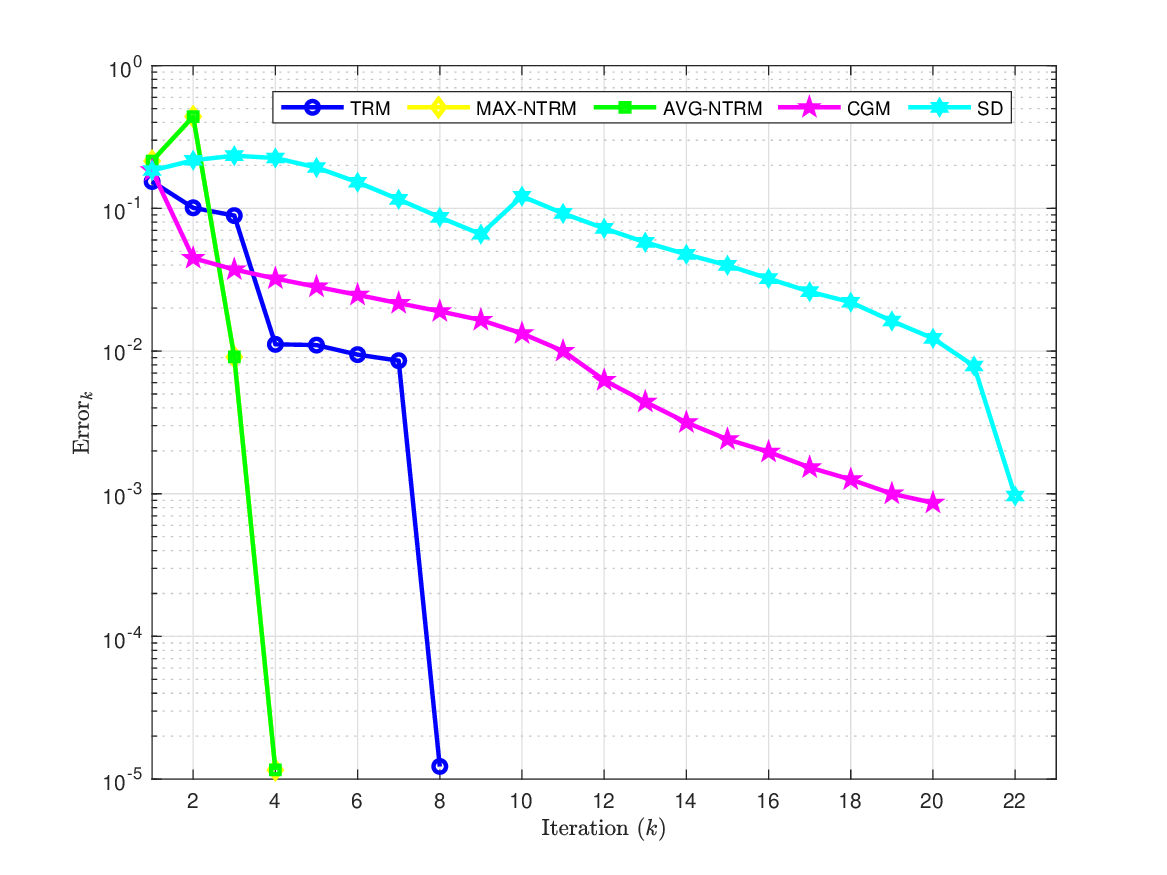}
    \label{error_graph_Modified_Ex5.3}
}
\subfigure[DTLZ5 ($n=3, m=3$)]{
    \includegraphics[width=0.4\textwidth]{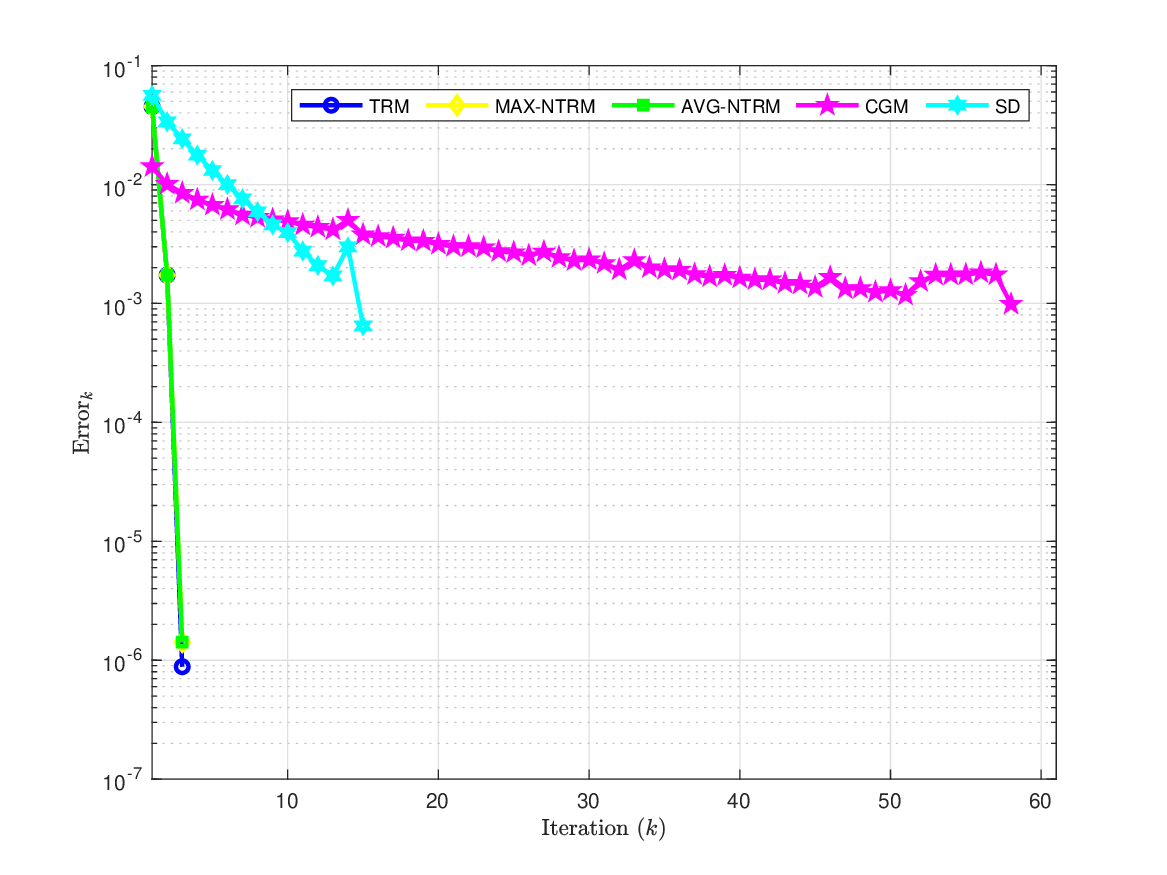}
\label{error_graph_DTLZ5_n3_m3}
}
\subfigure[DTLZ5 ($n=5, m=3$)]{
    \includegraphics[width=0.4\textwidth]{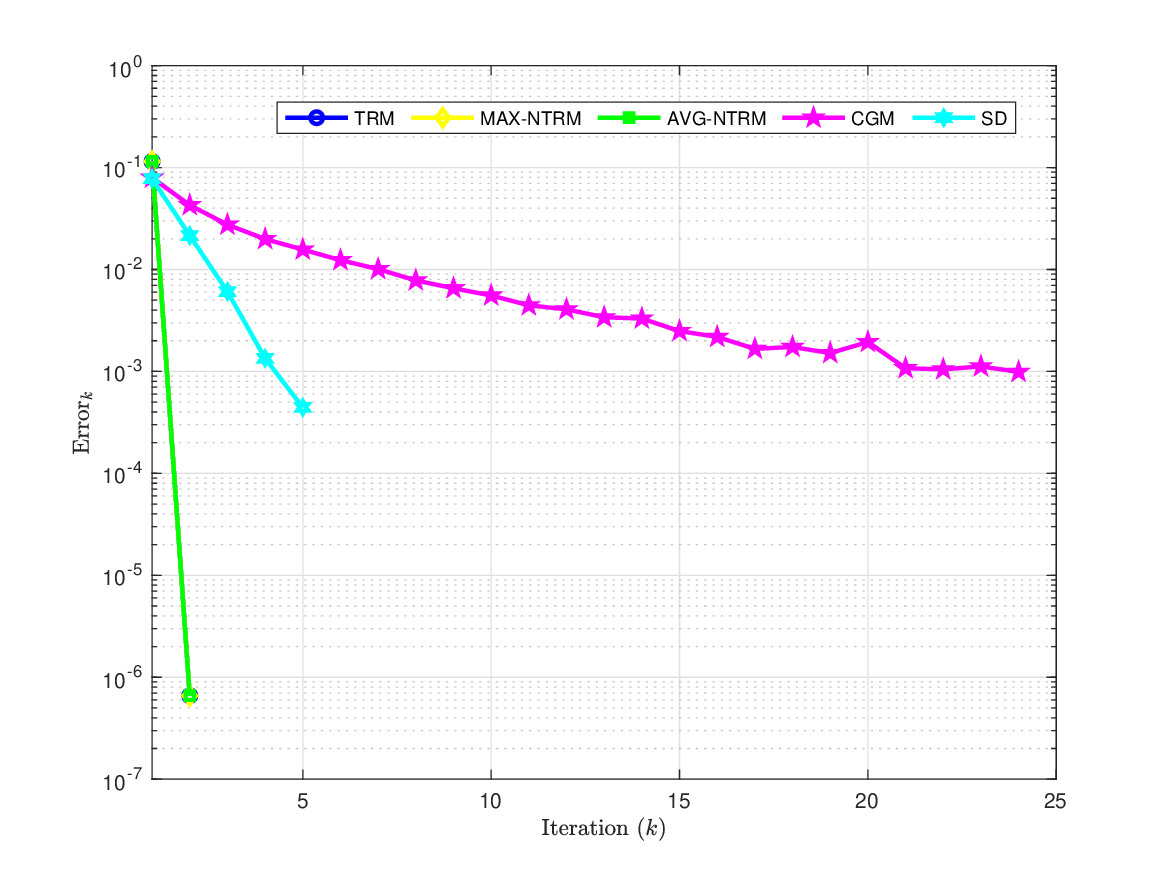}
    \label{error_graph_DTLZ5_n5_m3}
}

\caption{Error graphs of TRM, Max-NTRM, Avg-NTRM, CG, and SD for Brown and Dennis, Modified Ex5.3, DTLZ5 ($n=3, m=3$), and DTLZ5 ($n=5, m=3$).}
\label{error_graph_analysis}
\end{figure}

\section{Conclusion and Future Directions} \label{sect6}
\noindent 
In this article, we have studied two non-monotone trust-region methods, called Max-NTRM and Avg-NTRM, to find a $K$-critical point for \eqref{fgcx}. For this, we have modified TRM for SOP proposed by Ghosh et al. \cite{trust2025setopt} by utilising the non-monotone strategies of Ramiraz et al. \cite{Ramirez} and Ghalavand et al. \cite{ghalavand2023two} from multi-objective optimization literature. The main modification is in the step acceptance requirement by relaxing the strict monotonic decrement requirement of the monotone TRM of Ghosh et al. \cite{trust2025setopt}. Instead of considering the function values of only the current iteration, we have taken into account the maximum over successive function values from the last few iterations for Max-NTRM (Definition \ref{ghvt_ugh_huh}) and an weighted moving average of successive function values up to current iteration for Avg-NTRM (Definition \ref{tbddv_uiet}), which are used to define two new reduction ratios (Definition \ref{Ratios_NTR_Max} and Definition \ref{NTR_Average_Ratios}) for Algorithm \ref{alg_max} and Algorithm \ref{alg_avg}, respectively. The well-definedness of these algorithms has been discussed in Subsection \ref{welldefine_non_monotone_trust_algori}.
Under some assumptions, the convergence analysis (Section \ref{global_conv_ana}) of the two proposed methods has also been discussed, where we have derived the following. 
\begin{enumerate}[(i)]
\item At an $x_{k}$ iterate, the solution $s_{k}$ of the subproblem (\ref{ghtyu}) satisfies certain relations that lower bound the sufficient decrease of the model $m^{a^{k}}_{k}$ (Corollary \ref{utrw_kkjx}).

\item The sequence of maximum over previous successive function values $\{(f^{i,r}(x_{l^{i,r}(k)}))_{r \in [m]}\}_{i \in [p]}$ and weighted moving average of previous successive function values with current function values $\{(C^{i,r}_{k})_{r \in [m]})\}_{i\in [p]}$ decreases monotonically at every iteration and eventually admits a limit as $k \to \infty$ (Lemma \ref{admits_limit_max_avg}).

\item A step is eventually accepted after a finite number of unsuccessful iterations along with a reduction in trust-region radius (Theorem \ref{rtyue_ouit}). 

\item The global convergence of regular iterative sequence $\{x_{k}\}$ generated by Algorithm \ref{alg_max} and Algorithm \ref{alg_avg} converges to a $K$-critical point for \eqref{fgcx} under bounded level set assumption (Theorem \ref{gjhg_uyhugy}).
\end{enumerate}
Finally, in the numerical experiments performed for some well-known and newly defined $22$ test problems, we have found that in terms of probability of convergence, NTRM performs the best; in terms of number of iterations, NTRM performs the best for most of the problems; and in terms of CPU time, SD performs the best.

\bigskip
In future work, several directions can be pursued:
\begin{itemize}
\item Use of the Cauchy point as the step length, as the Cauchy step does not require matrix factorization. This can further simplify the method computationally.

\item For max-type NTRM in \cite{ghalavand2023two}, we have not adjusted $r^{i}_{k}$ from Equation (10) in \cite{ghalavand2023two} for the SOP context, as most results involving $r^{i}_{k}$ appear consistent with our findings in this paper. Nonetheless, we plan to verify this rigorously in future studies. In contrast, for the average-type NTRM, we have improved $c^{i}_{k}$ using formula (\ref{tbddv_uiet}).

\item The convergence analysis in this paper relies on the regularity condition, which makes $\preceq^{l}_{K}$-minimal and weakly $\preceq^{l}_{K}$-minimal elements coincide and keeps the cardinality function locally constant. While useful for theory, this assumption is restrictive, and future research could relax it by allowing variable cardinality or by studying non-regular points where partition sets may change.

\item The objective function in problem \eqref{fgcx} is treated as a set-valued map with finite cardinality. Extending the proposed method to handle set-valued maps with continuum cardinality is another promising direction.
\end{itemize}
\begin{appendices}

\section{Test problems for set-valued optimization}\label{appendix_ind}
Below, we provide a list of test problems for set-valued optimization. These problems are formulated from the test problems of multi-objective optimization problems (see \cite{steepmethset,NewtonMethod, huband2006review, mita2019nonmonotone}). For each of the following test problems, the map $F:\mathbb R^{n} \to \mathbb R^{m}$ is given by $F(x): = \{f^{1}(x), f^{2}(x), \ldots, f^{100}(x)\}$. 
For the DTLZ-type problems, $x_{m}$ is the set of the last $x_{i}$ variables of the vector $x$ in $\mathbb R^{n}$, where $i$ varies from $1$ to $\ell=n-m+1$. The cardinality of the set  $x_{m}$, denoted as $|x_{m}|$, is $\ell$.

\begin{prob}\label{test_inst_ZDT1_22}
ZDT1-type SOP. For each $i \in [100]$, $f^{i}: \mathbb R^{n} \to \mathbb R^{2}$ is given by
\begin{align*}
f^i(x):=\begin{pmatrix}
   f_{1}(x) + \left(0.02+ 0.02 \cos^{16}(\frac{4 \pi i}{100})\right) \cos(\frac{2 \pi i}{100})    \\
    g(x) h(f_{1}(x),g(x)) +  0.15+ 0.15\cos^{16}\left(\frac{4 \pi i}{100} \right) \sin(\frac{2 \pi i}{100})  
\end{pmatrix}
\end{align*}
with $ 
f_{1}(x) := x_{1}, g(x) := 1 + 9 \sum_{i=2}^{n} x_{i} \text{ and } h(f_{1}, g) := 1 - \sqrt{\frac{f_{1}}{g}}.$
\end{prob}

\begin{prob}\label{test_inst_ZDT4}
ZDT4-type SOP. Here $m = 2$ and for each $i \in [100]$, $f^{i}: \mathbb R^{n} \to \mathbb R^{2}$ is given by
\begin{align*}
 f^{i}(x):= 
\begin{pmatrix}
 f_{1}(x)+(1 +  \cos^{16}(\frac{4\pi i}{100}) \cos(\frac{2\pi i}{100})) \\
 g(x)h(f_{1}(x),g(x)) + (1 + \cos^{16}(\frac{4\pi i}{100})\sin(\frac{2\pi i}{100}))
\end{pmatrix}
\end{align*}
with 
$f_{1}(x) := x_{1}, g(x) := 1+ 10(n-1)+ \sum_{i=2}^{n} (x^{2}_{i}-10\cos(4\pi x_{i})) \text{ and } h(f_{1},g) := 1 - \sqrt{ \frac{f_{1}}{g}}.$  
 \end{prob} 

\begin{prob}\label{test_inst_DTLZ1_n100}
DTLZ1-type SOP. For each $i \in [100]$, $f^{i}: \mathbb R^{n} \to \mathbb R^{m}$ is given by
\begin{align*}
f^{i}(x) := ~&\begin{pmatrix}
    (1+g(x_{m}))x_{1} x_{2}\cdots x_{m-1}\\
    (1+g(x_{m})) x_{1} x_{2}\cdots (1-x_{m-1})\\
    \vdots \\
   \frac{1}{2} (1+g(x_{m}))\frac{1}{2} x_{1} (1-x_{2})\\
    \frac{1}{2}(1-x_{1}) (1+g(x_{m})).
\end{pmatrix}+ \begin{pmatrix}
 \cos(\phi_{i})\sin(\psi_{i})\\
 \sin(\phi_{i})\sin(\psi_{i})\\
 (\cos(\psi_{i})+\ln \left( \tan\frac{\psi_{i}}{2}\right)+0.2 \phi_{i}\\
 0 \\
\vphantom{\int^0}\smash[t]{\vdots} \\
0 
\end{pmatrix} 
\end{align*}
 \mbox{with} 
 \begin{align*}
g(x_{m}) := 100 \left(\lvert x_{m} \rvert + \sum_{x_{i}\in x_{m}} (x_{i}-0.5)^{2}-\cos(20 \pi (x_{i}-0.5))\right).
\end{align*}
 Here,  the set $\{ (\phi_{i}, \psi_{i}) : i\in [100] \}$ is an enumeration of the set $\{\tfrac{\pi}{5}(j - 1): j \in [10]\} \times \{\tfrac{\pi}{5}(\ell - 1): \ell \in [10]\}$. 
 \end{prob}

\begin{prob}\label{test_inst_DTLZ3_54}
 DTLZ3-type SOP. For each $i\in [100]$, $f^{i}: \mathbb R^{n} \to \mathbb R^{m}$ is given by
\begin{align*}
f^{i}(x) := ~& \begin{pmatrix}
(1+g(x_{m}))\cos(\frac{x_{1}\pi}{2})\cdots \cos(\frac{x_{m-2}\pi}{2}) \cos(\frac{x_{m-1}\pi}{2})\\
(1+g(x_{m}))\cos(\frac{x_{1}\pi}{2})\cdots \cos(\frac{x_{m-2}\pi}{2}) \sin(\frac{x_{m-1}\pi}{2})\\
 (1+g(x_{m})) \cos(\frac{x_{1}\pi}{2}) \cos(\frac{x_{1}\pi}{2}) \\
 \vdots \\
 (1+ g(x_{m})) \sin(\frac{x_{1}\pi}{2}).
\end{pmatrix} + \begin{pmatrix}
 \sech(\phi_{i}) \cos(\phi_{2})\\
 \sech(\phi_{i}) \sin(\psi_{i})\\
 \phi_{i}-\tanh(\phi_{i})\\
 0 \\ \vphantom{\int^0}\smash[t]{\vdots} \\
0
\end{pmatrix} 
\end{align*}
\mbox{with}~ 
\begin{align*}
&g(x_{m}) := 100 \left(\lvert x_{m} \rvert + \sum_{x_{i} \in x_{m}} (x_{i}-0.5)^{2}-\cos(20\pi(x_{i}-0.5))\right).
\end{align*}
Here, the set $\{(\phi_{i}, \psi_{i}): i \in [100] \}$ is an  enumeration of the set $\{\frac{\pi}{5}(j-1): j \in [10]\} \times \{\frac{\pi}{5}(l-1): l \in [10]\}$.
 \end{prob}

\begin{prob}\label{test_inst_FDSa}
 FDSa-type SOP. For each $i \in [100]$, $f^{i}: \mathbb R^{n} \to \mathbb R^{m}$  is given by
\begin{align*}
&f^i(x):=\begin{pmatrix}
  G_{1}(x) + (1+ \cos(\phi_{i})\cos(\psi_{i}))\\
   G_{2}(x) + (1 + \cos(\phi_{i})\sin(\psi_{i}))    \\ 
    G_{3}(x) + sin(\phi_{i}) \end{pmatrix}
\mbox{with} ~G_{1}(x) := \frac{1}{n^2} \sum_{i = 1}^{n} i(x_{i}-i)^4, \\
 &  G_{2}(x) := \exp\left(\sum_{i=1}^{n}\frac{x_{i}}{n}\right) + {\lVert x \rVert}^{2}_{2}, ~\text{and}
  ~~ G_{3}(x) := \frac{1}{n(n+1)}  \sum_{i=1}^{n} i(n-i+1) \exp(-x_{i}).
\end{align*}
\end{prob}

 \begin{prob}\label{test_inst_DTLZ5}
 DTLZ5-type SOP. For each $i \in [100]$, $f^{i}: \mathbb R^{n} \to \mathbb R^{m}$ is given by
 \begin{align*}
f^{i}(x) := \begin{pmatrix}
(1+g(x_{m}))\cos(\tfrac{\theta_{1}\pi}{2})\cdots\cos(\tfrac{\theta_{m-2}\pi}{2})\cos(\tfrac{\theta_{m-1}\pi}{2})\\
(1+g(x_m))\cos(\tfrac{\theta_{1}\pi}{2})\cdots\cos(\tfrac{\theta_{m-2}\pi}{2})\sin(\tfrac{\theta_{m-1}\pi}{2})\\
(1+g(x_{m})) \cos(\tfrac{\theta_{1}\pi}{2})\cdots \sin(\tfrac{\theta_{m-2}\pi}{2})\\
\vdots \\
(1+ g(x_{m}))\sin(\tfrac{\theta_{1}\pi}{2})
\end{pmatrix}+ \begin{pmatrix}
5\tfrac{\psi_{i}}{2\pi}\\
\lambda \tfrac{\cos(\phi_{i})}{10} \\
\lambda \tfrac{\sin(\phi_{i})}{10} \\
0 \\
\vphantom{\int^0}\smash[t]{\vdots}\\
0
\end{pmatrix},
\end{align*}
where
\begin{align*}
\theta_{i}(x):= &\frac{1}{2(1+g(x_{m}))} (1+g(x_{m})x_{i}) ~\text{for}~i = 2, 3,\ldots,(m-1),~ \\\text{and}~ g(x_{m}) := &\sum_{x_{i} \in x_{m}} (x_{i}-0.5)^2. 
\end{align*}
Here, the set $\{ (\phi_{i}, \psi_{i}) : i\in [100] \} $ is an enumeration of the set $\{\tfrac{\pi}{5}(j - 1): j \in [10]\} \times \{\tfrac{\pi}{5}(l - 1): l \in [10]\}$. 
\end{prob}

\begin{prob}\label{test_inst_DGO1} 
 DGO1-type SOP. For each $i \in [100], f^{i}: \mathbb R \to \mathbb R^{2}$ is given by 
\begin{align*}
   f^{i}(x) := ~&g(x) + \begin{pmatrix} 
    \sin(\tfrac{\pi i}{50} + \cos(\tfrac{\pi i}{50}))\\
    \cos(\tfrac{\pi i
    }{50} + \sin(\tfrac {\pi i}{50}))
   \end{pmatrix}\mbox{ with }~ 
g(x) := \begin{pmatrix}
           \sin(x) \\
           \sin(x+0.7)
     \end{pmatrix}.
\end{align*}
 \end{prob}

\begin{prob}\label{test_inst_DGO2}
 DGO2-type SOP. For each $i \in [100], f^{i}: \mathbb R \to \mathbb R^{2}$ is given by
\begin{align*}
   f^{i}(x) :=  g(x) + \begin{pmatrix}
       \sin(\tfrac{\pi i}{50} + \cos(\tfrac{\pi i}{50}))\\
        \cos(\tfrac{\pi i}{50} + \sin(\tfrac{2\pi i}{50}))
   \end{pmatrix} 
\mbox{with}~ 
g(x) := \begin{pmatrix}
      x^{2}\\
      9 -\sqrt{81-x^2}
    \end{pmatrix}.
\end{align*}
 \end{prob}

\begin{prob}\label{test_inst_Hil}
 Hil-type SOP. For each $i \in [100]$, $f^{i}: \mathbb R^{n} \to \mathbb R^{2}$ is given by
\begin{align*}
    f^{i}(x) := g(x) + \begin{pmatrix}
     10 ((9+ \exp(\sin(\tfrac{\pi i}{25}))- \sin(\tfrac{\pi i}{25})+ 2(\cos(\tfrac{2\pi i}{25}))^2)/128) \cos( \tfrac{\pi i}{50})\\ 10 ((9+ \exp(\sin(\tfrac{ \pi i}{25}))-\sin(\tfrac{\pi i}{25})+ 2(\cos(\tfrac{2\pi i}{25}))^2)/128) \sin(\tfrac{\pi i}{50})
    \end{pmatrix}
\end{align*}
with 
\begin{align*}
 g(x) :=
   \begin{pmatrix}
      \cos((\tfrac{\pi}{180})(45+40\sin(2\pi x_{1})+ 25 \sin(2\pi x_{2}))(1+0.5 \cos(2\pi x_{1}) ) \\ 
      \sin((\tfrac{\pi}{180})(45+40\sin(2\pi x_{1})+ 25 \sin(2\pi x_{2}))(1+0.5 \cos(2\pi x_{1}))
    \end{pmatrix}.
\end{align*}
\end{prob}

\begin{prob}\label{test_inst_JOS1a}
 JOS1a-type SOP. For each $i \in [100], f^{i}: \mathbb R^{n} \to \mathbb R^{2}$ is given by 
 \begin{align*}
    f^{i}(x) := g(x) + \begin{pmatrix}
         0.1 \cos(\tfrac{\pi i}{50})\\
         50 \sin(\tfrac{\pi i}{50})
    \end{pmatrix} 
\mbox{with} ~
 g(x) := \begin{pmatrix}
         \frac{1}{n} \sum_{i=1}^{n} x^{2}_{i} \\
         \frac{1}{n} \sum_{i=1}^{n} (x_{i}-2)^2
    \end{pmatrix}.
\end{align*}
\end{prob}

\begin{prob}\label{test_inst_Rosenbrock}
 Rosenbrock-type SOP. For each $i \in [100], f^{i}: \mathbb R^{n} \to \mathbb R^{3}$ is given by 
\begin{align*}
    f^{i}(x) :=  \begin{pmatrix}
         100 (x_{2}-x^{2}_{1})^{2}+ (x_{2}-1)^{2}+ (r^{2}(\cos({\phi_{i}}) \cos(\psi_{i}) \sin(\psi_{i}))) \\
         100 (x_{3}-x^{2}_{2})^{2} + (x_{3}-1)^{2} + (r^{2}(\cos({\phi_{i}}) \sin(\psi_{i}) \sin(\psi_{i})))\\
         100 (x_{4}-x^{2}_{3})^{2} + (x_{4}-1)^{2}+ (r^{2}(\cos({\phi_{i}}) \sin(\psi_{i}) \cos^{2}(\psi_{i})))
    \end{pmatrix}.
\end{align*}
 We set $r = 16$ and the set $\{(\phi_{i}, \psi_{i}): i \in [100]\}$ is an enumeration of the set $\{ \frac{\pi}{5}(j-1):j\in [10]\}\times \{ \frac{\pi}{5}(l-1):l\in [10]\}$.
 \end{prob}

\begin{prob}\label{test_inst_Brown_Dennis}
Brown and Dennis-type SOP. For each $i \in [100]$, $f^{i}: \mathbb R^{4} \to \mathbb R^{3}$ is given by 
\begin{align*}
& f^{i}(x) := 
\begin{pmatrix}
(x_{1}+ \frac{1}{5}x_{2}-\exp({\frac{1}{5}}))^{2}+ (x_{3}+x_{4}\sin(\frac{1}{5}) -\cos(\frac{1}{5}))^{2} + \cos(\phi_{i}) \sin(\psi_{i})\\
(x_{1} + \frac{2}{5}x_{2}-\exp({\frac{2}{5}}))^{2} + (x_{3}+ x_{4}\sin(\frac{2}{5}) - \cos(\frac{2}{5}))^{2} + \sin(\phi_{i}) \sin(\psi_{i})\\
(x_{1} + \frac{3}{5}x_{3} - \exp(\frac{3}{5}))^{2} +  (x_{3} + x_{4} \sin(\frac{3}{5}) - \cos(\frac{3}{5}))^{2} +  (\cos(\psi_{i}) +   \log(\tan(\frac{\psi_{i}}{2}))) + 0.5\phi_{i} \\ 
\end{pmatrix} 
\end{align*}
The set $\{(\phi_{i}, \psi_{i}): i \in [100]\}$ is an enumeration of the set $\{ \frac{2\pi}{5}(j-1):j\in [10]\}\times \{ 0.01+ 0.098(l-1):l\in [10]\}$.
 \end{prob}

\begin{prob}\label{test_inst_Trigonometric}
Trigonometric type SOP. For each $i \in [100], f^{i}: \mathbb R^{n} \to \mathbb R^{4}$ is given by
\begin{align*}
& f^{i}(x) := \begin{pmatrix}
 (1 - \cos x_{1} + (1-\cos x_{1}) -\sin x_{1})^{2} + \cos(\phi_{i}) \sin(\psi_{i})  \\
 (2 - \cos (x_{1}+x_{2}) + 2 (1- \cos x_{2})-\sin x_{2})^{2} + \sin(\phi_{i})\sin(\psi_{i}) \\
 (3 - \cos(x_{1} + x_{2} + x_{3}) + 3(1 - \cos x_{3}) - \sin x_{3})^2 + (\cos(\psi) + \log(\tan(\frac{\psi_{i}}{2}))) + 0.2 \phi_{i} \\
 (4 -\cos(x_{1}+x_{2}+x_{3}+x_{4})+4(1-\cos x_{4})-\sin x_{4}) \\ 
 \end{pmatrix}. 
\end{align*}
 The set $\{(\phi_{i}, \psi_{i}): i \in [100]\}$ is an enumeration of the set $\{ \frac{2\pi}{5}(j-1):j\in [10]\}\times \{ 0.01+ 0.098(l-1):l\in [10]\}$.
\end{prob}

\begin{prob}\label{test_inst_Das_Dennis}
 Das and Dennis type SOP. For each $i \in [100], f^{i}: \mathbb R^{n} \to \mathbb R^{2}$ is given by
\begin{align*}
& f^{i}(x) := \begin{pmatrix}
(x^{2}_{1}+x^{2}_{2}+x^{2}_{3}+x^{2}_{4} + x^{2}_{5} + (\sin(\frac{i\pi}{50})+ \cos(\frac{i\pi}{50})) \\
(3 x_{1} + 2 x_{2} - \frac{x_{3}}{3} + 0.01(x_{4} - x_{5})^{3} + (\sin(\frac{i\pi}{50})+\cos(\frac{i\pi}{50}))
\end{pmatrix}.
\end{align*}
 \end{prob}

\begin{prob}\label{Example_5.1_stee} Test Instance 5.1 in \cite{steepmethset}-type SOP. For each $i \in [5], f: \mathbb R \to \mathbb R^{2}$ is defined as 
\begin{align*}
    f^{i}(x) := 
    \begin{pmatrix}
      x \\
     \frac{x}{2} \sin(x)
    \end{pmatrix}
+  \cos^{2}(x) \bigg[\frac{(i-1)}{4}
\begin{pmatrix}
     1\\
     -1
\end{pmatrix}+ \bigg(1-\frac{i-1}{4} \bigg)\bigg].
    \end{align*}
\end{prob}

\begin{prob}\label{Example_5.3_stee}
 Test Instance 5.3 in \cite{steepmethset}-type SOP.  
 For each $i \in [100]$, $f^{i}: \mathbb R^{n} \to \mathbb R^{2}$ is given by 
\begin{align*}
f^{i}(x_1, x_2) : = 
\begin{pmatrix}
e^{\frac{x_1}{2}}\cos x_{2} + x_{1} \cos x_{2} \sin \tfrac{\pi(i-1)}{50} - x_{2} \sin x_{2} \cos^3 \tfrac{\pi(i-1)}{50} \\
e^{\frac{x_{2}}{20}} \sin x_{1} + x_{1} \sin x_{2} \sin^{3} \tfrac{\pi(i-1)}{50} + x_{2} \cos x_{2} \cos \tfrac{\pi(i-1)}{50}
\end{pmatrix}.
\end{align*}
\end{prob}

\begin{prob}\label{Sphere}
Sphere SOP. For each $i \in [100]$, $f^{i}: \mathbb R^{3} \to \mathbb R^{3}$ is given by 
\begin{align*}
   f^{i}(x_1, x_2, x_3) :=\begin{pmatrix}
    (1 + g(x_3)) \cos u(x_1) \cos v(x_1, x_2, x_3) \\ 
    (1 + g(x_3)) \cos u(x_1) \sin v(x_1, x_2, x_3) \\ 
    (1 + g(x_3)) \sin u(x_1) 
    \end{pmatrix} + \tfrac{1}{16}
   \begin{pmatrix}
   \cos \phi_i \\ 
   \cos \psi_i \sin \phi_i \\ 
   \sin \psi_i \sin \phi_i 
   \end{pmatrix}
\end{align*}
with $ g(x_3) := (x_3 - \tfrac{1}{2})^2,~ u(x_1) := \frac{\pi x_1}{2},~  
    v(x_1, x_2, x_3) := \tfrac{\pi (1 + 2 g(x_3) x_2)}{4 \left(1 + g\left(\sqrt{x_1^2 + x_2^2 + x_3^2}\right)\right)} $
 and the set $\{(\phi_{i}, \psi_{i}): i \in [100]\}$ is an enumeration of the set $\{\tfrac{\pi}{10}(j - 1): j \in [10]\} \times \{\tfrac{\pi}{5}(l - 1): l \in [10]\}$. 
\end{prob}
\end{appendices}

\subsubsection*{Funding}
Debdas Ghosh is thankful for the financial support from the Core Research Grant (CRG/2022/001347) from SERB, India. Zai-Yun Peng is supported by the National Natural Science Foundation of China (12271067), the Chongqing Natural Science Foundation (CSTB2024NSCQ-MSX0973), and the Science and Technology Research Key Program of Chongqing Municipal Education Commission (KJZD-K202200704).

\subsubsection*{Data availability}
There is no data associated with this paper.

\subsection*{Declarations}    

\subsubsection*{Funding and/or Conflicts of interests/Competing interests} 
The authors do not have any conflicts of interest to declare. They also do not have any funding conflicts to declare.

\subsubsection*{Ethics approval} 
This article does not involve any human and/or animal studies.


\end{document}